\documentclass[reqno]{amsart}
\usepackage[active]{srcltx}
\usepackage {color}
\usepackage{amsmath,amsthm,amssymb}
\usepackage{epsfig}
\usepackage{graphicx}
\usepackage{amssymb}
\usepackage{latexsym}
\usepackage{pdfpages}

\usepackage{enumitem}

\usepackage{ulem} 

\usepackage{ifpdf}
\newcommand{\res}{\!\!\mathop{\hbox{
                                \vrule height 7pt width .5pt depth 0pt
                                \vrule height .5pt width 6pt depth 0pt}}
                                \nolimits}

\def\z{{\bf z}}

\ifpdf 
  \usepackage[hidelinks]{hyperref}
\else 
\fi 

 \usepackage[usenames,dvipsnames]{pstricks}
 \usepackage{epsfig}
 \usepackage{pst-grad} 
 \usepackage{pst-plot} 

\usepackage{color}

\usepackage{subcaption}

\newtheorem{theorem}{Theorem}[section]
\newtheorem{lemma}[theorem]{Lemma}
\newtheorem{definition}[theorem]{Definition}
\newtheorem{proposition}[theorem]{Proposition}
\newtheorem{corollary}[theorem]{Corollary}
\newtheorem{remark}[theorem]{Remark}
\newtheorem{example}[theorem]{Example}
\newtheorem*{theorem*}{\it Theorem}

\def\vint_#1{\mathchoice%
          {\mathop{\kern 0.2em\vrule width 0.6em height 0.69678ex depth -0.58065ex
                  \kern -0.8em \intop}\nolimits_{\kern -0.4em#1}}%
          {\mathop{\kern 0.1em\vrule width 0.5em height 0.69678ex depth -0.60387ex
                  \kern -0.6em \intop}\nolimits_{#1}}%
          {\mathop{\kern 0.1em\vrule width 0.5em height 0.69678ex
              depth -0.60387ex
                  \kern -0.6em \intop}\nolimits_{#1}}%
          {\mathop{\kern 0.1em\vrule width 0.5em height 0.69678ex depth -0.60387ex
                  \kern -0.6em \intop}\nolimits_{#1}}}
\def\vintslides_#1{\mathchoice%
          {\mathop{\kern 0.1em\vrule width 0.5em height 0.697ex depth -0.581ex
                  \kern -0.6em \intop}\nolimits_{\kern -0.4em#1}}%
          {\mathop{\kern 0.1em\vrule width 0.3em height 0.697ex depth -0.604ex
                  \kern -0.4em \intop}\nolimits_{#1}}%
          {\mathop{\kern 0.1em\vrule width 0.3em height 0.697ex depth -0.604ex
                  \kern -0.4em \intop}\nolimits_{#1}}%
          {\mathop{\kern 0.1em\vrule width 0.3em height 0.697ex depth -0.604ex
                  \kern -0.4em \intop}\nolimits_{#1}}}

\def\R{\mathbb R}
\def\N{\mathbb N}
\def\Z{\mathbb Z}
\def\g{\hbox{\bf g}}

\numberwithin{equation}{section}

\def\NN{{\mathbb{N}}}

\def\1{\raisebox{2pt}{\rm{$\chi$}}}

\def\g{{\bf g}}

\definecolor{violet(ryb)}{rgb}{0.53, 0.0, 0.69}
\usepackage{mathtools}
\mathtoolsset{showonlyrefs}

\begin{document}

\title[Total Variation Flow in  Metric Random Walk Spaces]{\bf The Total Variation Flow in  Metric Random Walk Spaces}

\author[J. M. Maz\'on, M. Solera, J. Toledo]{Jos\'e M. Maz\'on, Marcos Solera and Juli\'{a}n Toledo}

\address{J. M. Maz\'{o}n: Departament d'An\`{a}lisi Matem\`atica,
Universitat de Val\`encia, Dr. Moliner 50, 46100 Burjassot, Spain.
 {\tt mazon@uv.es }}
\address{ M. Solera: Departament d'An\`{a}lisi Matem\`atica,
Universitat de Val\`encia, Dr. Moliner 50, 46100 Burjassot, Spain.
 {\tt marcos.solera@uv.es  }
}
\address{J. Toledo: Departament d'An\`{a}lisi Matem\`atica,
Universitat de Val\`encia, Dr. Moliner 50, 46100 Burjassot, Spain.
 {\tt toledojj@uv.es }}


\keywords{Random walk, Total variation flow, Sets of finite perimeter, Nonlocal operators, Cheeger sets, Calibrable sets, Functions of bounded variation.\\
\indent 2010 {\it Mathematics Subject Classification:} 05C80, 35R02, 05C21, 45C99, 26A45.
}

\setcounter{tocdepth}{1}

%

\begin{abstract}
In this paper we study the Total Variation Flow (TVF) in  metric random walk spaces, which unifies into a broad framework the TVF on locally finite weighted connected graphs, the TVF  determined by finite Markov chains and some nonlocal evolution problems. Once the existence and uniqueness of solutions of the TVF has been proved, we study the asymptotic behaviour of those solutions and, with that aim in view, we establish some inequalities of Poincar\'{e} type. In particular, for finite weighted connected graphs, we show that the solutions reach the average of the initial data in finite time. Furthermore, we  introduce the concepts  of perimeter and mean curvature for subsets of a metric random walk space and we study the relation between isoperimetric inequalities and Sobolev inequalities. Moreover, we introduce the concepts of Cheeger and calibrable sets in metric random walk spaces and characterize calibrability by using the $1$-Laplacian operator. Finally, we study the eigenvalue problem  whereby we give a method to solve the  optimal Cheeger cut problem.

\end{abstract}

\maketitle

%
%
%
%
%
%
%
%
%
%
%

{ \renewcommand\contentsname{Contents }
\setcounter{tocdepth}{3}
\tableofcontents
 }

\section{Introduction and Preliminaries}

A  metric random walk space $[X,d,m]$ is a metric space  $(X,d)$ together with a family  $m = (m_x)_{x \in X}$ of probability measures that  encode the jumps of a Markov chain. Important examples of metric random walk spaces are: locally finite weighted connected graphs, finite Markov chains and $[\R^N, d, m^J]$ with $d$ the Euclidean distance and
$$m^J_x(A) :=  \int_A J(x - y) d\mathcal{L}^N(y) \ \hbox{ for every Borel set } A \subset  \R^N \ ,$$
where $J:\R^N\to[0,+\infty[$ is a measurable, nonnegative and radially symmetric
function with $\int J =1$. Furthermore, given a metric measure space $(X,d, \mu)$ satisfying certain properties we can obtain a metric random walk  space $[X, d, m^{\mu,\epsilon}]$,  called the {\it $\epsilon$-step random walk associated to $\mu$}, where
 $$m^{\mu,\epsilon}_x:= \frac{\mu \res B(x, \epsilon)}{\mu(B(x, \epsilon))}.$$

 Since its introduction as a means of solving the denoising problem in the seminal work by Rudin, Osher and Fatemi (\cite{ROF}), the total variation flow has remained one of the most popular tools in Image Processing. Recall that, from the mathematical point of view, the study of the total variation flow in $\R^N$ was established in \cite{ACMBook}. On the other hand, the use of  neighbourhood filters by Buades, Coll and Morel in \cite{BCM}, that was originally proposed by P. Yaroslavsky (\cite{Y1}), has led to an extensive literature in nonlocal models in image processing (see for instance \cite{BEM}, \cite{GO1}, \cite{KOJ}, \cite{LEL} and the references therein). Consequently, there is great interest in studying the total variation flow in the nonlocal context. As further motivation, note that an image can be considered as a weighted graph, where the pixels are taken as the vertices and the ``similarity'' between pixels as the weights. The way in which these weights are defined depends on the problem at hand, see for instance \cite{ELB} and  \cite{LEL}.

The aim of this paper is to study the total variation flow in metric random walk spaces, obtaining  general results that can be applied, for example, to the different points of view in Image Processing. In this regard, we introduce the $1$-Laplacian operator associated with a metric random walk space, as well as the notions of perimeter and mean curvature for subsets of a metric random walk space. In doing so, we generalize results obtained in \cite{MRT1} and \cite{MRTLibro} for the particular case of $[\R^N, d, m^J]$, and, moreover, generalize results in graph theory. We then proceed to prove existence and uniqueness of solutions of the total variation flow in metric random walk spaces and to study its asymptotic behaviour with the help of some Poincar\'{e} type inequalities. Furthermore, we introduce the concepts of Cheeger and calibrable sets in metric random walk spaces and characterize calibrability by using the $1$-Laplacian operator.  Let us point out that, to our knowledge, some of these results were not yet known for graphs, nonetheless,   we have specified in the main text which important results were already known for graphs. Moreover, in the forthcoming paper \cite{MSTDecomp}, we  apply the theory developed here to obtain the $(BV,L^p)$-decomposition, $p=1,2$, of functions in  metric random walk spaces. This decomposition can be applied to Image Processing if, for example, images are regarded as graphs and, moreover, to other nonlocal models.

Partitioning data into sensible groups is a fundamental problem in machine learning, computer science, statistics and science in general. In these fields, it is usual to face large amounts of empirical data, and getting  a first impression of the data by identifying groups with similar properties can prove to be very useful. One of the most popular approaches to this problem is to find the best balanced cut of a graph representing the data, such as the Cheeger ratio cut (\cite{Cheeger}). Consider  a finite weighted connected graph $G =(V, E)$, where $V = \{x_1, \ldots , x_n \}$ is the set of vertices  (or nodes) and $E$ the set of edges, which are weighted by a function  $w_{ji}= w_{ij} \geq 0$, $(i,j) \in E$. The degree of the vertex $x_i$ is denoted by $d_i:= \sum_{j=1}^n w_{ij}$, $i=1,\ldots, n$. In this context, the Cheeger cut value of a partition $\{ S, S^c\}$ ($S^c:= V \setminus S$) of $V$ is defined as
$$\mathcal{C}(S):= \frac{{\rm Cut}(S,S^c)}{\min\{{\rm vol}(S), {\rm vol}(S^c)\}},$$
where
$${\rm Cut}(A,B) = \sum_{i \in A, j \in B} w_{ij}, $$
and ${\rm vol}(S)$ is the volume of $S$, defined as ${\rm vol}(S):= \sum_{i \in S} d_i$. Furthermore,
$$h(G) = \min_{S \subset V} \mathcal{C}(S)$$
is called the Cheeger constant, and a partition $\{ S, S^c\}$ of $V$ is called a Cheeger cut of $G$ if $h(G)=\mathcal{C}(S)$. Unfortunately, the Cheeger minimization problem of computing $h(G)$ is NP-hard (\cite{HB}, \cite{SB1}). However, it turns out that $h(G)$ can be approximated by the second eigenvalue $\lambda_2$ of the graph Laplacian thanks to the following Cheeger inequality (\cite{Ch}):
\begin{equation}\label{CheegerIne1}
\frac{\lambda_2}{2} \leq h(G) \leq \sqrt{2\lambda_2}.
\end{equation}
 This motivates the spectral clustering method (\cite{Luxburg}), which, in its simplest form, thresholds the second eigenvalue of the graph Laplacian to get an approximation to the Cheeger constant and, moreover, to a Cheeger cut. In order to achieve a better approximation than the one provided by the classical spectral clustering method, a spectral clustering based on the graph $p$-Laplacian was developed in \cite{BH1}, where it is showed that the second eigenvalue of the graph $p$-Laplacian  tends to the Cheeger constant $h(G)$ as $p \to 1^+$. In  \cite{SB1} the idea was taken up by directly considering the variational characterization of the Cheeger constant $h(G)$
 \begin{equation}\label{cheegerSB}
 h(G) = \min_{u \in L^1} \frac{ \vert u \vert_{TV}}{\Vert u - {\rm median}(u)) \Vert_1},
 \end{equation}
where
$$\vert u \vert_{TV} := \frac{1}{2} \sum_{i,j=1}^n w_{ij} \vert u(x_i) - u(x_j) \vert.$$
The subdifferential of the energy functional $\vert \cdot \vert_{TV}$ is the $1$-Laplacian in graphs $\Delta_1$. Using the nonlinear eigenvalue problem $0 \in \Delta_1 u - \lambda \, {\rm sign}(u)$, the theory of $1$-Spectral Clustering is developed in \cite{Chang1},  \cite{Changetal01},  \cite{ChSZ} and \cite{HB}, and good results on the Cheeger minimization problem have been obtained.

In \cite{MST0}, we obtained a generalization, in the framework of metric random walk spaces, of the Cheeger inequality \eqref{CheegerIne1} and of the variational characterization of the Cheeger constant \eqref{cheegerSB}. In this paper, in connection with the $1$-Spectral Clustering, also in metric random walk spaces, we study the eigenvalue problem of the $1$-Laplacian and then relate it to the optimal Cheeger cut problem. Then again,  these results apply, in particular, to locally finite weighted connected graphs, complementing the results given in \cite{Chang1}, \cite{Changetal01}, \cite{ChSZ} and \cite{HB}.

Additionally, regarding the notion of a function of bounded variation in a metric measure space $(X,d, \mu)$ introduced by Miranda in \cite{Miranda1}, we provide, via the $\epsilon$-step random walk associated to $\mu$, a characterization of these functions.

 \subsection{Metric Random Walk Spaces}

Let $(X,d)$ be a  Polish metric  space equipped with its Borel $\sigma$-algebra.
A {\it random walk} $m$ on $X$ is a family of probability measures $m_x$ on $X$, $x \in X$, satisfying the two technical conditions: (i) the measures $m_x$  depend measurably on the point  $x \in X$, i.e., for any Borel set $A$ of $X$ and any Borel set $B$ of $\R$, the set $\{ x \in X \ : \ m_x(A) \in B \}$ is Borel; (ii) each measure $m_x$ has finite first moment, i.e. for some (hence any,  by the triangle inequality) $z \in X$, and for any $x \in X$ one has $\int_X d(z,y) dm_x(y) < +\infty$ (see~\cite{O}).

A {\it metric random walk  space} $[X,d,m]$  is a  Polish metric space $(X,d)$ together with a  random walk $m$ on $X$.

Let $[X,d,m]$ be a metric random walk  space. A   Radon measure $\nu$ on $X$ is {\it invariant} for the random walk $m=(m_x)$ if
$$d\nu(x)=\int_{y\in X}d\nu(y)dm_y(x),$$
that is,  for any $\nu$-measurable set $A$, it holds that $A$ is $m_x$-measurable  for $\nu$-almost all $x\in X$,  $\displaystyle x\mapsto  m_x(A)$ is $\nu$-measurable, and
$$\nu(A)=\int_X m_x(A)d\nu(x).$$
Consequently, if $\nu$ is an invariant measure with respect to $m$ and $f \in L^1(X, \nu)$, it holds that $f \in L^1(X, m_x)$ for $\nu$-a.e. $x \in X$, $\displaystyle x\mapsto \int_X f(y) d{m_x}(y)$ is $\nu$-measurable, and \label{paginv}
$$\int_X f(x) d\nu(x) = \int_X \left(\int_X f(y) d{m_x}(y) \right)d\nu(x).$$

The measure $\nu$ is said to be {\it reversible} for $m$ if, moreover, the following detailed balance condition holds:
\begin{equation}\label{repo001} dm_x(y)d\nu(x)  =   dm_y(x)d\nu(y),
 \end{equation}
 that is, for any Borel set $C \subset X \times X$,  $$  \int_{X}\left(\int_X \1_{C}(x,y)  dm_x(y)\right)d\nu(x)  =  \int_X\left(\int_X\1_C(x,y) dm_y(x)\right)d\nu(y),$$
  where $\1_{C}$ is the characteristic function of the set $C$ defined as $$\1_{C}(x):=\left\{ \begin{array}{ll}
1 & \hbox{if } x\in C,\\
0 & \hbox{otherwise.}
\end{array}\right.$$
Note that the reversibility condition implies the invariance condition. However, we will sometimes write that $\nu$ is invariant and reversible so as to emphasize both conditions.

 We now give some examples of metric random walk spaces that illustrate the general abstract setting. In particular, Markov chains serve as paradigmatic examples that capture many of the properties of this general setting that we will encounter during our study.

 \begin{example}\label{JJ}
  (1)
Consider $(\R^N, d, \mathcal{L}^N)$, with $d$ the Euclidean distance and $\mathcal{L}^N$ the Lebesgue measure. For simplicity we will write $dx$ instead of $d\mathcal{L}^N(x)$. Let  $J:\R^N\to[0,+\infty[$ be a measurable, nonnegative and radially symmetric
function  verifying  $\int_{\R^N}J(x)dx=1$. In $(\R^N, d, \mathcal{L}^N)$ we have the following random walk, starting at $x$,
$$m^J_x(A) :=  \int_A J(x - y) dy \quad \hbox{ for every Borel set } A \subset  \R^N  .$$
Applying Fubini's Theorem it is easy to see that the Lebesgue measure $\mathcal{L}^N$ is an invariant and reversible measure for this random walk.

 Observe that, if we assume that in $\R^N$ we have an homogeneous population and $J(x-y)$ is thought of as the probability distribution of jumping from location $x$ to location $y$, then, for a Borel set $A$ in $\R^N$, $m^J_x(A)$ is measuring how many individuals are going to $A$ from $x$ following the law given by~$J$.
See also the interpretation of the $m$-interaction between sets given in Section~\ref{lasecper}. Finally, note that the same ideas are applicable to the countable spaces given in the following two examples.

\noindent (2)  Let $K: X \times X \rightarrow \R$ be a Markov kernel on a countable space $X$, i.e.,
 $$K(x,y) \geq 0 \quad \forall x,y \in X, \quad \quad \sum_{y\in X} K(x,y) = 1 \quad \forall x \in X.$$
 Then, for $$m^K_x(A):= \sum_{y \in A} K(x,y),$$
 $[X, d, m^K]$ is a metric random walk space for  any  metric  $d$ on $X$.

  Moreover, in Markov chain theory terminology, a measure $\pi $ on $X$ satisfying
 $$\sum_{x \in X} \pi(x) = 1 \quad \hbox{and} \quad \pi(y) = \sum_{x \in X} \pi(x) K(x,y) \quad  \forall y \in X,$$
 is called a stationary probability measure (or steady state) on $X$. This is equivalent to the definition of invariant probability measure for the metric random walk space $[X, d, m^K]$. In general, the existence of such a stationary probability measure on $X$ is not guaranteed. However, for irreducible and positive recurrent Markov chains (see, for example,~\cite{HLL} or~\cite{Norris}) there exists a unique stationary probability measure.

Furthermore, a stationary probability measure $\pi$ is said to be reversible for $K$ if the following detailed balance equation holds:
 $$K(x,y) \pi(x) = K(y,x) \pi(y) \ \hbox{ for } x, y \in X.$$
 By Tonelli's Theorem for series, this balance condition is equivalent to the one given in~\eqref{repo001} for $\nu=\pi$:
 $$dm^K_x(y)d\pi(x)  =   dm^K_y(x)d\pi(y).$$

\noindent (3) Consider  a locally finite weighted discrete graph $G = (V(G), E(G))$, where each edge $(x,y) \in E(G)$ (we will write $x\sim y$ if $(x,y) \in E(G)$) has a positive weight $w_{xy} = w_{yx}$ assigned. Suppose further that $w_{xy} = 0$ if $(x,y) \not\in E(G)$.

 A finite sequence $\{ x_k \}_{k=0}^n$  of vertices on the graph is called a {\it  path} if $x_k \sim x_{k+1}$ for all $k = 0, 1, ..., n-1$. The {\it length} of a path $\{ x_k \}_{k=0}^n$ is defined as the number $n$ of edges in the path. Then, $G = (V(G), E(G))$ is said to be {\it connected} if, for any two vertices $x, y \in V$, there is a path connecting $x$ and $y$, that is, a path $\{ x_k \}_{k=0}^n$ such that $x_0 = x$ and $x_n = y$.  Finally, if $G = (V(G), E(G))$ is connected, define the graph distance $d_G(x,y)$ between any two distinct vertices $x, y$ as the minimum of the lengths of the paths connecting $x$ and $y$. Note that this metric is independent of the weights. We will always assume that the graphs we work with are connected.

For $x \in V(G)$ we define the weight   at the vertex $x$ as $$d_x:= \sum_{y\sim x} w_{xy} = \sum_{y\in V(G)} w_{xy},$$
and the neighbourhood of $x$ as $N_G(x) := \{ y \in V(G) \, : \, x\sim y\}$. Note that, by definition of locally finite graph, the sets $N_G(x)$ are finite. When $w_{xy}=1$ for every $x\sim y$, $d_x$ coincides with the degree of the vertex $x$ in a graph, that is,  the number of edges containing vertex $x$.

  For each $x \in V(G)$  we define the following probability measure
\begin{equation}\label{discRW}m^G_x:=  \frac{1}{d_x}\sum_{y \sim x} w_{xy}\,\delta_y.\\ \\
\end{equation}
 We have that $[V(G), d_G, m^G]$ is a metric random walk space and it is not difficult to see that the measure $\nu_G$ defined as
 $$\nu_G(A):= \sum_{x \in A} d_x,  \quad A \subset V(G),$$
is an invariant and  reversible measure for this random walk.

Given a locally finite weighted discrete graph $G = (V(G), E(G))$, there is a natural definition of a Markov chain
on the vertices. We define the Markov kernel  $K_G: V(G)\times V(G) \rightarrow \R$ as
$$K_G(x,y):= \frac{1}{d_x}  w_{xy}.$$
We have that $m^G$ and $m^{K_G}$ define the same random walk.  If $\nu_G(V(G))$ is finite,   the unique stationary and reversible probability measure is given by
$$\pi_G(x):= \frac{1}{\nu_G(V(G))} \sum_{z \in V(G)} w_{xz}. $$

 \noindent (4) From a metric measure space $(X,d, \mu)$ we can obtain a metric random walk space, the so called {\it $\epsilon$-step random walk associated to $\mu$}, as follows. Assume that balls in $X$ have finite measure and that ${\rm Supp}(\mu) = X$. Given $\epsilon > 0$, the $\epsilon$-step random walk on $X$ starting at $x\in X$, consists in randomly jumping in the ball of radius $\epsilon$ centered at $x$ with probability proportional to $\mu$; namely
 $$m^{\mu,\epsilon}_x:= \frac{\mu \res B(x, \epsilon)}{\mu(B(x, \epsilon))}.$$
Note that $\mu$ is an invariant and reversible measure for the metric random walk space $[X, d, m^{\mu,\epsilon}]$.

 \noindent (5) Given a  metric random walk  space $[X,d,m]$ with invariant and reversible measure $\nu$, and given a $\nu$-measurable set $\Omega \subset X$ with $\nu(\Omega) > 0$, if we define, for $x\in\Omega$,
$$m^{\Omega}_x(A):=\int_A d m_x(y)+\left(\int_{X\setminus \Omega}d m_x(y)\right)\delta_x(A) \ \hbox{ for every Borel set } A \subset  \Omega  ,
$$
 we have that $[\Omega,d,m^{\Omega}]$ is a metric random walk space and it easy to see that $\nu \res \Omega$ is  reversible for $m^{\Omega}$.

In particular, if $\Omega$ is a closed and bounded subset of $\R^N$, we obtain the metric random walk space $[\Omega, d, m^{J,\Omega}]$, where $m^{J,\Omega} = (m^J)^{\Omega}$, that is
$$m^{J,\Omega}_x(A):=\int_A J(x-y)dy+\left(\int_{\R^n\setminus \Omega}J(x-z)dz\right)d\delta_x \ \hbox{ for every Borel set } A \subset  \Omega .$$

\end{example}

 From this point onwards, when dealing with a metric random walk space, we will assume that there exists an invariant and reversible measure for the random walk, which we will always denote by $\nu$. In this regard, when it is clear from the context, a measure denoted by $\nu$ will always be an invariant and reversible measure for the random walk under study. Furthermore, we assume that the metric measure space $(X,d,\nu)$ is $\sigma$-finite.

 \subsection{Completely Accretive Operators and Semigroup Theory}\label{secacc}
Since Semigroup Theory will be used along the paper, we would like to conclude this introduction with some notations and results from this theory along with results from the theory of completely accretive operators (see \cite{BCr2}, \cite{Brezis} and \cite{CrandallLiggett}, or   the Appendix in \cite{ElLibro}, for more details).
We denote by $J_0$ and $P_0$ the following sets of functions:
$$J_0 := \{ j : \R \rightarrow [0, +\infty] \ : \ \mbox{$j$ is convex,
lower semi-continuous and} \ j(0) = 0 \},$$
  $$ P_0:= \left\{q\in  C^\infty(\R) \ : \ 0\le q'\le 1, \hbox{ supp}(q')  \hbox{ is compact and }
  0\notin \hbox{supp}(q) \right\}.
  $$
Let $u,v\in L^1(X,\nu)$. The following relation between $u$ and $v$ is defined in \cite{BCr2}:
\begin{equation}\label{menormenor}u\ll v \ \hbox{ if, and only if,} \ \int_{X} j(u)\,  d\nu  \leq \int_{X} j(v)
 \, d\nu \ \ \hbox{for all} \ j \in J_0.\end{equation}
An operator $\mathcal{A} \subset L^1(X,\nu)\times L^1(X,\nu)$ is called {\it completely accretive} if
$$\int_X (v_1 - v_2) q(u_1 - u_2)d\nu \geq 0 \quad \hbox{for every} \ \ q \in P_0$$
and every $(u_i, v_i) \in \mathcal{A}$, $i=1,2$. Moreover, an operator $\mathcal{A}$ in $L^1(X,\nu)$   is m-{\it
completely accretive} in $L^1(X,\nu)$ if $\mathcal{A}$ is completely accretive and ${Range}(I
+ \lambda \mathcal{A}) = L^1(X,\nu)$ for all $\lambda > 0$ (or, equivalently, for some $\lambda>0$).

\begin{theorem}[\cite{BCr2}, \cite{Brezis}]\label{teointronls}
   If $\mathcal{A}$
 is an m-completely accretive operator in $L^1(X,\nu)$, then, for every $u_0 \in \overline{D(\mathcal{A})}$ (the closure of the domain of $\mathcal{A}$), there exists a unique mild solution (see \cite{CrandallLiggett})
  of the problem
\begin{equation}\label{AACCPP}
\left\{
\begin{array}{l}
\displaystyle \frac{du}{dt} + \mathcal{A}u \ni 0, \\[10pt]
u(0) = u_{0}.
\end{array}
\right.
\end{equation}
Moreover, if $\mathcal{A}$ is the subdifferential of a proper convex and lower semicontinuous function in $L^2(X,\nu)$ then the mild solution of the above problem is a strong solution.

Furthermore  we have the following contraction and maximum principle in any $L^q(X,\nu)$ space, $1\le q\le +\infty$: for $u_{1,0},u_{2,0} \in \overline{D(\mathcal{A})}$ and denoting by $u_i$  the  unique mild solution
  of the problem
\begin{equation}\label{AACCPPiii}
\left\{
\begin{array}{l}
\displaystyle \frac{du_i}{dt} + \mathcal{A}u_i \ni 0, \\[10pt]
u_i(0) = u_{i,0},
\end{array}
\right.
\end{equation}
$ i=1,2,$ we have
\begin{equation}\label{maxpriniiii}\Vert (u_1(t)-u_2(t))^+\Vert_{L^q(X,\nu)}\le \Vert (u_{1,0}-u_{2,0})^+\Vert_{L^q(X,\nu)}\quad \forall\, 0<t<T,
\end{equation}
 where $r^+:=\max \{r,0\}$ for $r\in\R$.
\end{theorem}

\section{Perimeter, Curvature and Total Variation in Metric Random Walk Spaces} \label{sect-perimeter}

\subsection{$m$-Perimeter}\label{lasecper}
 Let $[X,d,m]$ be a metric random walk  space with invariant and  reversible measure $\nu$. We define the {\it $m$-interaction} between two $\nu$-measurable subsets $A$ and $B$ of $X$ as
\begin{equation}\label{nlinterdos} L_m(A,B):= \int_A \int_B dm_x(y) d\nu(x).
 \end{equation}
Whenever  $L_m(A,B) < +\infty$, by the reversibility assumption on $\nu$ with respect to $m$, we have
 $$L_m(A,B)=L_m(B,A).$$

  Following the interpretation given after Example~\ref{JJ}~(1), for a $\nu$-homogeneous population which moves according to the law provided by the random walk $m$, $L_m(A,B)$ measures how many individuals are moving from $A$ to $B$, and, thanks to the reversibility, this is equal to the amount of individuals moving from $B$ to $A$. In this regard, the following concept  measures the total flux of individuals that cross the ``boundary''   (in a very weak sense)  of a set.

 We define the concept of $m$-perimeter of a $\nu$-measurable subset $E \subset X$ as
 $$P_m(E)=L_m(E,X\setminus E) = \int_E \int_{X\setminus E} dm_x(y) d\nu(x).$$
 It is easy to see that
\begin{equation}\label{secondf}P_m(E) = \frac{1}{2} \int_{X}  \int_{X}  \vert \1_{E}(y) - \1_{E}(x) \vert dm_x(y) d\nu(x).
\end{equation}
Moreover,  if $E$ is $\nu$-integrable, we have
 \begin{equation}\label{secondf021}\displaystyle P_m(E)=\nu(E) -\int_E\int_E dm_x(y) d\nu(x).
\end{equation}

The notion of $m$-perimeter can be localized to a bounded open set $\Omega \subset X$ by defining
 \begin{equation}\label{nonlp} \displaystyle P_m(E, \Omega):= L_m( E\cap\Omega, X \setminus E) + L_m(E \setminus \Omega,  \Omega \setminus E).
\end{equation}
Observe that
\begin{equation}\label{nonl2}\displaystyle  L_m(E, X\setminus E) =L_m( E\cap\Omega, X \setminus E) + L_m(E \setminus \Omega,  \Omega \setminus E)+L_m(E\setminus \Omega, X \setminus (E \cup \Omega)) \end{equation}
and, consequently, we have
\begin{equation}\label{nonl3}P_m(E, \Omega) = \int_E \int_{X \setminus E} dm_x(y) d\nu(x) - \int_{E\setminus \Omega} \int_{X \setminus (E \cup \Omega)} dm_x(y) d\nu(x),
\end{equation}
 when both integrals are finite.

\begin{example}\label{dt}
(1) Let  $[\R^N, d, m^J]$ be   the metric random walk space given in Example   \ref{JJ}~(1) with invariant measure $\mathcal{L}^N$.  Then,
    $$P_{m^J} (E) = \frac{1}{2} \int_{\R^N}  \int_{\R^N}  \vert \1_{E}(y) - \1_{E}(x) \vert J(x -y) dy dx,$$
    which coincides with the concept of $J$-perimeter introduced in \cite{MRT1}. On the other hand,
    $$P_{ m^{J,\Omega}} (E) = \frac{1}{2} \int_{\Omega}  \int_{\Omega}  \vert \1_{E}(y) - \1_{E}(x) \vert J(x -y) dy dx.$$
    Note that, in general, $P_{ m^{J,\Omega}} (E) \not= P_{m^J} (E).$

    Moreover,
    $$P_{ m^{J,\Omega}} (E) = \mathcal{L}^N(E) - \int_E \int_E dm_x^{J,\Omega}(y) dx = \mathcal{L}^N(E) - \int_E \int_E J(x-y) dy dx - \int_E \left( \int_{\R^N \setminus \Omega} J(x - z) dz\right) dx$$
    and, therefore,
    \begin{equation}\label{perimomega}
    P_{ m^{J,\Omega}} (E) = P_{ m^{J}} (E)  - \int_E \left( \int_{\R^N \setminus \Omega} J(x - z) dz\right) dx, \quad \forall \, E \subset \Omega.
    \end{equation}

\noindent (2)
In the case of the metric random walk space $[V(G), d_G, m^G ]$ associated to a finite weighted discrete graph $G$, given $A, B \subset V(G)$, ${\rm Cut}(A,B)$ is defined as
$${\rm Cut}(A,B):= \sum_{x \in A, y \in B} w_{xy} = L_{m^G}(A,B),$$
and the perimeter of a set $E \subset V(G)$ is given by
$$\vert \partial E \vert := {\rm Cut}(E,E^c) = \sum_{x \in E, y \in V \setminus E} w_{xy}.$$
Consequently, we have that
\begin{equation}\label{perim}
\vert \partial E \vert = P_{m^G}(E) \quad \hbox{for all} \ E \subset V(G).
\end{equation}

\end{example}

Let us now give some properties of the $m$-perimeter.

  \begin{proposition}\label{launion01}\ Let $A,\ B \subset X$ be $\nu$-measurable sets with finite $m$-perimeter such that $\nu(A \cap B) = 0$. Then,
  $$P_m( A \cup B) = P_m( A) + P_m(B) - 2 L_m(A,B).$$

  \end{proposition}
\begin{proof} We have
$$
\begin{array}{l}
\displaystyle
P_m( A \cup B) = \int_{ A \cup B} \left(\int_{X \setminus (A \cup B)} dm_x(y) \right) d\nu(x) \\[10pt]
\displaystyle \qquad =  \int_{ A} \left(\int_{X \setminus (A \cup B)} dm_x(y) \right) d\nu(x) +  \int_{B} \left(\int_{X \setminus (A \cup B)} dm_x(y) \right) d\nu(x)
\\[10pt]
\displaystyle \qquad =  \int_{ A} \left(\int_{X\setminus A} dm_x(y) - \int_B dm_x(y)  \right) d\nu(x) +  \int_{B} \left(\int_{X \setminus B} dm_x(y) -  \int_A dm_x(y)\right) d\nu(x),
\end{array}$$
and then, by the reversibility assumption on $\nu$ with respect to $m$,
$$
\begin{array}{l}
\displaystyle
P_m( A \cup B) =
P_m( A) + P_m(B)- 2 \int_A \left( \int_B dm_x(y)\right) d\nu(x).
\end{array}$$

\end{proof}

  \begin{corollary}\label{Juli1}
Let $A,\ B,\ C$ be $\nu$-measurable sets in $X$ with pairwise $\nu$-null intersections. Then
 $$ P_m( A \cup B\cup C)=P_m( A \cup B) +P_m(A\cup C) + P_m(B\cup C) - P_m( A) - P_m(B)-P_m(C) .$$
\end{corollary}

\subsection{$m$-Mean Curvature}\label{def.curvatura}
 Let $E \subset X$ be $\nu$-measurable. For a point $x  \in X$ we define  the {\it $m$-mean curvature of $\partial E$ at $x$} as
\begin{equation}\label{defcurdefdef}H^m_{\partial E}(x):= \int_{X}  (\1_{X \setminus E}(y) - \1_E(y)) dm_x(y).\end{equation}
Observe that
\begin{equation}\label{defcur}H^m_{\partial E}(x) =  1 - 2 \int_E  dm_x(y).\end{equation}

Note that $H^m_{\partial E}(x)$ can be computed for every $x \in X$, not only for points in $\partial E$. This fact will be used later in the paper. Having in mind \eqref{secondf021}, we have that,  for a $\nu$--integrable set $E\subset X$,
$$\int_E H^m_{\partial E}(x) d\nu(x) = \int_E \left( 1 - 2 \int_E  dm_x(y) \right)  d\nu(x) = \nu(E) - 2\int_E\int_E dm_x(y) d\nu(x)$$ $$ = P_m(E) - \int_E\int_E dm_x(y) d\nu(x) = 2P_m(E) -\nu(E).$$ Consequently,
 \begin{equation}\label{1secondf021}\displaystyle \int_E H^m_{\partial E}(x) d\nu(x)=2P_m(E) -\nu(E).
\end{equation}

\subsection{$m$-Total Variation}
Associated to the  random walk $m=(m_x)$ and the invariant measure $\nu$,  we define the space
 $$BV_m(X,\nu):= \left\{ u :X \to \R \ \hbox{ $\nu$-measurable} \, : \, \int_{X}  \int_{X}  \vert u(y) - u(x) \vert dm_x(y) d\nu(x) < \infty \right\}.$$
 We have that $L^1(X,\nu)\subset BV_m(X,\nu)$. The {\it $m$-total variation} of a function $u\in BV_m(X,\nu)$ is defined by
 $$TV_m(u):= \frac{1}{2} \int_{X}  \int_{X}  \vert u(y) - u(x) \vert dm_x(y) d\nu(x).$$
 Note that
\begin{equation}\label{theabove01}P_m(E) = TV_m(\1_E).\end{equation}

Observe that the space $BV_m(X,\nu)$ is the {\it nonlocal} counterpart of classical local bounded variation spaces. Note further that, in the local context, given a Lebesgue measurable set $E\subset\R^n$, its perimeter is equal to the total variation of its characteristic function (see~\eqref{forarb001}) and the above equation~\eqref{theabove01} provides the nonlocal counterpart. In~\eqref{connvver1} and Theorem~\ref{Thequiv} we illustrate further relations between these spaces.

 However, although they represent analogous concepts in different settings, the classical local BV-spaces and the nonlocal BV-spaces are of a different nature. For example, in our nonlocal framework $L^1(X,\nu)\subset BV_m(X,\nu)$ in contrast with classical local bounded variation spaces that are, by definition, contained in $L^1$. Indeed, since each $m_x$ is a probability measure, $x\in X$, and $\nu$ is invariant with respect to $m$, we have that
$$TV_m(u)\leq \frac12 \int_X\int_X |u(y)|dm_x(y)d\nu(x)+\frac12\int_X\int_X|u(x)|dm_x(y)d\nu(x)= \Vert u\Vert_{L^1(X,\nu)}.$$

 Recall the definition of the generalized product measure $\nu \otimes m_x$ (see, for instance, \cite{AFP}), which is defined as the measure on $X \times X$ given by
 \begin{equation}\label{forarb002} \nu \otimes m_x(U) := \int_X   \int_X \1_{U}(x,y) dm_x(y)   d\nu(x)\quad\hbox{for }  U\in \mathcal{B}(X\times X),
 \end{equation} where it is required that the map $x \mapsto m_x(E)$ is $\nu$-measurable for any Borel set $E \in \mathcal{B}(X)$.
 Moreover, it holds that $$\int_{X \times X} g  d(\nu \otimes m_x)   = \int_X   \int_X g(x,y) dm_x(y)  d\nu(x)$$
for every  $g\in L^1(X\times X,\nu\otimes m_x)$. Therefore, we can write
$$TV_m(u)= \frac{1}{2} \int_{X\times X}  \vert u(y) - u(x) \vert d(\nu \otimes m_x)(x,y).$$

\begin{example}\label{exammplee} Let $[V(G), d_G, (m^G_x)]$  be the metric random walk space given in Example   \ref{JJ}~(3) with invariant and reversible  measure $\nu_G$. Then,
$$TV_{m^G} (u) = \frac{1}{2} \int_{V(G)}  \int_{V(G)}  \vert u(y) - u(x) \vert dm^G_x(y) d\nu_G(x) = \frac{1}{2} \int_{V(G)} \frac{1}{d_x} \left(\sum_{y \in V(G)} \vert u(y) - u(x) \vert w_{xy}\right) d\nu_G(x)$$ $$=  \frac{1}{2} \sum_{x \in V(G)} d_x \left(\frac{1}{d_x} \sum_{y \in V(G)} \vert u(y) - u(x) \vert w_{xy}\right) = \frac{1}{2}  \sum_{x \in V(G)} \sum_{y \in V(G)} \vert u(y) - u(x) \vert w_{xy},$$
which coincides with the anisotropic total variation
  defined in \cite{GGOB}.
\end{example}

In the following results we give some properties of the total variation.

 \begin{proposition}\label{lemLipsch} If $\phi : \R \rightarrow \R$ is Lipschitz continuous  then, for every $u \in BV_m(X,\nu)$, $\phi(u) \in BV_m(X,\nu)$ and
  $$TV_m(\phi(u)) \leq \Vert \phi \Vert_{Lip} TV_m(u).$$
\end{proposition}
\begin{proof}
$$TV_m(\phi(u)) = \frac{1}{2} \int_{X}  \int_{X}  \vert \phi(u)(y) - \phi(u)(x) \vert dm_x(y) d\nu(x) $$ $$\leq \Vert \phi \Vert_{Lip} \frac{1}{2} \int_{X}  \int_{X}  \vert u(y) - u(x) \vert dm_x(y) d\nu(x) = \Vert \phi \Vert_{Lip}  TV_m(u).$$
\end{proof}

 \begin{proposition}\label{lemsc1} $TV_m$ is convex and continuous in $L^1(X, \nu)$.
\end{proposition}
\begin{proof} Convexity follows easily. Let us see that it is continuous. Let $u_n \to u$ in $L^1(X, \nu)$. Since $\nu$ is invariant and  reversible with respect to $m$, we have
$$ \vert  TV_m(u_n) - TV_m(u) \vert =  \frac{1}{2}  \left\vert \int_{X}  \int_{X}  \left(\vert u_n(y) - u_n(x) \vert - \vert u(y) - u(x) \vert\right) dm_x(y) d\nu(x)\right\vert$$ $$\leq \frac{1}{2} \left( \int_{X}  \int_{X} \vert u_n(y) - u(y) \vert  dm_x(y) d\nu(x) +\int_{X}  \int_{X} \vert u_n(x) - u(x) \vert  dm_x(y) d\nu(x)\right)$$ $$= \frac{1}{2} \left( \int_{X}   \vert u_n(y) - u(y) \vert   d\nu(y) +\int_{X}  \vert u_n(x) - u(x) \vert  d\nu(x)\right) = \Vert u_n - u \Vert_{L^1(X, \nu)}.$$
\end{proof}

As in the local case, we have the following coarea formula relating the total variation of a function with the perimeter of its superlevel sets.

\begin{theorem}[\bf Coarea formula]\label{coarea1}
 For any $u \in L^1(X,\nu)$, let $E_t(u):= \{ x \in X \ : \ u(x) > t \}$. Then,
\begin{equation}\label{coaerea}
TV_m(u) = \int_{-\infty}^{+\infty} P_m(E_t(u))\, dt.
\end{equation}
\end{theorem}

\begin{proof} Since
\begin{equation}\label{cooare1}
u(x) = \int_0^{+\infty} \1_{E_t(u)}(x) \, dt - \int_{-\infty}^0 (1 - \1_{E_t(u)}(x)) \, dt,
\end{equation}
we have
$$ u(y) - u(x)  = \int_{-\infty}^{+\infty}  \1_{E_t(u)}(y) - \1_{E_t(u)} (x) \, dt.$$
Moreover, since  $u(y) \geq u(x)$ implies $\1_{E_t(u)}(y) \geq \1_{E_t(u)} (x)$, we obtain that
$$\vert u(y) - u(x) \vert = \int_{-\infty}^{+\infty} \vert \1_{E_t(u)}(y) - \1_{E_t(u)} (x) \vert \,dt.$$
Therefore, we get
$$
\begin{array}{l}
\displaystyle
TV_m(u) =  \frac{1}{2} \int_{X}  \int_{X} \vert u(y) - u(x) \vert  dm_x(y) d\nu(x)
\displaystyle \\ \\ \qquad =  \displaystyle\frac{1}{2} \int_{X}  \int_{X}  \left(\int_{-\infty}^{+\infty} \vert \1_{E_t(u)}(y) - \1_{E_t(u)} (x) \vert dt \right)  dm_x(y) d\nu(x)
\\ \\
\displaystyle \qquad = \int_{-\infty}^{+\infty} \left(\frac{1}{2} \int_{X}  \int_{X}  \vert \1_{E_t(u)}(y) - \1_{E_t(u)} (x) \vert  dm_x(y) d\nu(x) \right) dt = \int_{-\infty}^{+\infty} P_m(E_t(u)) dt,
\end{array}
$$
 where Tonelli-Hobson's Theorem is used in the third equality.
\end{proof}

 Let us recall the following concept of  $m$-connectedness introduced in~\cite{MRT1}:
 A metric random walk space $[X,d,m]$ with invariant and reversible  measure $\nu$ is $m$-connected if,  for any pair of $\nu$-non-null measurable sets $ A,B\subset X$ such that
$A\cup B=X$, we have $L_m(A,B)> 0$.
Moreover, in~ \cite[Theorem 2.19]{MST0}, we see that  this concept is equivalent to the following concept of ergodicity (see~\cite{HLL}) when $\nu$ is a probability measure.
\begin{definition}  Let $[X, d, m]$ be a metric random walk space with invariant and reversible   probability measure~$\nu$. A Borel set $B \subset X$ is said to be {\it invariant} with respect to the random walk $m$ if $m_x(B) = 1$ whenever $x$ is in $B$.
 The invariant  probability measure~$\nu$ is said to be {\it ergodic} if $\nu(B) = 0$ or $\nu(B) = 1$ for every invariant set $B$ with respect to the random walk $m$.

\end{definition}

 Furthermore, by \cite[Theorem 2.21]{MST0}, we have that $\nu$ is  ergodic if, and only if,  for $u\in L^2(X,\nu)$, $\Delta_m u = 0$ implies that $u$ is $\nu$-a.e. equal to a constant, where
$$\Delta_m u(x) := \int_X (u(y) - u(x)) dm_x(y). $$

 As an example, note that the metric random walk space associated to an irreducible and positive recurrent Markov chain on a countable space together with its steady state is $m$-connected (see~\cite{HLL}).  Moreover, the metric random walk space $[V(G), d_G, m^G]$ associated to a locally finite weighted  connected  discrete graph $G = (V(G), E(G))$ is  $m^G$-connected.  In~\cite{MST0} we give further examples involving the metric random walk space given in  Example~\ref{JJ}~(1).

 Observe that, for a metric random walk space $[X,d,m]$ with invariant and reversible measure $\nu$, if the space is $m$-connected, then the $m$-perimeter of any $\nu$-measurable set $E$ with $0<\nu(E)<\nu(X)$ is positive.

  \begin{lemma}\label{lemita1} Assume that $\nu$ is ergodic and let $u\in BV_m(X,\nu)$. Then,
 $$TV_m(u) = 0 \iff u \ \hbox{is constant} \ \nu-\hbox{a.e.}.$$
 \end{lemma}

 \begin{proof} ($\Leftarrow$)  Suppose that $u$ is $\nu$-a.e. equal to a constant $k$, then, since $\nu$ is invariant with respect to $m$, we have
 $$\begin{array}{l}
 \displaystyle
 TV_m(u) = \frac{1}{2} \int_{X}\int_X  \vert u(y) - u(x) \vert dm_x(y)d\nu(x)\\[10pt]
 \displaystyle
 \phantom{TV_m(u)} =\int_{X}\int_X  \vert u(y) - k \vert dm_x(y)d\nu(x)
 \\[10pt]
 \displaystyle
 \phantom{TV_m(u)} =\int_{X}  \vert u(x) - k \vert  d\nu(x)=0.
 \end{array}
 $$
  ($\Rightarrow$) Suppose that $$0 = TV_m(u) = \frac{1}{2} \int_{X}\int_X  \vert u(y) - u(x) \vert dm_x(y)d\nu(x).$$
Then, $\int_X |u(y)-u(x)| dm_x(y)=0$ for $\nu$-a.e. $x\in X$, thus
 $$|\Delta_mu(x)|=\left|\int_X \big(u(y)-u(x)\big) dm_x(y)\right|\le \int_X |u(y)-u(x)| dm_x(y)=0 \quad \hbox{for $\nu$-a.e. $x\in X$},$$
and we are done by the comments preceding the lemma.
\end{proof}

  From now on we will assume that the metric random walk spaces that we work with are $m$-connected (this assumption is only dropped in subsection 2.5).  However, we would like to point out that if a metric random walk space $[X,d,m]$ is not $m$-connected then it may be broken down as $X=A\cup B$ where $A$, $B\subset X$ have $\nu$-positive measure and $L_m(A,B)=0$, allowing us to work with $A$ and $B$ independently. Then, for example, if $E\subset X$ is a $\nu$-measurable set we get $$P_m(E)=P_m(E\cap A)+P_m(E\cap B)$$
and, if $u\in BV_m(X,\nu)$, $$TV_m(u)= \frac{1}{2} \int_{A}\int_A  \vert u(y) - u(x) \vert dm_x(y)d\nu(x)+ \frac{1}{2} \int_{B}\int_B  \vert u(y) - u(x) \vert dm_x(y)d\nu(x).$$

\subsection{Isoperimetric and Sobolev inequalities} The $n$-dimensional {\it isoperimetric inequality} states that
\begin{equation}\label{nisop}
\mathcal{L}^n(\Omega){}^{\frac{n-1}{n}} \leq c_n \mathcal{H}^{n-1}(\partial \Omega)
\end{equation}
for every domain $\Omega \subset \R^n$ with smooth boundary and compact closure, where $c_n = \frac{1}{n \omega_n}$, and $\omega_n$ is the volume of the unit ball. It is well known (see for instance \cite{Mazya}) that \eqref{nisop} is equivalent to the {\it Sobolev inequality }
\begin{equation}\label{sobolinq}
\Vert u \Vert_{\frac{n}{n-1}} \leq c_n \int_{\R^n} \vert \nabla u \vert dx \quad \forall  u \in C_0^{\infty}(\R^n).
\end{equation}

If we replace the Euclidean space $\R^n$ by a  Riemannian manifold $M$ with   measure $\mu_n$,
then the isoperimetric inequality takes the following form:
\begin{equation}\label{nisopmanif}
\mu_n(\Omega)^{\frac{n-1}{n}} \leq C_n \mu_{n-1}(\partial \Omega)
\end{equation}
for all bounded sets $\Omega \subset M$ with smooth boundary, being $\mu_{n-1}$ the surface measure. As in the Euclidean case  (see \cite{MazyaBook} or  \cite{S-C}),   \eqref{nisopmanif} is equivalent to the Sobolev inequality
\begin{equation}\label{SoboInq1}
\left(\int_M \vert u \vert^{\frac{n}{n-1}} d\mu_n \right)^{\frac{n-1}{n}} \leq C_n \int_M \vert \nabla u \vert d\mu_n \quad \forall u \in C_0^{\infty} (M).
\end{equation}
Consequently, it is natural to say that a Riemannian manifold $M$ has {isoperimetric dimension $n$} if \eqref{SoboInq1} holds (see~\cite{CGL}).
 The equivalence between isoperimetric inequalities and Sobolev inequalities in the context of Markov chains was obtained by Varopoulos in \cite{V}.  Let us state these results under the context treated here.

\begin{definition} Let $[X,d,m]$ be a metric random walk space with invariant and reversible measure $\nu$.   We say that $[X,d,m,\nu]$ has  {\it isoperimetric dimension $n$} if there exists a constant  $I_n>0$ such that
\begin{equation}\label{Isop1}
\nu(A)^{\frac{n-1}{n}} \leq I_n P_m(A) \quad \hbox{for all $A\subset X$ with }   0<\nu(A)<\nu(X).
\end{equation}
\end{definition}
 We assume that, for $n = 1$, $ \frac{n}{n-1} = +\infty$ by convention.

 We will denote by $BV^0_m(X,\nu)$ the set of functions $u \in BV_m(X,\nu)$ satisfying that there exists $A\subset X$, with $ 0 < \nu(A)<\nu(X)$, such that $u=0$ in $X\setminus A$.

\begin{theorem}\label{IsoPoint}  $[X,d,m,\nu]$ has   isoperimetric dimension $n$ if, and only if,
\begin{equation}\label{eIsoPoint}
\Vert u \Vert_{L^{\frac{n}{n-1}}(X, \nu) } \leq I_n TV_m(u) \quad \hbox{for all} \ u \in BV^0_m(X,\nu).
\end{equation}
The constant  $I_n$ is the same as in \eqref{Isop1}.
\end{theorem}
\begin{proof} ($\Leftarrow$)  Given  $A \subset X$ with $0<\nu(A)<\nu(X)$,
 applying \eqref{eIsoPoint} to $\1_A$, we get
$$\nu(A)^{\frac{n-1}{n}} = \Vert \1_A \Vert_{L^{\frac{n}{n-1}}(X, \nu)} \leq I_n TV_m(\1_A) = I_n P_m(A).$$
($\Rightarrow$) Let us see that \eqref{Isop1} implies \eqref{eIsoPoint}. Since  $TV_m(|u|) \leq TV_m(  u )$, we may assume that $u \geq 0$ without loss of generality.

Suppose first that $n=1$ and let $u \in  BV^0_m(X,\nu)$ such that $u\geq 0$
and is not $\nu$-a.e. equal to $0$ (otherwise, \eqref{eIsoPoint} is trivially satisfied).  Note that, in this case, since $u$ is null outside of a $\nu$-measurable set $A$ with $\nu(A)<\nu(X)$, we have $\nu(E_t(u))<\nu(X)$ for $t>0$ and, moreover, by the definition of the $L^\infty(X,\nu)$-norm, $0<\nu(E_t(u))$ for $t<\Vert u \Vert_{L^\infty(X,\nu)}$. Then, by the coarea formula and \eqref{Isop1},  we have
$$TV_m(u) = \int_{0}^{+\infty} P_m(E_t(u))\, dt =  \int_{0}^{\Vert u \Vert_{L^\infty(X,\nu)}} P_m(E_t(u))\, dt\geq \int_{0}^{\Vert u \Vert_{L^\infty(X,\nu)}} \frac{1}{I_n} dt = \frac{1}{I_n} \Vert u \Vert_{L^\infty(X,\nu)}.$$

Therefore, we may suppose that $n >1$. Let $p:= \frac{n}{n-1}$. Again, by the coarea formula and \eqref{Isop1}, if $u \in BV^0_m(X,\nu)$, $u\geq 0$  and not identically $\nu$-null, we get
\begin{equation}\label{e2coaereapos}TV_m(u) = \int_{0}^{+\infty} P_m(E_t(u))\, dt \geq \int_{0}^{\Vert u \Vert_{L^\infty(X,\nu)}}  \frac{1}{I_n} \nu(E_t(u))^{\frac{1}{p}}\, dt,
\end{equation}
where $\Vert u \Vert_{L^\infty(X,\nu)}=+\infty$ if $u\notin  L^\infty(X,\nu)$.
On the other hand, since the function $\varphi(t):= \nu(E_t(u))^{\frac{1}{p}}$ is nonnegative and non-increasing, we have
$$p t^{p-1} \varphi(t)^p \leq p \left( \int_0^t \varphi(s) ds \right)^{p-1} \varphi(t) = \frac{d}{dt} \left[\left( \int_0^t \varphi(s) ds \right)^{p}\right].$$
Integrating  over $(0, t)$ and  letting $t \to \Vert u \Vert_{L^\infty(X,\nu)}$, we obtain
$$\int_0^{\Vert u \Vert_{L^\infty(X,\nu)}} p t^{p-1} \varphi(t)^p  \, dt \leq \left( \int_0^{\Vert u \Vert_{L^\infty(X,\nu)}} \varphi(t) dt \right)^{p},$$
that is,
\begin{equation}\label{e3coaereapos}
\int_0^{\Vert u \Vert_{L^\infty(X,\nu)}}p t^{p-1} \nu(E_t(u))  \, dt \leq \left( \int_0^{\Vert u \Vert_{L^\infty(X,\nu)}} \nu(E_t(u))^{\frac{1}{p}} dt \right)^{p}.
\end{equation}
Now, $$\Vert u \Vert^p_{L^p(X, \nu)} = \int_X u^p(x) d \nu(x) =  \int_X \left( \int_0^{u(x)} \frac{dt^p}{dt} dt \right) d\nu(x) $$ $$= \int_X \left(\int_0^{\Vert u \Vert_{L^\infty(X,\nu)}}  p t^{p-1} \1_{E_t(u)} dt \right) d\nu(x) = \int_0^{\Vert u \Vert_{L^\infty(X,\nu)}}  p t^{p-1} \nu(E_t(u)) dt.$$
Thus, by \eqref{e3coaereapos}, we get
\begin{equation}\label{e4coaereapos}
\Vert u \Vert_{L^{p}(X, \nu)} \leq \int_0^{\Vert u \Vert_{L^\infty(X,\nu)}}  \nu(E_t(u))^{\frac{1}{p}} dt.
\end{equation}
Finally, from \eqref{e2coaereapos} and \eqref{e4coaereapos}, we obtain \eqref{eIsoPoint}.
\end{proof}

 Note that, if we take  $\Psi_n (r):= \frac{1}{I_n} r^{-\frac{1}{n}}$, we can rewrite \eqref{Isop1} as
 \begin{equation}\label{NIsop1}
\nu(A) \Psi_n(\nu(A))  \leq P_m(A) \quad   \hbox{for all $A\subset X$ with }   0<\nu(A)<\nu(X).
\end{equation}
The next definition was given in \cite{CGL} for Riemannian manifolds.

\begin{definition} {
 Given a non-increasing  function $\Psi : ]0, \infty[ \rightarrow [0,\infty[$, we say that $[X,d,m,\nu]$ satisfies  a  {\it $\Psi$-isoperimetric inequality} if
\begin{equation}\label{Isop1gen} \nu(A) \Psi(\nu(A)) \leq  P_m(A) \quad \hbox{for all $A\subset X$ with }   0<\nu(A)<\nu(X).
\end{equation}
}
\end{definition}

\begin{example}

(1)  In \cite{TILLICH}  (see also the references therein) it is shown   that the lattice $\Z^n$ has isoperimetric dimension $n$ with constant $I_n= \frac{1}{2n}$,   and  that the complete graph $K_n$  satisfies a $\Psi$-isoperimetric inequality  with $\Psi (r) =  n - r $. In addition, it is also proved that the $n$-cube $Q_n$ satisfies a $\Psi$-isoperimetric inequality with $\Psi (r) = {\rm log}_2 (\frac{\nu(Q_n)}{r})$.

\noindent (2) In \cite{MRTLibro}, for $[\mathbb{R}^N,d,m^J]$, it is proved that
\begin{equation}\label{diffic}\Psi_{_{J,N}}(|A|)\le P_J(A)  \quad \hbox{for all } \ A \subset X  \ \hbox{with} \ |A| < +\infty,
\end{equation}
being
\begin{equation}\label{mi958}
\Psi_{_{J,N}}(r)=
 \int_{B_{\left(r/\omega_N\right)^\frac{1}{N}}}H^J_{\partial B_{\|x\|}}(x)dx
  =\int_0^{r}  H_{\partial B_{\left(s/\omega_N\right)^\frac{1}{N}}}^J (\left(s/\omega_N\right)^\frac{1}{N}, 0, \ldots , 0)ds,
\end{equation}
where $B_r$ is the ball of radius $r$ centered at $0$ and $H_{\partial B_r}^J$ is the $m^J$-mean curvature of $\partial B_r$  (see Subsection~\ref{def.curvatura}).
Therefore, $[\mathbb{R}^N,d,m^J, \mathcal{L}^N]$  satisfies a  $\Psi$-isoperimetric inequality, where $\Psi(r) = \frac{1}{r} \Psi_{_{J,N}}(r)$ is a decreasing function.

\end{example}

  The next result was proved in \cite{CGL} for Riemannian manifolds and in \cite{Coulhon} for graphs (see also \cite[Theorem 2]{TILLICH}).

\begin{proposition}\label{trivil} Given a non-increasing  function $\Psi : ]0, \infty[ \rightarrow [0,\infty[$, we have that  $[X,d,m,\nu]$ satisfies  a  $\Psi$-isoperimetric inequality if, and only if, the following inequality holds:
\begin{equation} \label{eq:16}
\Psi(\nu(A)) \Vert u \Vert_{L^1(X, \nu)} \leq TV_m(u)
 \end{equation}
for all $\nu$-measurable sets $A \subset X$ with $0<\nu(A)<\nu(X)$  and all $u \in L^1(X, \nu)$ with $u = 0 \ \hbox{in} \ X \setminus A.$

\end{proposition}
\begin{proof} Taking $u = \1_A$ in \eqref{eq:16}, we obtain that $[X,d,m,\nu]$ satisfies  a  $\Psi$-isoperimetric inequality. Conversely, since  $TV_m(|u|) \leq TV_m(  u )$, it is enough to prove \eqref{eq:16}  for $u \geq 0$. If $u\equiv 0$ in $X$ the result is trivial. Therefore, let $A$ be a $\nu$-measurable set with $0<\nu(A)<\nu(X)$ and $0 \leq u \in L^1(X, \nu)$ a non-$\nu$-null function with $u \equiv 0 \ \hbox{in} \ X \setminus A$. For $t>0$ we have that $E_t(u) \subset A$ and, therefore, $\nu(E_t(u)) \leq \nu(A)$, thus, since $\Psi$ is non-increasing, we have that $\Psi(\nu(E_t(u)) \geq \Psi(A)$. Therefore, by the coarea formula, we have
$$
\begin{array}{c}\displaystyle
TV_m(u) = \int_0^{+\infty} P_m(E_t(u)) dt = \int_0^{\Vert u \Vert_{L^\infty(X,\nu)}} P_m(E_t(u)) dt \geq \int_0^{\Vert u \Vert_{L^\infty(X,\nu)}}  \nu(E_t(u)) \Psi(\nu(E_t(u))) dt \\ \\ \displaystyle
\geq \Psi(\nu(A)) \int_0^{+\infty} \nu(E_t(u)) dt = \Psi(\nu(A)) \Vert u \Vert_{L^1(X, \nu)}.
\end{array}$$
\end{proof}

 As a consequence of Theorem \ref{IsoPoint} and Proposition  \ref{trivil}, we obtain the following result.

 \begin{corollary}
  The following assertions are equivalent:
 \item(i)
  $\Vert u \Vert_{L^{\frac{n}{n-1}}(X, \nu) } \leq I_n TV_m(u) \quad \forall    u \in BV^0_m(X,\nu).$\\
 \item(ii) $ \Vert u \Vert_{L^1(X, \nu)} \leq I_n \nu(A)^{\frac{1}{n}}TV_m(u)$  for all   $A \subset X$ with $0<\nu(A)<\nu(X)$   and   all $u \in L^1(X, \nu)$ with $u = 0$ in $X \setminus A.$
 \end{corollary}

Consider the Dirichlet  energy functional $\mathcal{H}_m : L^2(X, \nu) \rightarrow [0, + \infty]$  defined as
$$\mathcal{H}_m(u)= \left\{ \begin{array}{ll} \displaystyle\frac{1}{2} \int_{X \times X} (u(x) - u(y))^2 dm_x(y) d\nu(x) \quad &\hbox{ if $u\in L^2(X, \nu) \cap  L^1(X, \nu)$.} \\ \\ + \infty, \quad &\hbox{else}. \end{array}\right.$$

  The next result, in the context of Markov chains, was obtained by Varopoulos in \cite{V}.

\begin{theorem}\label{megustaa}  Let $n >2$. If the Sobolev inequality
\begin{equation}\label{eIsoPointN}
\Vert u \Vert_{L^{\frac{n}{n-1}}(X, \nu) } \leq I_n TV_m(u) \quad \hbox{for all} \ u \in BV^0_m(X,\nu)
\end{equation}
holds, then there exists $C_n >0$ such that
\begin{equation}\label{eIsoPointNN}
\Vert u \Vert_{L^{\frac{2n}{n-2}}(X, \nu)}^2 \leq C_n \mathcal{H}_m(u) \quad \hbox{for all} \ u \in BV^0_m(X,\nu)
\end{equation}
\end{theorem}
 \begin{proof} We can assume that $u \geq 0$. Let $p:= \frac{2(n-1)}{n-2}$. By \eqref{eIsoPointN}, we have
\begin{equation}\label{E1V}\Vert u \Vert_{\frac{2n}{n-2}}^p = \Vert u \Vert_{\frac{pn}{n-1}}^p = \Vert u^p \Vert_{\frac{n}{n-1}} \leq I_n TV_m(u^p).\end{equation}
On the other hand, since, for $a,b>0$,
 $$\vert b^p - a^p\vert \leq p(a^{p-1} + b^{p-1} ) \vert b - a\vert$$
 by the convexity of $|x|^p$, and having in mind the reversibility of $\nu$, we have
$$TV_m(u^p) \leq \frac12 \int_X \int_X p(u^{p-1}(x) + u^{p-1}(y) ) \vert u(y) - u(x)\vert dm_x(y) d\nu(x) $$ $$= p \int_X \int_X u^{p-1}(x) \vert u(y) - u(x)\vert dm_x(y) d\nu(x) $$ $$\leq p \left(\int_X \int_X u^{2(p-1)}(x)dm_x(y) d\nu(x) \right)^{\frac12} \left(\int_X \int_X  \vert u(y) - u(x)\vert^2 dm_x(y) d\nu(x)\right)^{\frac12} $$ $$= p \Vert u^{p-1} \Vert_{L^2(X,\nu)} \left( 2 \mathcal{H}_m(u)\right)^{\frac12}.$$
Then, by \eqref{E1V}, we get
\begin{equation}\label{d10s002}\Vert u \Vert_{\frac{2n}{n-2}}^p \leq p I_n \Vert u^{p-1} \Vert_{L^2(X,\nu)} \left( 2 \mathcal{H}_m(u)\right)^{\frac12}.
\end{equation}
Now,
$$\Vert u^{p-1} \Vert_{L^2(X,\nu)} = \left(\int_X u^{\frac{2n}{n-2}} d\nu\right)^{\frac12} = \Vert u \Vert_{\frac{2n}{n-2}}^{\frac{n}{n-2}},$$
thus, from~\eqref{d10s002},
$$  \Vert u \Vert_{\frac{2n}{n-2}}^{\frac{2(n-1)}{ n-2}} \leq \textstyle \frac{2(n-1)}{n-2} I_n \Vert u \Vert_{\frac{2n}{n-2}}^{\frac{n}{n-2}} \left( 2 \mathcal{H}_m(u)\right)^{\frac12},$$
and, therefore,
$$\Vert u \Vert_{\frac{2n}{n-2}}^2 \leq C_n \mathcal{H}_m(u)$$
where $C_n = \frac{8(n-1)^2}{(n-2)^2} I_n^2.$
 \end{proof}

  Following Theorem~\ref{IsoPoint} and Theorem~\ref{megustaa} we can also obtain a Sobolev inequality as a consequence of the isoperimetric dimensional inequality.
 \begin{corollary}\label{NIsoPointN}  Assume that $\nu(X) < \infty$. Let  $n >2$. If $[X,d,m,\nu]$ has   isoperimetric dimension $n$ then  there exists $C_n >0$ such that
\begin{equation}\label{eIsoPointNNcoro}
\Vert u \Vert_{L^{\frac{2n}{n-2}}(X, \nu)}^2 \leq C_n \mathcal{H}_m(u) \quad \hbox{for all} \ u \in  BV_m^0(X,\nu).
\end{equation}
 \end{corollary}

 Let us point out that an important consequence of this result is Theorem 5 in \cite{ChY}, which corresponds to Corollary \ref{NIsoPointN} for the particular case of finite weighted graphs.

\subsection{$m$-$TV$ versus $TV$ in Metric Measure Spaces}\label{sec25}

Let $(X, d, \nu)$ be a metric measure space and recall that, for functions in $L^1(X, \nu)$, Miranda introduced a local notion of total variation in \cite{Miranda1} (see also  \cite{ADiM}). To define this notion, first note that for a function $u : X \rightarrow \R$, its {\it slope} (or {\it local Lipschitz constant}) is defined as
 $$\vert \nabla u \vert(x) := \limsup_{y \to x} \frac{\vert u(y) - u(x)\vert}{d(x,y)}, \ \ x\in X,$$
 with the convention that  $\vert \nabla u \vert(x) = 0$ if $x$ is an isolated point.

A function $u \in L^1(X, \nu)$ is said to be a BV-function if there exists a sequence $(u_n)$ of locally Lipschitz functions converging to $u$ in $L^1(X, \nu)$ and such that
 $$\sup_{n \in \N} \int_X \vert \nabla u_n \vert d\nu(x) < \infty.$$
 We shall denote the space of all BV-functions by $BV(X,d, \nu)$. Let $u \in BV(X,d, \nu)$, the total variation of $u$ on an open set $A \subset X$ is defined as:
 $$\vert D u \vert_{\nu}(A):= \inf \left\{ \liminf_{n \to \infty} \int_A \vert \nabla u_n \vert(x) d \nu(x) \ : \ u_n \in  Lip_{loc}(X,\nu), \ u_n \to u \ \hbox{in} \ L^1(A, \nu) \right\}.$$
 A set $E \subset X$ is said to be of finite perimeter if $\1_E \in BV(X,d, \nu)$ and its perimeter is defined as
 \begin{equation}\label{forarb001}{\rm Per}_{\nu}(E):= \vert D \1_E \vert_{\nu}(X).
 \end{equation}
  We want to point out that in \cite{ADiM} the BV-functions are characterized using different notions of total variation.

 As aforementioned, the local classical BV-spaces and the nonlocal BV-spaces are of different nature although they represent analogous concepts in different settings. In this section we compare these spaces, showing that it is possible to relate the nonlocal concept to the local one after rescaling and taking limits.

 \begin{remark}\label{lipBV} Obviously, \begin{equation}\label{clarif}
 \vert D u \vert_\nu \leq \vert \nabla u \vert \, \nu \quad \hbox{if $u$ is locally Lipschitz}.
 \end{equation}
  Furthermore, there exist metric measures spaces in which the equality in this expression  does not hold (see \cite[ Remark 4.4]{APS}).
 \end{remark}

\begin{proposition} Let $[X,d,m]$ be a metric random walk  space with invariant and reversible  measure~$\nu$.
Let $u\in BV(X,d,\nu)$. Then $u\in BV(X,d,m_x)$ for $\nu$-a.e. $x\in X$ and
$$\int_X |Du|_{m_x}(X) d\nu(x)\leq |Du|_{\nu}(X).$$
\end{proposition}

\begin{proof}
Since $u \in BV(X,d, \nu)$,  there exists a sequence $\{ u_n \}_{n \in \NN} \subset  Lip_{\text{loc}}(X,\nu)$ such that
$$\lim_{n \to \infty} \Vert u_n - u \Vert_{L^1(X,\nu)} = 0 \quad \hbox{ and} \quad \lim_{n \to \infty} \int_{X} \vert \nabla u_n \vert(x)  d\nu(x) =   \vert Du \vert_{\nu}(X).$$
Now, using the invariance of $\nu$,
$$
\begin{array}{l}
\displaystyle
\int_X \Vert u_n - u \Vert_{L^1(X, m_x)} d\nu(x)=  \int_{X} \left(\int_{X} |u_n(y) - u(y)| dm_x(y) \right) d\nu(x)  \\[10pt]
\displaystyle \qquad =  \int_{X} |u_n(y) - u(y)| d\nu(y) =  \Vert u_n - u \Vert_{L^1(X,\nu)} \xrightarrow{n\to\infty} 0 .
\\[10pt]
\end{array}$$
Therefore, we may take a subsequence, which we still denote by $u_{n}$, such that $\lim_{n \to \infty}\Vert u_n - u \Vert_{L^1(X, m_x)}=0$ for $\nu$-a.e. $x\in X$.

Moreover, by Fatou's lemma and the invariance of $\nu$,
$$
\begin{array}{l}
\displaystyle
\int_X \left(\liminf_{n \to \infty} \int_{X} \vert \nabla u_n \vert(y)  dm_x(y)\right) d\nu(x) \leq \liminf_{n \to \infty}\int_X  \left(\int_{X} \vert \nabla u_n \vert(y)  dm_x(y)\right) d\nu(x)
\\[10pt]
\displaystyle \qquad =  \liminf_{n \to \infty}\int_X  \vert \nabla u_n \vert(y)  d\nu(x) = |Du|_{\nu}(X).
\\[10pt]
\end{array}$$
Consequently, $\liminf_{n \to \infty} \int_{X} \vert \nabla u_n \vert(y)  dm_x(y)<\infty$ and $\lim_{n\to\infty}u_n=u$ in $L^1(X,m_x)$ for $\nu$-a.e. $x\in X$, thus $u\in BV(X,d,m_x)$ for $\nu$-a.e. $x\in X$, and
\begin{equation}\label{mon291151}
\int_X |Du|_{m_x}(X) d\nu(x)\leq |Du|_{\nu}(X) \, .
\end{equation}
\end{proof}

It is  shown in \cite{MRTLibro} that, in the context of Example \ref{JJ}~(1), and assuming that $J$ satisfies
\begin{equation}\label{vi1338}M_J:=\int_{\R^N}J(z)|z|dz<+\infty,
\end{equation}
we have that
\begin{equation}\label{ee1}
TV_{m^J}(u) \leq  \frac{M_J}{2}|Du|_{\mathcal{L}^N}
\end{equation}
for every $u \in BV(\R^N)$.

In the next example we  see that  there exist  metric random walk spaces  in which it is not possible to obtain an inequality like \eqref{ee1}.
\begin{example} Let $G= (V(G), E(G))$ be a locally finite weighted discrete graph with weights $w_{x,y}$. For a fixed $x_0 \in V(G)$ the function $u = \1_{\{x_0\}}$ is a Lipschitz function and, since every vertex is isolated for the graph distance, $|\nabla u|\equiv 0$,  thus
$$|Du|_{\nu_G}(V(G))\leq \int |\nabla u| d\nu_G(x) =0 .$$
However, by Example \ref{exammplee}, we have
$$TV_{m^G}(u) = \frac{1}{2}  \sum_{x \in V(G))} \sum_{y \in V(G)} \vert u(x) - u(y) \vert w_{xy} = \sum_{x \in V(G)), x \not= x_0} w_{x_0 x} > 0. $$
\end{example}

Let $[\mathbb{R}^N,d,m^J]$ be the metric random walk space of Example \ref{JJ}~(1). Then, if $J$ is compactly supported and $u \in BV(\R^N)$ has compact support we have that (see \cite{Davila} and \cite{MRT1})
\begin{equation}\label{connvver1}
\lim_{\epsilon \downarrow 0} \frac{C_J}{\epsilon} TV_{m^{J_{\epsilon}}}(u) = \int_{\R^N} \vert Du \vert d\mathcal{L}^N,
\end{equation}
where $$ J_\epsilon(x):=\frac{1}{\epsilon^{N}}
J\left(\frac{x}{\epsilon}\right) \quad \hbox{and} \quad C_J = \frac{2}{\displaystyle \int_{\R^N} J(z) \vert z_N \vert dz}.$$
In particular, if we take $$J(x):= \frac{1}{\mathcal{L}^N(B(0,1))}\1_{B(0,1)}(x),$$ then
$$J_\epsilon(x) = \frac{1}{\mathcal{L}^N(B(0,\epsilon))}\1_{B(0,\epsilon)}(x).$$
Hence,
$$m^{{\mathcal{L}^N,\epsilon}}_x = m^{J_{\epsilon}}_x,$$
and, consequently, by \eqref{connvver1}, we have
\begin{equation}\label{2connvver1}
\lim_{\epsilon \downarrow 0} \frac{C_J}{\epsilon} TV_{m^{{\mathcal{L}^N,\epsilon}}}(u) = \int_{\R^N} \vert Du \vert d\mathcal{L}^N = \vert Du \vert_{\mathcal{L}^N} (\R^N).
\end{equation}
Therefore, it is natural to pose the following problem: Let $(X,d, \mu)$ be a metric measure space and let $m^{\mu,\epsilon}$ be the  $\epsilon$-step random walk associated to $\mu$, that is,
 $$m^{\mu,\epsilon}_x:= \frac{\mu \res B(x, \epsilon)}{\mu(B(x, \epsilon))}.$$
Are there metric measure spaces for which
\begin{equation}\label{3connvver1}
\lim_{\epsilon \downarrow 0} \frac{1}{\epsilon}TV_{m^{\mu,\epsilon}}(u) \approx \vert Du \vert_{\mu}(X) \quad  \hbox{for all}  \ u \in BV(X,d,\mu)?
\end{equation}

To give a positive answer to the previous question we recall the following concepts on a metric measure space $(X,d, \nu)$: The measure $\nu$ is said to be {\it doubling} if there exists  a constant $C_D \geq 1$ such that
$$0 < \nu(B(x,2r)) \leq C_D \nu(B(x,r)) < \infty \quad \forall \, x \in X, \ \hbox{and all} \ r >0.$$
A doubling measure $\nu$ has the following property. For every $x \in X$ and $0 < r \leq R < \infty$ if $y \in B(x,R)$ then
\begin{equation}\label{egoodi}
\frac{\nu(B(x,R))}{\nu(B(y,r))} \leq C \left( \frac{R}{r} \right)^{q_{\nu}},
\end{equation}
where $C$ is a positive constant depending only on $C_D$ and $q_{\nu} = \log_2 C_D$.

 On the other hand, the metric measure space $(X,d, \nu)$ is said to {\it support a $1$-Poincar\'{e} inequality} if there exist constants $c>0$ and $\lambda \geq1$ such that, for any $u \in {\rm Lip}(X,d)$, the inequality
$$\int_{B(x,r)} \vert u(y) - u_{B(x,r)} \vert d\nu(y) \leq c r \int_{ B(x,\lambda r)} \vert \nabla u \vert (y) d \nu(y)$$
holds, where
$$u_{B(x,r)}:= \frac{1}{\nu(B(x,r))}\int_{ B(x,r) } u(y) d\nu(y).$$

The following result is proved in \cite[Theorem 3.1]{MMS}.

\begin{theorem}[\cite{MMS}]\label{ThMMS} Let $(X,d, \nu)$ be a metric measure space with $\nu$ doubling and supporting a $1$-Poincar\'{e} inequality. Given $u \in L^1(X, \mu)$, we have that $u \in BV(X,d,\nu)$ if, and only if,
$$\liminf_{\epsilon \to 0^+} \frac{1}{\epsilon} \int_{\Delta_{\epsilon}} \frac{\vert u(y) - u(x) \vert}{\sqrt{\nu(B(x,\epsilon))} \sqrt{\nu(B(y,\epsilon))}} d\nu(y) d \nu(x) < \infty,$$
where $\Delta_\epsilon := \{ (x,y) \in X \times X \ : \ d(x,y) < \epsilon \}$.
Moreover, there is a constant $C \geq 1$, that depends only on $(X,d, \nu)$, such that
\begin{equation}\label{vertgood}
\frac{1}{C} \vert Du \vert_\nu(X) \leq \liminf_{\epsilon \to 0^+} \frac{1}{\epsilon} \int_{\Delta_{\epsilon}} \frac{\vert u(y) - u(x) \vert}{\sqrt{\nu(B(x,\epsilon))} \sqrt{\nu(B(y,\epsilon))}} d\nu(y) d \nu(x) \leq C \vert Du \vert_\nu(X).
\end{equation}
\end{theorem}

Now, by Fubini's Theorem, we have
\begin{equation}\label{egood0} \int_{\Delta_{\epsilon}} \frac{\vert u(y) - u(x) \vert}{\sqrt{\nu(B(x,\epsilon))} \sqrt{\nu(B(y,\epsilon))}} d\nu(y) d \nu(x) = \int_X \int_{B(x,\epsilon)} \frac{\vert u(y) - u(x) \vert}{\sqrt{\nu(B(x,\epsilon))} \sqrt{\nu(B(y,\epsilon))}} d\nu(y) d \nu(x).\end{equation}
On the other hand, by \eqref{egoodi}, there exists a constant $C_1 >0$, depending only on $C_D$, such that
\begin{equation}\label{egood2}
\frac{\nu(B(x,\epsilon))}{\nu(B(y,\epsilon))} \leq C_1 \ .
\end{equation}
By  \eqref{egood2}, we have
\begin{equation}\label{egood3}
\frac{1}{\sqrt{C_1}} \frac{1}{\nu(B(x,\epsilon))} \leq \frac{1}{\sqrt{\nu(B(x,\epsilon))}  \sqrt{\nu(B(y,\epsilon))}} \leq \sqrt{C_1} \frac{1}{\nu(B(x,\epsilon))} \quad \forall \, y \in B(x, \epsilon).
\end{equation}
Hence, from \eqref{egood0} and \eqref{egood3}, we get

\begin{equation}
\begin{array}{ll}
\displaystyle \frac{1}{\sqrt{C_1}}  \frac{1}{\epsilon}TV_{m^{\nu,\epsilon}}(u)&\displaystyle =\frac{1}{\sqrt{C_1}}  \frac{1}{\epsilon} \frac{1}{2} \int_X  \frac{1}{\nu(B(x,\epsilon))}  \int_{B(x,\epsilon)} \vert u(y) - u(x) \vert d\nu(y) d \nu(x) \\
&\displaystyle \leq \frac{1}{\epsilon} \frac{1}{2} \int_{\Delta_{\epsilon}} \frac{\vert u(y) - u(x) \vert}{\sqrt{\nu(B(x,\epsilon))} \sqrt{\nu(B(y,\epsilon))}} d\nu(y) d \nu(x) \\
& \displaystyle \leq \sqrt{C_1}  \frac{1}{\epsilon} \frac{1}{2} \int_X  \frac{1}{\nu(B(x,\epsilon))}  \int_{B(x,\epsilon)} \vert u(y) - u(x) \vert d\nu(y) d \nu(x) \\
&\displaystyle = \sqrt{C_1}  \frac{1}{\epsilon}TV_{m^{\nu,\epsilon}}(u).
\end{array}
\end{equation}
  Therefore, we can rewrite  Theorem \ref{ThMMS} as follows.

\begin{theorem}\label{Thequiv} Let $(X,d, \nu)$ be a metric measure space with doubling measure $\nu$ and supporting a $1$-Poincar\'{e} inequality. Given $u \in L^1(X, \nu)$, we have that $u \in BV(X,d,\nu)$ if, and only if,
$$\liminf_{\epsilon \to 0^+} \frac{1}{\epsilon}TV_{m^{\nu,\epsilon}}(u) < \infty.$$
Moreover, there is a constant $C \geq 1$, that depends only on $(X,d, \nu)$, such that
\begin{equation}\label{vertgood1}
\frac{1}{C} \vert Du \vert_\nu(X) \leq \liminf_{\epsilon \to 0^+} \frac{1}{\epsilon}TV_{m^{\nu,\epsilon}}(u) \leq C \vert Du \vert_\nu(X).
\end{equation}
\end{theorem}

\begin{remark}  Monti, in \cite{Monti}, defines
$$\Vert \nabla u\Vert_{L^1(X,\nu)}^-:= 2\liminf_{\epsilon \to 0^+} \frac{1}{\epsilon}TV_{m^{\nu,\epsilon}}(u),$$ and uses this
to prove  rearrangement theorems in the setting of metric measure spaces.
Moreover, he proposes $\Vert \nabla u\Vert_{L^1(X,\mu)}^-$ as  a possible definition of the
$L_1$-length of the gradient of functions in metric measure spaces.
\end{remark}

\section{The $1$-Laplacian and the Total Variation Flow in Metric Random Walk Spaces}\label{1Lpalce}

Let $[X,d,m]$ be a metric random walk  space with invariant and  reversible measure $\nu$.   Assume, as aforementioned, that $[X,d,m]$ is $m$-connected.

Given a function $u : X \rightarrow \R$ we define its nonlocal gradient $\nabla u: X \times X \rightarrow \R$ as
$$\nabla u (x,y):= u(y) - u(x) \quad \forall \, x,y \in X,$$
 which should not be confused with the {\it slope $\vert \nabla u \vert(x)$, $x\in X$,}  introduced in Section~\ref{sec25}.

For a function $\z : X \times X \rightarrow \R$, its {\it $m$-divergence} ${\rm div}_m \z : X \rightarrow \R$ is defined as
 $$({\rm div}_m \z)(x):= \frac12 \int_{X} (\z(x,y) - \z(y,x)) dm_x(y),$$
and, for $p \geq 1$, we define the space
$$X_m^p(X):= \left\{ \z \in L^\infty(X\times X, \nu \otimes m_x) \ : \ {\rm div}_m \z \in L^p(X,\nu) \right\}.$$

Let $u \in BV_m(X,\nu) \cap L^{p'}(X,\nu)$ and $\z \in X_m^p(X)$,  $1\le p\le \infty$, having in mind that $\nu$ is reversible, we have the following {\it Green's formula}:
\begin{equation}\label{Green}
\int_{X} u(x) ({\rm div}_m \z)(x) d\nu(x) = -\frac12 \int_{X \times X} \nabla u(x,y) \z(x,y)   d(\nu \otimes m_x)(x,y).
\end{equation}

In the next result we characterize $TV_m$ and the $m$-perimeter using the $m$-divergence operator.
Let us denote by $\hbox{sign}_0(r)$ the usual sign function and  by $\hbox{sign}(r)$ the multivalued sign function:
$$\begin{array}{cc}
{\rm sign}_0(r):=  \left\{ \begin{array}{lll} 1 \quad \quad &\hbox{if} \ \  r > 0, \\   \ 0 \quad \quad &\hbox{if} \   \ r = 0,\\ -1 \quad \quad &\hbox{if} \ \ r < 0; \end{array}\right.\quad&\quad{\rm sign}(r):=  \left\{ \begin{array}{lll} 1 \quad \quad &\hbox{if} \ \  r > 0, \\   \left[-1,1\right] \quad \quad &\hbox{if} \   \ r = 0,\\ -1 \quad \quad &\hbox{if} \ \ r < 0. \end{array}\right.
\end{array}
$$
\begin{proposition}\label{lema1409} Let $1\le p\le \infty$. For $u \in BV_m(X,\nu) \cap L^{p'}(X,\nu)$, we have
\begin{equation}\label{Form1}
TV_m(u) =   \sup \left\{ \int_{X} u(x) ({\rm div}_m \z)(x) d\nu(x)  \ : \ \z \in X_m^p(X), \ \Vert \z \Vert_{L^\infty(X\times X, \nu\otimes m_x)} \leq 1 \right\}.
\end{equation}

In particular, for any $\nu$-measurable set $E \subset X$, we have
\begin{equation}\label{Form2}
P_m(E) =   \sup \left\{ \int_{E} ({\rm div}_m \z)(x)  d\nu (x)  \ : \ \z \in  X_m^1(X), \  \Vert \z \Vert_{L^\infty(X\times X, \nu\otimes m_x)} \leq 1 \right\}.
\end{equation}
\end{proposition}

\begin{proof} Let $u \in BV_m(X,\nu) \cap L^{p'}(X,\nu)$. Given $\z \in X_m^p(X)$ with $\Vert \z \Vert_{L^\infty(X\times X, \nu\otimes m_x)} \leq 1$, applying Green's formula
\eqref{Green}, we have
$$
  \int_{X} u(x) ({\rm div}_m \z)(x) d\nu(x)  =-\frac{1}{2}    \int_{X\times X} \nabla u(x,y) \z(x,y) d(\nu \otimes m_x)(x,y) $$ $$\leq \frac{1}{2} \int_{X\times X} \vert u(y) - u(x)\vert dm_x(y)d\nu(x) = TV_m(u).$$
Therefore,
$$\sup \left\{ \int_{X} u(x) ({\rm div}_m \z)(x) d\nu(x)  \ : \ \z \in X_m^p(X), \ \Vert \z \Vert_{L^\infty(X\times X, \nu\otimes m_x)} \leq 1 \right\} \leq TV_m(u).$$
On the other hand, since $(X,d)$ is $\sigma$-finite, there exists a sequence of sets $K_1 \subset K_2 \subset \ldots \subset K_n \subset \ldots $ of $\nu$-finite measure, such that $X = \cup_{n=1}^\infty K_n$.  Then, if we define $\z_n(x,y):= {\rm sign}_0(u(y) - u(x))\1_{K_n \times K_n}(x,y)$, we have that $\z_n \in  X_m^p(X)$ with  $\Vert \z_n \Vert_{L^\infty(X\times X, \nu\otimes m_x)} \leq 1$ and
$$
\begin{array}{ll}
\displaystyle
TV_m(u) &\displaystyle= \frac{1}{2} \int_{X\times X} \vert u(y) - u(x)\vert d(\nu \otimes m_x)(x,y) = \lim_{n \to \infty} \frac{1}{2} \int_{K_n \times K_n} \vert u(y) - u(x)\vert d(\nu \otimes m_x)(x,y)\\[10pt]
&\displaystyle=\lim_{n \to \infty}  \frac{1}{2}    \int_{X \times X} \nabla u(x,y) \z_n(x,y) d(\nu \otimes m_x)(x,y)  = \lim_{n \to \infty} \int_{X} u(x) ({\rm div}_m (-\z_n))(x) d\nu(x)
\\[10pt]
&\displaystyle\leq \sup \left\{ \int_{X} u(x) ({\rm div}_m (\z ))(x) d\nu(x)  \ : \ \z \in X_m^p(X), \ \Vert \z \Vert_{L^\infty(X\times X, \nu\otimes m_x)} \leq 1 \right\}.
\end{array}
$$
\end{proof}

\begin{corollary}\label{semicont} $TV_m$ is lower semi-continuous with respect to the weak convergence in $L^2(X, \nu)$.
\end{corollary}
\begin{proof} If $u_n \rightharpoonup u$ weakly in $L^2(X, \nu)$ then, given $\z \in X_m^2(X)$ with $\Vert \z \Vert_{L^\infty(X\times X, \nu\otimes m_x)} \leq 1$, we have that
$$\int_X u(x) ({\rm div}_m \z)(x) d\nu(x) = \lim_{n \to \infty} \int_X u_n(x) ({\rm div}_m \z)(x) d\nu(x) \leq \liminf_{n \to \infty} TV_m(u_n)$$
by Proposition \ref{Form2}. Now, taking the supremum over $\z$ in this inequality, we get
$$TV_m(u) \leq \liminf_{n \to \infty} TV_m(u_n).$$
\end{proof}

Consider the formal nonlocal  evolution equation
\begin{equation}\label{exevol}
u_t(x,t) = \int_{X}  \frac{u(y,t) - u(x,t)}{\vert u(y,t) - u(x,t) \vert} dm_x(y), \quad x \in X, t \geq 0.
\end{equation}
In order to study the Cauchy problem associated to the previous equation, we will see  in Theorem~\ref{ExistUniq} that we can rewrite it as the gradient flow in   $L^2(X,\nu)$ of the functional $\mathcal{F}_m : L^2(X, \nu) \rightarrow ]-\infty, + \infty]$ defined by
$$\mathcal{F}_m(u):= \left\{ \begin{array}{ll} \displaystyle
TV_m(u)
 \quad &\hbox{if} \ u\in L^2(X,\nu)\cap BV_m(X,\nu), \\[10pt] + \infty \quad &\hbox{if } u\in L^2(X,\nu)\setminus BV_m(X,\nu), \end{array} \right.$$
which is convex and lower semi-continuous.
 Following the method used in \cite{ACMBook} we will characterize the subdifferential of the functional $\mathcal{F}_m$.

Given a functional $\Phi : L^2(X,\nu) \rightarrow [0, \infty]$, we define
$\widetilde {\Phi}: L^2(X,\nu) \rightarrow [0, \infty]$ as
\begin{equation}\label{rrt}\widetilde {\Phi}(v):= \sup \left\{ \frac{\displaystyle \int_{X} v(x) w(x) d\nu(x)}{\Phi(w)} \ : \ w \in L^2(X,\nu) \right\}\end{equation}
with the convention that $\frac{0}{0} = \frac{0}{\infty} = 0$.  Obviously, if $\Phi_1 \leq \Phi_2$, then $\widetilde {\Phi}_2 \leq \widetilde {\Phi}_1$.

\begin{theorem}\label{chsubd} Let $u \in  L^2(X,\nu)$ and $v \in L^2(X,\nu)$. The following assertions are equivalent:

\noindent (i) $v \in \partial \mathcal{F}_m (u)$;

\noindent
(ii) there   exists  $\z  \in X_m^2(X)$, $\Vert \z \Vert_{L^\infty(X\times X, \nu\otimes m_x)} \leq 1$ such that
\begin{equation}\label{eqq2}
   v = - {\rm div}_m \z
\end{equation}
and
\begin{equation}\label{eqq1}
\int_{X} u(x) v(x) d\nu(x) = \mathcal{F}_m (u);
\end{equation}

\noindent
(iii)  there exists $\z  \in X_m^2(X)$, $\Vert \z \Vert_{L^\infty(X\times X, \nu\otimes m_x)} \leq 1$ such that \eqref{eqq2} holds and
\begin{equation}\label{eqq3}
\mathcal{F}_m (u) = \frac12\int_{X \times X} \nabla u(x,y) \z(x,y) d(\nu \otimes m_x)(x,y);
\end{equation}

\noindent
(iv)  there exists $\hbox{\bf g}\in L^\infty(X\times X, \nu \otimes m_x)$ antisymmetric with $\Vert \hbox{\bf g} \Vert_{L^\infty(X \times X,\nu\otimes m_x)} \leq 1$
        such that
    \begin{equation}\label{1-lapla.var-ver}
-\int_{X}\g(x,y)\,dm_x(y)=  v(x) \quad \hbox{for }\nu-\mbox{a.e }x\in X,
\end{equation}
         and
     \begin{equation}\label{1-lapla.sign}-\int_{X} \int_{X}\hbox{\bf g}(x,y)dm_x(y)\,u(x)d\nu(x)=\mathcal{F}_m(u).
     \end{equation}

\noindent (v)
      there exists $\hbox{\bf g}\in L^\infty(X\times X, \nu \otimes m_x)$ antisymmetric with $\Vert \hbox{\bf g} \Vert_{L^\infty(X \times X,\nu\otimes m_x)} \leq 1$ verifying \eqref{1-lapla.var-ver} and
         \begin{equation}\label{1-lapla.sign2}\hbox{\bf g}(x,y) \in {\rm sign}(u(y) - u(x)) \quad \hbox{for }(\nu \otimes m_x)-a.e. \ (x,y) \in X \times X.
     \end{equation}
\end{theorem}
\begin{proof}
Since $\mathcal{F}_m$ is convex, lower semi-continuous and positive homogeneous of degree $1$, by \cite[Theorem 1.8]{ACMBook}, we have
\begin{equation}\label{saii}\partial \mathcal{F}_m (u) = \left\{ v \in L^2(X,\nu)  \ : \ \widetilde{\mathcal{F}_m}(v) \leq 1, \  \int_{X} u(x) v(x) d\nu(x) = \mathcal{F}_m (u)\right\}.\end{equation}

We define, for $v \in   L^2(X,\nu)$,
\begin{equation}\label{Form3}\Psi(v):= \inf \left\{ \Vert \z \Vert_{L^\infty(X\times X, \nu\otimes m_x)} \ : \ \z  \in X^2_m(X), \ v = - {\rm div}_m \z \right\}.
\end{equation}
Observe that $\Psi$ is convex,  lower semi-continuous  and positive homogeneous of degree $1$. Moreover, it is easy to see that, if $\Psi(v) < \infty$, the infimum in \eqref{Form3} is attained i.e., there exists some $\z  \in X_m^2(X)$ such that \ $v = - {\rm div}_m\z$ and $\Psi(v) = \Vert \z \Vert_{L^\infty(X\times X, \nu\otimes m_x)}.$

Let us see that $$\Psi =  \widetilde{\mathcal{F}_m}.$$
 We begin by proving that $\widetilde{\mathcal{F}_m}(v) \leq   \Psi(v)$.
If $\Psi (v) = +\infty$ then this assertion is trivial. Therefore, suppose that $\Psi (v) < +\infty$. Let $\z \in L^\infty(X \times X,\nu\otimes m_x)$ such that $v = - {\rm div}_m \z$. Then, for $w\in L^2(X,\nu)$, we have
$$\int_{X} w(x) v(x) d\nu(x) = \frac12\int_{X \times X}  \nabla w(x,y) \z(x,y)d(\nu \otimes m_x)(x,y) \leq \Vert \z \Vert_{L^\infty(X\times X, \nu\otimes m_x)}  {\mathcal{F}_m}(w).$$
Taking the supremum over $w$ we obtain that $\widetilde{\mathcal{F}_m}(v) \leq  \Vert \z \Vert_{L^\infty(X\times X, \nu\otimes m_x)}$. Now, taking the infimum over $\z$, we get  $\widetilde{\mathcal{F}_m}(v) \leq   \Psi(v)$.

To prove the opposite inequality  let us denote
$$D:= \{  {\rm div}_m \z  \ : \ \z \in X^2_m(X) \}.$$
Then, by \eqref{Form1}, we have that, for $v\in L^2(X, \nu)$,
$$
\begin{array}{rl}
\displaystyle
 \displaystyle\widetilde{\Psi}(v)\!\!\!\! &\displaystyle
 = \sup_{w \in L^2(X,\nu)} \frac{\displaystyle\int_{X} w(x) v(x) d\nu(x)}{\Psi(w)}   \geq \sup_{w \in D }  \frac{\displaystyle\int_{X} w(x) v(x) d\nu(x)}{\Psi(w)} \\ \\
& \displaystyle=  \sup_{\z \in  X^2_m(X)}  \frac{\displaystyle\int_{X} {\rm div}_m\z(x) v(x) d\nu(x)}{\Vert \z \Vert_{L^\infty(X\times X, \nu\otimes m_x)}}   =   \mathcal{F}_m(v).
\end{array}
$$
 Thus, $ \mathcal{F}_m \leq \widetilde{ \Psi}$, which implies, by \cite[Proposition 1.6]{ACMBook}, that $\Psi = \widetilde{\widetilde{\Psi}} \leq \widetilde{ \mathcal{F}_m}$. Therefore, $\Psi = \widetilde{\mathcal{F}_m}$, and, consequently, from \eqref{saii}, we get
$$
\begin{array}{l}
\displaystyle
\partial \mathcal{F}_m (u) = \left\{ v \in L^2(X,\nu)  \ : \   \Psi(v) \leq 1, \  \int_{X} u(x) v(x) d\nu(x) = \mathcal{F}_m(u)\right\} \\[10pt]
  \displaystyle
\phantom{\partial \mathcal{F}_m (u)}
   = \left\{ v \in L^2(X,\nu)  \ : \ \exists \z  \in X_m^2(X), \ v = - {\rm div}_m\z, \ \Vert \z \Vert_{L^\infty(X\times X, \nu\otimes m_x)}  \leq 1, \  \int_{X} u(x) v(x) d\nu(x) = \mathcal{F}_m(u)\right\},
\end{array}
$$
from where the equivalence between (i) and (ii) follows .

To prove the equivalence between (ii) and (iii) we only need to apply Green's formula \eqref{Green}.

 On the other hand, to see that (iii) implies (iv), it is enough to take $\hbox{\bf g}(x,y)=\frac12(\z(x,y)-\z(y,x))$. Moreover, to see that (iv) implies (ii), take $\z(x,y)= \hbox{\bf g}(x,y)$ (observe that, from \eqref{1-lapla.var-ver},  $- {\rm div}_m(\g) = v$, so $\g \in X^2_m(X)$). Finally, to see that  (iv) and (v) are equivalent, we need to show that \eqref{1-lapla.sign} and \eqref{1-lapla.sign2} are equivalent. Now, since $\hbox{\bf g}$ is antisymmetric with $\Vert \hbox{\bf g} \Vert_{L^\infty(X \times X,\nu\otimes m_x)} \leq 1$ and $\nu$ is reversible, we have
 $$- 2 \int_{X} \int_{X}\hbox{\bf g}(x,y)dm_x(y)\,u(x)d\nu(x) = \int_{X\times X}\hbox{\bf g}(x,y) (u(y)- u(x))d(\nu \otimes m_x)(x,y) ,$$
 from where the equivalence between \eqref{1-lapla.sign} and \eqref{1-lapla.sign2} follows.
\end{proof}

By Theorem~\ref{chsubd} and  following \cite[Theorem 7.5]{ElLibro}, the next result is easy to prove.

\begin{proposition}\label{11caract1}    $\partial \mathcal{F}_m$ is an {\rm m}-completely accretive operator in $L^2(X,\nu)$.
 \end{proposition}

\begin{definition}\label{def35}
We define in $L^2(X,\nu)$ the multivalued operator $\Delta^m_1$ by
\begin{center}$(u, v ) \in \Delta^m_1$ \ if, and only if, \ $-v \in \partial \mathcal{F}_m(u)$.
\end{center}
\noindent  As usual, we will write $v\in \Delta^m_1 u$ for $(u,v)\in \Delta^m_1$.
\end{definition}

 Chang in \cite{Chang1} and  Hein and B\"uhler in \cite{HB}   define  a similar operator in the particular case of finite graphs:

\begin{example}\label{1Laplagraph}  Let $[V(G), d_G, (m^G_x)]$  be the metric random walk  given in Example~\ref{JJ}~(3) with invariant measure $\nu_G$. By Theorem \ref{chsubd}, we have
\begin{align*}
(u, v ) \in \Delta^{m^G}_1 \iff &\hbox{there exists} \ \hbox{\bf g}\in L^\infty(V(G)\times V(G), \nu_G \otimes m^G_x) \ \hbox{ antisymmetric with} \\
&\Vert \hbox{\bf g} \Vert_{L^\infty(V(G)\times V(G), \nu_G \otimes m^G_x)}\leq 1 \ \hbox{such that }
 \frac{1}{d_x}\sum_{y \in V(G)}\g(x,y) w_{xy}=  v(x) \quad \forall \, x\in V(G),
\end{align*}
         and
         $$\hbox{\bf g}(x,y) \in {\rm sign}(u(y) - u(x)) \quad \hbox{for }(\nu_G \otimes m^G_x)-a.e. \ (x,y) \in V(G) \times V(G).$$

   The next example shows that the operator $\Delta^{m^G}_1$ is indeed multivalued. Let $V(G) = \{ a, b \}$ and $w_{aa} = w_{bb} = p$, $w_{ab} = w_{ba} = 1- p$, with $0 < p <1$. Then,
    \begin{align*}
    (u, v ) \in \Delta^{m^G}_1 \iff &\hbox{there exists} \ \hbox{\bf g}\in L^\infty(\{ a, b \}\times \{ a, b \}, \nu_G \otimes m^G_x)  \hbox{ antisymmetric with} \\
    &\Vert \hbox{\bf g} \Vert_{ L^\infty(\{ a, b \}\times \{ a, b \}, \nu_G \otimes m^G_x)} \leq 1\hbox{ such that }  \g(a,a) p + \g(a,b) (1-p)=  v(a),\\
     & \g(b,b) p + \g(b,a) (1-p)=  v(b) \end{align*}
   and
   $$\hbox{\bf g}(a,b) \in {\rm sign}(u(b) - u(a))  .$$

   Now, since $\hbox{\bf g}$ is antisymmetric, we get
         $$v(a) = \g(a,b) (1-p), \quad v(b) =  - \g(a,b) (1-p) \quad \hbox{and} \quad \hbox{\bf g}(a,b) \in {\rm sign}(u(b) - u(a)).$$

\end{example}

\begin{proposition}\label{intpartt}[Integration by parts] For  any $(u,v)\in \Delta_1^m$   it holds that
\begin{equation}\label{intpart}
- \int_X v w d \nu \leq TV_m(w) \qquad \hbox{for all } \  w \in BV_m(X,\nu)\cap L^2(X,\nu),
\end{equation}
and
 \begin{equation}\label{intpart2}
- \int_X v u d \nu = TV_m(u).
\end{equation}
\end{proposition}
\begin{proof} Since $-v \in \partial \mathcal{F}_m(u)$, given $w \in BV_m(X,\nu)$, we have that
$$-\int_X v w d \nu \leq \mathcal{F}_m(u+w) - \mathcal{F}_m(u) \leq \mathcal{F}_m(w),$$
so we get \eqref{intpart}.  On the other hand, \eqref{intpart2} is given in Theorem \ref{chsubd}.
\end{proof}

As a consequence of Theorem~\ref{chsubd}, Proposition~\ref{11caract1} and  on account of Theorem~\ref{teointronls}, we can give the following existence and uniqueness result for    the Cauchy problem
\begin{equation}\label{CaychyP}
\left\{ \begin{array}{ll} u_t - \Delta^m_1 u \ni 0 \quad &\hbox{in} \ (0,T) \times X
\\[8pt] u(0,x) = u_0 (x) \quad & \hbox{for } x \in X, \end{array}\right.
\end{equation}
which is a rewrite of the formal expression~\eqref{exevol}.

\begin{theorem}\label{ExistUniq}  For every $u_0 \in L^2( X,\nu)$ and any $T>0$, there exists a unique solution of the Cauchy problem~\eqref{CaychyP} in $(0,T)$ in the following sense: $u \in W^{1,1}(0,T; L^2(X,\nu))$, $u(0, \cdot) = u_0$  in $L^2(X,\nu)$, and, for almost all $t \in (0,T)$,
$$
u_t(t,\cdot) - \Delta^m_1 u(t) \ni 0.
$$
Moreover, we have the following contraction and maximum principle in any $L^q(X,\nu)$--space, $1\le q\le \infty$:
\begin{equation}\label{maxprin}\Vert (u(t)-v(t))^+\Vert_{L^q(X,\nu)}\le \Vert (u_0-v_0)^+\Vert_{L^q(X,\nu)}\quad \forall\, 0<t<T,
\end{equation}
for any pair of solutions, $u,\, v$, of problem~\eqref{CaychyP} with initial data $u_0,\, v_0$ respectively.
\end{theorem}

\begin{definition}\label{ladefdeTVM}{
Given $u_0 \in L^2(X, \nu)$, we denote by $e^{t \Delta^m_1}u_0$ the unique solution of problem \eqref{CaychyP}. We call the semigroup $\{e^{t\Delta^m_1} \}_{t \geq 0}$ in $L^2(X, \nu)$ the {\it Total Variational Flow in the metric random walk  space} $[X,d,m]$ with invariant and reversible measure $\nu$.
}\end{definition}

In the next result we give an important property of the total variational flow in  metric random walk  spaces.
\begin{proposition}\label{propert1}  The TVF satisfies the mass conservation property:
for $u_0 \in L^2(X, \nu)$,  $$\int_X e^{t\Delta^m_1 }u_0 d \nu = \int_X u_0 d \nu \quad \hbox{for any} \ t \geq 0.$$
\end{proposition}

\begin{proof}  By Proposition \ref{intpartt}, we have
$$- \frac{d}{dt} \int_X e^{t\Delta^m_1}u_0 d \nu
\leq TV_m(1) = 0, $$ and
$$  \frac{d}{dt} \int_X e^{t\Delta^m_1}u_0 d \nu
\leq TV_m(-1) = 0. $$
Hence,
$$\frac{d}{dt} \int_X e^{t\Delta^m_1}u_0 d \nu =0,$$
and, consequently,
$$\int_X e^{t\Delta^m_1}u_0 d \nu = \int_X u_0 d \nu \quad \hbox{for any} \ t \geq 0.$$
 \end{proof}

\section{Asymptotic Behaviour of the TVF and Poincar\'e type Inequalities}

Let $[X,d,m]$ be a metric random walk space with invariant and  reversible measure $\nu$.   Assume as always that $[X,d,m]$ is $m$-connected.

\begin{proposition}\label{prop41} For every initial data $u_0 \in L^2(X, \nu)$,
$$\lim_{t \to \infty} e^{t\Delta^m_1}u_0 = u_\infty \quad \hbox{in } L^2(X,\nu),$$
with $$u_\infty
\in \{ u \in L^2(X, \nu) \, : \, 0 \in \Delta_1^m (u) \}.$$
Moreover,
 if $\nu(X) < \infty$  then $$u_\infty = \frac{1}{\nu(X)} \int_X u_0(x) d\nu(x).$$
\end{proposition}

\begin{proof} Since $\mathcal{F}_m$  is a proper and lower semicontinuous function in $X$ attaining a minimum at the constant zero function and,  moreover, $\mathcal{F}_m$ is even, by \cite[Theorem 5]{Bruck}, we have   $$\lim_{t \to \infty} e^{t\Delta^m_1}u_0 = u_\infty \quad \hbox{in } L^2(X,\nu),$$
with $$u_\infty
\in \{ u \in L^2(X, \nu) \, : \, 0 \in \Delta_1^m (u) \}.$$
Now, since $0 \in \Delta_1^m  (u_\infty) $, we have that $TV_m(u_\infty) = 0$ thus, by Lemma \ref{lemita1}, if  $\nu(X)<\infty$ (then $\frac{1}{\nu(X)}\nu$ is ergodic) we get that $u_\infty$ is constant. Therefore, by Proposition \ref{propert1},
$$u_\infty = \frac{1}{\nu(X)} \int_X u_0(x) d\nu(x).$$
\end{proof}

Let us see that we can get   a rate of convergence of the total variational flow $(e^{t\Delta^m_1})_{t \geq 0}$ when a  {\it Poincar\'{e} type inequality} holds.

 From now on in this section  we will assume that $$\nu(X) < +\infty.$$
Hence, $\mathcal{F}_m(u) =
TV_m(u)$ for all $u\in L^2(X,\nu)$.

\begin{definition}{ We say that $[X,d,m,\nu]$ satisfies a {\it $(q,p)$-Poincar\'{e} inequality} ($p, q\in[1,+ \infty[$) if there exists a constant $c>0$ such that, for any $u \in L^q(X,\nu)$,
$$  \left\Vert  u \right\Vert_{L^p(X,\nu)}  \leq c\left(\left(\int_{X}\int_{X} |u(y)-u(x)|^q dm_x(y) d\nu(x) \right)^{\frac1q}+\left| \int_X u\,d\nu\right|\right),$$
or, equivalently  (by the triangle inequality for one direction and taking $\tilde{u}=u-\nu(u)$ for the other), there exists a $\lambda > 0$ such that
$$\lambda \left\Vert u - \nu(u) \right\Vert_{L^p(X, \nu)} \leq  \Vert \nabla u \Vert_{L^q(X \times X, \nu \otimes m_x) }  \quad \hbox{for all} \ u \in L^q(X,\nu),$$
where $\nu(u):= \frac{1}{\nu(X)} \int_X u(x) d \nu(x)$.

When $[X,d,m,\nu]$ satisfies a $(q,p)$-Poincar\'{e} inequality, we will denote
$$\lambda^{(q,p)}_{[X,d,m,\nu]} := \inf \left\{ \frac{\Vert \nabla u \Vert_{L^q(X \times X, \nu \otimes m_x) }}{\Vert u \Vert_{ L^p(X, \nu)}} \ : \ \Vert u \Vert_{L^p(X,\nu)} \not= 0, \ \int_X u(x) d \nu(x) = 0  \right\}.$$

When $[X,d,m,\nu]$ satisfies a $(1,p)$-Poincar\'{e} inequality, we will say that $[X,d,m,\nu]$ satisfies a $p$-Poincar\'{e} inequality and write
\begin{equation}\label{pararb001}\lambda^p_{[X,d,m,\nu]} := \lambda^{(1,p)}_{[X,d,m,\nu]} = \inf \left\{ \frac{TV_m(u)}{\Vert u \Vert_{L^p(X, \nu)}} \ : \ \Vert u \Vert_{L^p(X, \nu)}\not= 0, \ \int_X u(x) d \nu(x) = 0  \right\}.\end{equation}
}
\end{definition}

  The following result was proved in \cite[Theorem 7.11]{ElLibro} for the particular case of the metric random walk space $[\Omega, d, m^{J,\Omega}]$.

\begin{theorem}\label{Asimpt1}
If $[X,d,m,\nu]$ satisfies a $1$-Poincar\'{e} inequality, then, for any $u_0 \in L^2(X, \nu)$,
$$\left\Vert  e^{t\Delta^m_1}u_0 -  \nu(u_0) \right\Vert_{L^1(X, \nu)} \leq \frac{1}{2 \lambda^{1}_{[X,d,m,\nu]} } \frac{\Vert u_0 \Vert^2_{L^2(X, \nu)}}{t} \quad \hbox{for all} \ t >0.$$
\end{theorem}
\begin{proof} Since the semigroup $\{e^{t\Delta^m_1 } \ : \ t \geq 0 \}$ preserves the mass (Proposition \ref{propert1}), we have
$$v(t): =  e^{t\Delta^m_1}u_0 - \frac{1}{\nu(X)} \int_X e^{t\Delta^m_1}u_0 d \nu = e^{t\Delta^m_1}u_0 - \frac{1}{\nu(X)} \int_X u_0 d \nu.$$
 Furthermore, the complete accretivity of the operator $-\Delta^m_1$ (see Section~\ref{secacc}) implies that
$$\mathcal{L}(v):= \left\Vert v - \nu( u_0)\right\Vert_{L^1(X, \nu)}$$
is a Liapunov functional for the semigroup $\{e^{t\Delta^m_1 } \ : \ t \geq 0 \}$, which implies that
\begin{equation}\label{e1asint}
\Vert v(t) \Vert_{L^1(X, \nu)} \leq \Vert v(s) \Vert_{L^1(X, \nu)} \quad \hbox{if} \ t \geq s.
\end{equation}
 Now, by the Poincar\'{e} inequality we get
\begin{equation}\label{e2asint}\lambda^1_{[X,d,m,\nu]} \Vert v(s) \Vert_{L^1(X, \nu)} \leq TV_m(v(s))\end{equation}
and, by \eqref{e1asint} and \eqref{e2asint}, we obtain that
\begin{equation}\label{e3asint}
t \Vert v(t) \Vert_{L^1(X, \nu)} \leq \int_0^t \Vert v(s) \Vert_{L^1(X, \nu)} ds \leq \frac{1}{\lambda^{(1,1)}_{[X,d,m,\nu]} } \int_0^t TV_m(v(s)) ds.
\end{equation}
On the other hand, by integration by parts (Proposition \ref{intpartt}),
$$- \frac12 \frac{d}{dt} \Vert e^{t\Delta^m_1}u_0 \Vert^2_{L^2(X, \nu)} =  - \int_X e^{t\Delta^m_1}u_0 \frac{d}{dt}e^{t\Delta^m_1}u_0 d\nu = TV_m(e^{t\Delta^m_1}u_0),$$
and then
$$\frac12  \Vert e^{t\Delta^m_1}u_0 \Vert^2_{L^2(X, \nu)} - \frac12  \Vert u_0 \Vert^2_{L^2(X, \nu)} = - \int_0^t  TV_m(e^{s\Delta^m_1}u_0) ds = -\int_0^t TV_m(v(s)) ds,$$
which implies
$$\int_0^t TV_m(v(s)) ds \leq \frac12 \Vert u_0 \Vert^2_{L^2(X, \nu)}.$$
Hence, by \eqref{e3asint}
$$\Vert v(t) \Vert_{L^1(X, \nu)} \leq \frac{1}{2 \lambda^1_{[X,d,m,\nu]} } \frac{\Vert u_0 \Vert^2_{L^2(X, \nu)}}{t},$$
which concludes the proof.
\end{proof}

To obtain a family of metric random walk  spaces for which a $1$-Poincar\'{e} inequality holds, we need the following result.

\begin{lemma}\label{lemmaPoincgen}
Suppose that  $\nu$ is a probability measure (thus ergodic)
and
\begin{equation}\label{mie1720}  m_x\ll\nu\quad\hbox{for
all  }x\in X.
\end{equation}
    Let    $q\ge  1$.  Let
$\{u_n\}_n\subset L^q(X,\nu)$ be a bounded sequence in $L^1(X,\nu)$     satisfying
\begin{equation}\label{003}
\lim_n \int_{X}\int_{X}|u_n(y)-u_n(x)|^q dm_x(y)d\nu(x)= 0 .
\end{equation}
Then, there exists $\lambda\in \mathbb{R}$ such that
\begin{equation}\label{jue1213}
 u_n \to \lambda\quad\hbox{for } \nu-\hbox{a.e. } x\in X,
\end{equation}
\begin{equation}\label{conclmar1351}
\Vert u_n-\lambda\Vert_{L^q (X,m_x)}\to 0\quad\hbox{for } \nu-\hbox{a.e. } x\in X.
\end{equation}

\end{lemma}

\begin{proof}
Let $$F_n(x,y)=|u_n(y)-u_n(x)|$$ and
$$f_n(x)=\int_{X} |u_n(y)-u_n(x)|^q\, dm_x(y).$$ From
\eqref{003}, it follows that
$$f_n\to 0\quad\hbox{in } L^1(X,\nu).$$ Passing to a subsequence if necessary, we can assume that
\begin{equation}\label{005}
f_n(x)\to 0\quad\forall x\in X\setminus B_1,\quad \hbox{where } B_1\subset X \hbox{ is } \nu\hbox{-null}.
\end{equation}
On the other hand, by \eqref{003}, we also have that $$F_n\to 0\quad\hbox{in }
L^q(X\times X, \nu\otimes m_x).$$ Therefore, we can suppose that, up to a subsequence,
\begin{equation}\label{006}
F_n(x,y)\to 0\quad\forall (x,y)\in X^2\setminus
C,\quad \hbox{where } C\subset X\times X \hbox{ is } \nu\otimes m_x\hbox{-null}.
\end{equation}
Let $B_2\subset X$ be a $\nu$-null set satisfying that,
\begin{equation}\label{007}  \hbox{for all $x\in
X\setminus B_2$, the section $C_{x}  := \{ y \in X  :  (x,y) \in C \}$ of $C$ is
$m_x$-null.}\end{equation}
Finally,  set $B:=B_1\cup B_2.$

Fix  $x_0\in X\setminus B$. Up to a subsequence we have that $u_n(x_0)\to\lambda$ for some $\lambda\in[-\infty,+\infty]$, but then, by \eqref{006}, for every $y\in X\setminus C_{x_0}$ we also have that $u_n(y)\to\lambda$. However, since $m_{x_0}\ll\nu$ and $m_{x_0}( X \setminus C_{x_0})>0$, we have that $\nu(X\setminus C_{x_0})>0$; thus, if $A=\{x\in X: u_n(x)\to\lambda\}$ then $\nu(A)>0$.

 Let us see that
 \begin{equation}\label{dco2235}
 m_x(X\setminus A)=0 \quad\hbox{for all } x\in A\setminus B.
  \end{equation}
  Indeed, let $x\in A\setminus B$. Then, for $y\in X  \setminus C_x$, $u_n(y)\to \lambda$, thus $y\in A$; that is,  $X\setminus C_x\subset  A$, and, consequently, $m_x(A)=1$.
Now, since $m_x(B)=0$, we have
\begin{equation}\label{dco2235bisbis}
 m_x(X\setminus (A\setminus B))=0 \quad\hbox{for all } x\in A\setminus B.
  \end{equation}
 Therefore,  since  $\nu$ is ergodic, \eqref{dco2235bisbis}~implies that $1=\nu(A\setminus B)=\nu(A)$.

Consequently, we have obtained that $u_n$ converges $\nu$-a.e. in $X$ to $\lambda$:
$$
u_n(x)\to\lambda\quad \hbox{ for } x\in A,\ \nu(X\setminus A)=0.$$   Since $\Vert u_n\Vert_{L^1 (X,\nu)}$ is bounded, by Fatou's Lemma, we must have that $\lambda \in \R$.
On the other hand, by \eqref{005},
$$F_n(x,\cdot)\to 0\quad\hbox{in } L^{q}(X,m_x)\ ,$$
for every $x\in X\setminus B_1$. In other words, $\Vert u_n(\cdot)-u_n(x)\Vert_{L^q (X, m_x)}\to 0$. Thus
\begin{equation}\label{mar1351}
\Vert u_n-\lambda\Vert_{L^q (X,m_x)}\to 0\quad\hbox{for $\nu$-a.e. }x\in X.
\end{equation}
\end{proof}

\begin{theorem}\label{lab197}
 Suppose that $\nu$ is a probability measure
 and
\begin{equation}\label{jue1224} m_x\ll\nu\quad\hbox{for
all  }x\in X.
\end{equation}
Let (H1) and (H2) denote the following hypothesis.
\item{(H1)} Given a $\nu$-null set $B$, there exist $x_1,x_2,\ldots, x_N\in X\setminus B$, $\nu$-measurable sets $\Omega_1,\Omega_2,\ldots,\Omega_N\subset X$ and $\alpha>0$, such that $\displaystyle  X= \bigcup_{i=1}^N\Omega_i$  and  $\displaystyle\frac{dm_{x_i}}{d\nu}\ge \alpha>0$ on $\Omega_i$, $i=1,2,...,N$.

\item{(H2)} Let $ 1\le p<q$. Given a $\nu$-null set $B$, there exist $x_1,x_2,\ldots, x_N\in X\setminus B$ and $\nu$-measurable sets $\Omega_1,\Omega_2,\ldots,\Omega_N\subset X$, such that $\displaystyle  X= \bigcup_{i=1}^N\Omega_i$  and, for $g_i:=\displaystyle\frac{dm_{x_i}}{d\nu} $ on $\Omega_i$,     $\displaystyle g_i^{-\frac{p}{q-p}}\in L^{1}(\Omega_i,\nu)$, $i=1,2,...,N$.

Then, if (H1) holds, we have that $[X,d,m,\nu]$ satisfies a  $(p,p)$-Poincar\'{e} inequality  for every $p\ge 1$, and, if (H2) holds, then $[X,d,m,\nu]$  satisfies a $(q,p)$-Poincar\'{e} inequality.
\end{theorem}

\begin{proof}
Let $1\le p\le q$. We want to prove that there exists a constant $c>0$ such that
$$  \left\Vert  u \right\Vert_{L^p(X,\nu)}  \leq c\left(\left(\int_{X}\int_{X} |u(y)-u(x)|^q dm_x(y) d\nu(x) \right)^{\frac1q}+\left| \int_X u\,d\nu\right|\right) \, \, \, \hbox{for every $u \in L^q(X,\nu)$},$$
for any $p=q\ge1$ when assuming $(H1)$ and for the $1\le p<q$ appearing in $(H2)$ when this hypothesis is assumed.
Suppose that this inequality is not satisfied. Then, there exists a sequence $(u_n)_{n\in\N}\subset   L^q(X,\nu)$, with $\Vert u_n\Vert_{L^p (X,\nu)}=1$,    satisfying
\begin{equation}\label{rabs003}
\lim_n \int_{X}\int_{X}|u_n(y)-u_n(x)|^q dm_x(y)d\nu(x)= 0
\end{equation}
and
\begin{equation}\label{rabs004}
\lim_n\int_Xu_n\, d\nu= 0.
\end{equation}
Therefore, by Lemma~\ref{lemmaPoincgen}, there exist $\lambda\in \mathbb{R}$ and a $\nu$-null set $B\subset X$ such that
\begin{equation}\label{rabsconclmar1351}
u_n\to \lambda\ \hbox{ and }\  \Vert u_n-\lambda\Vert_{L^q (X,m_x)}\to 0\quad\hbox{for } x\in X\setminus B.
\end{equation}

 We will now prove, distinguishing the cases in which we assume hypothesis $(H1)$ or $(H2)$, that
\begin{equation}\label{vie1222}\Vert u_n-\lambda\Vert_{L^p (X,\nu)}\to 0.
\end{equation}

Suppose first that hypothesis $(H1)$ is satisfied. Then, there exist $x_1,x_2,\ldots, x_N\in X\setminus B$, $\nu$-measurable sets $\Omega_1,\Omega_2,\ldots,\Omega_N\subset X$ and $\alpha>0$, such that $\displaystyle  X= \bigcup_{i=1}^N\Omega_i$  and  $\displaystyle g_i:=\frac{dm_{x_i}}{d\nu}\ge \alpha>0$ on $\Omega_i$, $i=1,2,...,N$. Note that, in this case, $p=q$ in the previous computations. Now,
\begin{equation}\label{c2c2mi0946}
\begin{array}{l }
\displaystyle
\Vert u_n-\lambda\Vert^q_{L^q (\Omega_i,\nu)}=\int_{\Omega_i}|u_n(y)-\lambda|^qd\nu(y)
\\ \\
\displaystyle
\phantom{\Vert u_n-\lambda\Vert^q_{L^q (\Omega_i,\nu)}}
\le \frac{1}{\alpha}\int_{\Omega_i}|u_n(y)-\lambda|^qg_i(y)d\nu(y)=\frac{1}{\alpha}\int_{\Omega_i}|u_n(y)-\lambda|^qdm_{x_i}(y). \end{array}
\end{equation}
   Consequently, since $X= \bigcup_{i=1}^N\Omega_i$,
\begin{equation}\label{jue1821}
\Vert u_n-\lambda\Vert^q_{L^q (X,\nu)}\le \frac{1}{\alpha}\sum_{i=1}^N\Vert u_n-\lambda\Vert^q_{L^q (\Omega_i,m_{x_i})}.
\end{equation}
Therefore,
 \begin{equation}\label{c2mi0946}\Vert u_n-\lambda\Vert_{L^q (X,\nu)}\to 0.
\end{equation}

Suppose now that hypothesis $(H2)$ holds. Then, there exist $ 1\le p<q$, such that, given a $\nu$-null set $B$, there exist $x_1,x_2,\ldots, x_N\in X\setminus B$ and $\nu$-measurable sets $\Omega_1,\Omega_2,\ldots,\Omega_N\subset X$, such that $\displaystyle  X= \bigcup_{i=1}^N\Omega_i$  and, for $g_i:=\displaystyle\frac{dm_{x_i}}{d\nu}$ on $\Omega_i$,     $\displaystyle g_i^{-\frac{p}{q-p}} \in L^{1}(\Omega_i)$, $i=1,2,...,N$. Hence,
\begin{equation}\label{c2c2mi0946pq}
\begin{array}{l }
\displaystyle
\Vert u_n-\lambda\Vert^p_{L^p (\Omega_i,\nu)}=\int_{\Omega_i}|u_n(y)-\lambda|^pd\nu(y)
\\ \\
\displaystyle
\phantom{\Vert u_n-\lambda\Vert^p_{L^p (\Omega_i,\nu)}}
=\int_{\Omega_i}|u_n(y)-\lambda|^p
\frac{g_i(y)^{\frac{p}{q}}}{g_i(y)^{\frac{p}{q}}}d\nu(y)
\\ \\
\displaystyle
\phantom{\Vert u_n-\lambda\Vert^p_{L^p (\Omega_i,\nu)}}
\le
\left(\int_{\Omega_i}|u_n(y)-\lambda|^q g_i(y)d\nu(y)\right)^{\frac{p}{q}}
\left(\int_{\Omega_i}\frac{1}{g_i(y)^{\frac{p}{q-p}}}d\nu(y)\right)^{\frac{q-p}{q}}
\\ \\
\displaystyle
\phantom{\Vert u_n-\lambda\Vert^p_{L^p (\Omega_i,\nu)}}
=
\left(\int_{\Omega_i}|u_n(y)-\lambda|^q dm_{x_i}(y)\right)^{\frac{p}{q}}
\left(\int_{\Omega_i}\frac{1}{g_i(y)^{\frac{p}{q-p}}}d\nu(y)\right)^{\frac{q-p}{q}}
. \end{array}
\end{equation}
   Consequently, since $X= \bigcup_{i=1}^N\Omega_i$,
\begin{equation}\label{jue1821pq}
\Vert u_n-\lambda\Vert^p_{L^p (X,\nu)}
\le
\sum_{i=1}^{N} \Vert u_n-\lambda\Vert^p_{L^q (\Omega_i,m_{x_i})}
\left\Vert \frac{1}{g_i^{\frac{p}{q-p}}}\right\Vert^{\frac{q-p}{q}}_{L^{1} (\Omega_i,\nu)} .
\end{equation}
Therefore,
 \begin{equation}\label{c2mi0946pq}\Vert u_n-\lambda\Vert_{L^p (X,\nu)}\to 0,
\end{equation}
which concludes the proof of \eqref{vie1222} in both cases.

Now, since $\displaystyle \lim_n\int_X u_n\,d\nu=0$, by \eqref{vie1222} we get that $\lambda=0$, but this implies
 $$\Vert u_n \Vert_{L^p (X,\nu)}\to 0,$$
 which is a contradiction with  $||u_n||_p=1$, $n\in\N$, so we are done.
\end{proof}

On account of Theorem~\ref{Asimpt1}, we obtain the following result on the asymptotic behaviour of the TVF.

 \begin{corollary}\label{lab196}
Under the hypothesis of Theorem~\ref{lab197},
for any $u_0 \in L^2(X, \nu)$,
$$\left\Vert  e^{t\Delta^m_1}u_0 -  \nu(u_0) \right\Vert_{L^1(X, \nu)} \leq \frac{1}{2\lambda^1_{[X,d,m,\nu]} } \frac{\Vert u_0 \Vert^2_{L^2(X, \nu)}}{t} \quad \hbox{for all} \ t >0.$$
\end{corollary}

\begin{example}\label{lexam34N}  We give two examples of metric random walk spaces in which a $1$-Poincar\'{e} inequality does not hold.
\item(1)  A locally finite weighted discrete graph with infinitely many vertices: Let $[V(G),d_G,m^G]$ be the metric random walk space associated to the locally finite weighted discrete graph with vertex set $V(G):=\{x_3,x_4,x_5\ldots,x_n\ldots \}$ and weights:
$$w_{x_{3n},x_{3n+1}}=\frac{1}{n^3} , \  w_{x_{3n+1},x_{3n+2}}=\frac{1}{n^2}, \ w_{x_{3n+2},x_{3n+3}}=\frac{1}{n^3} , $$
for  $n\geq 1$, and $w_{x_i,x_j}=0$ otherwise (recall Example \ref{JJ}~(3)). Moreover, let
$$f_n(x):=\left\{ \begin{array}{ll} \displaystyle n^2 \quad &\hbox{if} \ \ x=x_{3n+1},x_{3n+2} \\ \\ 0 \quad &\hbox{else}. \end{array}\right.
$$
Note that $\nu_G(V)<+\infty$ (we avoid its normalization for simplicity). Now,
$$2TV_{m^G}(f_n)=\int_V\int_V \vert f_n(x)-f_n(y)\vert dm_x(y)d\nu_G(x)  $$
$$ =d_{x_{3n}}\int_V \vert f_n(x_{3n})-f_n(y)\vert dm_{x_{3n}}(y)+d_{x_{3n+1}}\int_V \vert f_n(x_{3n+1})-f_n(y) \vert dm_{x_{3n+1}}(y)$$
$$+d_{x_{3n+2}}\int_V \vert f_n(x_{3n+2})-f_n(y) \vert dm_{x_{3n+2}}(y)+d_{x_{3n+3}}\int_V \vert f_n(x_{3n+3})-f_n(y) \vert dm_{x_{3n+3}}(y)$$
$$=n^2\frac{1}{n^3}+n^2\frac{1}{n^3}+n^2\frac{1}{n^3}+n^2\frac{1}{n^3}=\frac{4}{n} .$$
However, we have
$$\int_V f_n(x)d\nu_G(x)=n^2 (d_{x_{3n+1}}+d_{x_{3n+2}})=2n^2 \left(\frac{1}{n^2}+\frac{1}{n^3}\right)=2\left(1+\frac{1}{n}\right),$$
thus
$$\nu_G(f_n)=\frac{2\left(1+\frac{1}{n}\right)}{\nu_G(V)}=O\left(1\right),$$
where we use the notation $$\varphi(n) = O(\psi(n)) \iff    \limsup_{n \to \infty}   \left|\frac{\varphi(n)}{\psi(n)}\right| = C \neq 0.$$
Therefore,
$$\vert f_n(x)-\nu_G(f_n)\vert =\left\{ \begin{array}{ll} \displaystyle O(n^2) \quad &\hbox{if} \ \ x=x_{3n+1},x_{3n+2}, \\ \\ O\left(1\right) \quad &\hbox{otherwise}. \end{array}\right.
$$
Finally,
$$\int_V \vert f_n(x)-\nu_G(f_n)\vert d\nu_G(x)=O\left(1\right)\sum_{x\neq x_{3n+1},x_{3n+2}}d_{x}+ O(n^2)(d_{x_{3n+1}}+d_{x_{3n+2}})$$
$$=O\left(1\right)+2 O(n^2)(\frac{1}{n^2}+\frac{1}{n^3})=O(1) .$$
Consequently,

$$\inf \left\{ \frac{TV_{m^G}(u)}{\Vert u - \nu_G(u) \Vert_{L^1(V(G),\nu_G)}} \ : \ u\in L^1(V,\nu_G) , \ \Vert u \Vert_{L^1(V(G),\nu_G)} \not= 0 \right\}  =0,$$
and a $1$-Poincar\'{e} inequality does not hold for this space.

 \item(2) The metric random walk  space $[\R, d, m^J]$, where $d$ is the Euclidean distance and $J(x)= \frac12 \1_{[-1,1]}$:
Define, for $n \in \N$,   $$u_n= \frac{1}{2^{n+1}} \1_{[2^n, 2^{n+1}]}  - \frac{1}{2^{n+1}}\1_{[-2^{n+1}, - 2^n]}.$$
Then  $\Vert u_n \Vert_1 := 1$, $\displaystyle \int_\R u_n(x) dx = 0$ and it is easy to see that, for $n$ large enough,
$$TV_{m^J} (u_n) = \frac{1}{2^{n+1}}.$$
 Therefore, $(m^J, \mathcal{L}^1)$ does not satisfy a $1$-Poincar\'{e} inequality.

\end{example}

Let us see that, when $[X,d,m,\nu]$ satisfies a $2$-Poincar\'{e} inequality,   the solution  of the Total Variational Flow reaches the steady state in finite time.

\begin{theorem}\label{oomegusta} Let $[X,d,m]$ be a metric random walk space with invariant and reversible measure $\nu$. If $[X,d,m,\nu]$ satisfies a $2$-Poincar\'{e} inequality  then, for any $u_0 \in L^2(X, \nu)$,
        \begin{equation}
        \Vert e^{t\Delta^m_1}u_0-\nu(u_0)\Vert_{L^2(X,\nu)}\le \left(\Vert u_0-\nu(u_0)\Vert_{L^2(X,\nu)}-\lambda^{2}_{[X,d,m,\nu]}t\right)^+\quad\hbox{ for all $t \geq 0$},
    \end{equation}
    where $\lambda^{2}_{[X,d,m,\nu]}$ is given in~\eqref{pararb001}. Consequently,
    $$e^{t\Delta^m_1}u_0=\nu(u_0)\qquad\forall\, \hbox{$t\ge \hat t:=\frac{\left\Vert u_0-\nu(u_0)\right\Vert_{L^2(X,\nu)}}{\lambda^{2}_{[X,d,m,\nu]}}$}.$$

  \end{theorem}

  \begin{proof}  Let $v(t):= u(t) - \nu(u_0)$, where   $u(t):= e^{t\Delta^m_1}u_0$. Since  $\Delta_1^m u(t)=\Delta_1^m \big(u(t) - \nu(u_0)\big)$, we have that
  $$\frac{d}{dt} v(t) \in  \Delta_1^m v(t).$$
    Note that $v(t)\in BV_m(X,\nu)$ for every $t>0$. Indeed, since $-\Delta_m^1=    \partial \mathcal{F}_m$ is a maximal monotone operator in $L^2(X,\nu)$, by \cite[Theorem 3.7]{Brezis} in the context of the Hilbert space $L^2(X,\nu)$, we have that $v(t)\in D(\Delta_m^1)\subset BV_m(X,\nu)$ for every $t>0$.

 Hence, for each $t>0$, by Theorem~\ref{chsubd}, there exists  $\hbox{\bf g}_t\in L^\infty(X\times X, \nu \otimes m_x)$ antisymmetric with $\Vert \hbox{\bf g}_t \Vert_{L^\infty(X \times X,\nu\otimes m_x)} \leq 1$
        such that
    \begin{equation}\label{21-lapla.var-ver}
\int_{X}\hbox{\bf g}_t(x,y)\,dm_x(y)=  \frac{d}{dt} v(t)(x) \quad \hbox{for } \nu-\mbox{a.e }x\in X,
\end{equation}
         and
     \begin{equation}\label{21-lapla.sign}-\int_{X} \int_{X}\hbox{\bf g}_t(x,y)dm_x(y)\,v(t)(x)d\nu(x)=\mathcal{F}_m(v(t)) =TV_m(u(t)).
     \end{equation}

     Then, multiplying \eqref{21-lapla.var-ver} by $v(t)$ and integrating over $X$ with respect to $\nu$,  having in mind \eqref{21-lapla.sign}, we get
     \begin{equation}\label{fij001}\frac12 \frac{d}{dt} \int_X v(t)^2 d \nu + TV_m(v(t))=0, \ \forall t>0.
     \end{equation}

  Now, the semigroup $\{e^{t\Delta^m_1 } \ : \ t \geq 0 \}$ preserves the mass (Proposition \ref{propert1}), so we have that $\nu(u(t)) = \nu(u_0)$ for all $t \geq 0$, and, since $[X,d,m,\nu]$ satisfies a $2$-Poincar\'{e} inequality, we have
  $$\lambda^{2}_{[X,d,m,\nu]} \Vert v(t) \Vert_{L^2(X,\nu)} \leq TV_m(v(t))\quad\hbox{ for all $t \geq 0$}.$$
  Therefore,
  $$\frac12\frac{d}{dt}  \Vert v(t) \Vert_{L^2(X,\nu)}^2 +   \lambda^{2}_{[X,d,m,\nu]} \Vert v(t) \Vert_{L^2(X,\nu)} \leq 0\quad\hbox{for all $t \geq 0$}.$$
  Now, integrating this ordinary differential inequation we get
  $$ \Vert v(t) \Vert_{L^2(X,\nu)} \leq  \left(\Vert v(0) \Vert_{L^2(X,\nu)} - \lambda^{2}_{[X,d,m,\nu]} t\right)^+\quad\hbox{ for all $t \geq 0$} ,$$
 that is,
  $$\Vert u(t)-\nu(u_0)\Vert_{L^2(X,\nu)}\le \left(\Vert u_0-\nu(u_0)\Vert_{L^2(X,\nu)}-\lambda^{2}_{[X,d,m,\nu]}t\right)^+\quad\hbox{ for all $t \geq 0$}.  $$
     \end{proof}

    We define the {\it extinction time} as
     $$T^*(u_0):= \inf \{ t >0 \ : \  e^{t\Delta^m_1}u_0 = \nu(u_0) \} , \ u_0\in L^2(X,\nu).$$
    Under the conditions of Theorem \ref{oomegusta}, we have
    $$ T^*(u_0) \leq \frac{\left\Vert u_0-\nu(u_0)\right\Vert_{L^2(X,\nu)}}{\lambda^{2}_{[X,d,m,\nu]}}, \ u_0\in L^2(X,\nu)  .$$

 To obtain a lower bound on the extinction time, we introduce the following norm which,  in the continuous setting, was introduced in \cite{Meyer}. Given a function $f \in L^2(X, \nu)$, we define
 \begin{equation}\label{1meyer1}
 \Vert f \Vert_{m,*}:= \sup \left\{ \int_X f(x) u(x) d\nu(x)   : u \in L^2(X, \nu)\cap BV_m(X,\nu), \ TV_m(u) \leq 1\right\}.
 \end{equation}
     \begin{theorem}\label{lowerbound}  Let  $u_0 \in L^2(X, \nu)$. If $T^*(u_0) < \infty$   then
 \begin{equation}
 T^*(u_0) \geq \Vert u_0 - \nu(u_0)\Vert_{m,*}.
 \end{equation}
 \end{theorem}
 \begin{proof} If $u(t):= e^{t\Delta^m_1}u_0$, we have
 $$u_0 - \nu(u_0) = - \int_0^{T^*(u_0)}   u'(t)dt.$$
Then, by integration by parts (Proposition \ref{intpartt}), we get
 $$\Vert u_0 - \nu(u_0) \Vert_{m,*} = \sup \left\{\int_X w (u_0 - \nu(u_0)) d \nu \ : \ TV_m(w) \leq 1  \right\} $$ $$= \sup \left\{ \int_X w \left(  \int_0^{T^*(u_0)} -  u'(t)dt\right) d \nu \ : \ TV_m(w) \leq 1  \right\}$$ $$= \sup \left\{ \int_0^{T^*(u_0)} \int_X -w    u'(t)dt  d \nu \ : \ TV_m(w) \leq 1  \right\}$$  $$\leq  \sup \left\{\int_0^{T^*(u_0)} TV_m(w) dt \ : \ TV_m(w) \leq 1  \right\} = T^*(u_0).$$
 \end{proof}

We will now see that we can get a $2$-Poincar\'{e} inequality for finite graphs.

\begin{theorem}\label{Poincgraf} Let $G = (V(G), E(G))$ be a  finite weighted connected discrete graph. Then, following the notation of Example \ref{JJ}~(3), $[V(G),d_G,m^G, \nu_G]$ satisfies a $2$-Poincar\'{e} inequality, that is,
\begin{equation}\label{POINGRA}
\lambda^{2}_{[V(G),d_G,m^G, \nu_G]} =\inf \left\{ \frac{TV_{m^G}(u)}{\Vert u \Vert_{L^2(V(G),\nu_G)}} \ : \ \Vert u \Vert_{L^2(V(G),\nu_G)} \not= 0, \ \int_V u(x) d \nu_G(x) = 0  \right\} >0.
\end{equation}
\end{theorem}
\begin{proof} Let  $V := V(G) = \{x_1, \ldots, x_m\}$ and suppose that \eqref{POINGRA} is false. Then, there exists a sequence $(u_n)_{n\in\N} \subset L^2(V, \nu_G)$ with $\Vert u_n \Vert_{L^2(V,\nu_G)} =1$ and $\int_V u_n(x) d \nu_G(x) =0$, $n\in\N$, such that
$$0 = \lim_{n \to \infty}TV_{m^G}(u_n) = \lim_{n \to \infty} \sum_{k=1}^m \sum_{y \sim x_k}w_{x_k y} \vert u_n(x_k) - u_n(y) \vert.$$
Hence,  $$ \lim_{n \to \infty} \vert u_n(x_k) - u_n(y) \vert =0 \quad \hbox{if} \ y \sim x_k, \quad \hbox{for any} \ k \in \{1, \ldots, m\}.$$
  Moreover, since $\Vert u_n \Vert_{L^2(V,\nu_G)} =1$, we have that, up to a subsequence,
$$\lim_{n \to \infty} u_n(x_k) = \lambda_k \in \R \quad \hbox{for}  \ \  k= 1, \ldots, m.$$
Now, since the graph is connected, we have that $\lambda = \lambda_k$ for $k= 1, \ldots, m$, thus
$$\lim_{n \to \infty} u_n(y) = \lambda \in \R \quad  \hbox{for all} \  y \in V.$$
However, by the Dominated Convergence Theorem, we get that $u_n \to \lambda$ in $L^2(V,\nu_G )$ and, therefore, since $\int_V u_n(x) d \nu_G(x) =0$, we have $\lambda =0$, which is a contradiction with $\Vert u_n \Vert_{L^2(V,\nu_G)} =1$.
\end{proof}

As a consequence of this last result and~Theorem~\ref{oomegusta}, we get:

\begin{theorem}\label{Poincgrafdos} Let $G = (V(G), E(G))$ be a  finite weighted connected discrete graph. Then,
\begin{equation}
   \Vert e^{t\Delta^{m^G}_1}u_0-\nu(u_0)\Vert_{L^2(V(G),\nu_G)} \le
    \lambda^{2}_{[V(G),d_G,m^G,\nu_G]}\left(\hat t-t\right)^+,
   \end{equation}
where $\hat t:=\frac{\left\Vert u_0-\nu(u_0)\right\Vert_{L^2(V(G),\nu_G)}}{\lambda^{2}_{[V(G),d_G,m^G,\nu_G]}}$. Consequently,
   \begin{equation}
   e^{t\Delta^{m^G}_1}u_0 = \nu(u_0) \quad \hbox{ for all} \  \  t\ge \hat t.
   \end{equation}
\end{theorem}

\section{$m$-Cheeger and $m$-Calibrable Sets}\label{aire001}

Let $[X,d,m]$ be a metric random walk  space with invariant and reversible measure~$\nu$.   Assume, as before, that $[X,d,m]$ is $m$-connected.

 Given a  set $\Omega \subset X$ with    $0 < \nu(\Omega) < \nu(X)$, we define its {\it $m$-Cheeger constant}  by
\begin{equation}\label{cheeg1}h_1^m(\Omega) := \inf \left\{ \frac{P_m(E)}{\nu(E)} \, : \, E \subset \Omega, \ E \ \hbox{$\nu$-measurable with }  \,   \nu( E)>0 \right\},\end{equation}
  where the notation $h_1^m(\Omega)$ is chosen together with the one that we will use for the classical Cheeger constant (see \eqref{ClassicCheeger}). In both of these, the subscript $1$ is there to further distinguish them from the upcoming notation $h_m(X)$ for the $m$-Cheeger constant of $X$ (see \eqref{1523m}).   Note that, by \eqref{secondf021}, we have that $h_1^m(\Omega) \leq 1$.

A $\nu$-measurable set $E  \subset \Omega$ achieving the infimum in \eqref{cheeg1} is said to be an {\it $m$-Cheeger set} of $\Omega$. Furthermore, we say that $\Omega$ is {\it $m$-calibrable} if it is an $m$-Cheeger set of itself, that is, if
$$h_1^m(\Omega) = \frac{P_m(\Omega)}{\nu(\Omega)}.$$

For ease of notation, we will denote
$$\lambda^m_\Omega:= \frac{P_m(\Omega)}{\nu(\Omega )},$$
for any $\nu$-measurable set $\Omega \subset X$ with $0<\nu(\Omega)<\nu(X)$.

\begin{remark}\label{observ1}

 (1) Let  $[\R^N, d, m^J]$ be   the metric random walk space given in Example~\ref{JJ}~(1) with invariant and reversible measure $\mathcal{L}^N$. Then, the concepts of $m$-Cheeger set and $m$-calibrable set coincide with the concepts of $J$-Cheeger set and $J$-calibrable set introduced in \cite{MRT1} (see also \cite{MRTLibro}).

\noindent (2) If $G = (V(G), E(G))$ is a locally finite weighted discrete graph without loops (i.e., $w_{xx} =0$ for all $x \in V$) and more than two vertices, then any subset consisting of two vertices is $m^G$-calibrable. Indeed, let $\Omega = \{x, y \}$, then, by \eqref{secondf021}, we have
$$\frac{P_{m^G}(\{x\})}{\nu_G(\{x\})}= 1 - \int_{\{x\}}\int_{\{x\}} dm^G_x(z) d\nu_G(z) = 1 \geq \frac{P_{m^G}(\Omega)}{\nu_G(\Omega)},$$
and, similarly,
$$\frac{P_{m^G}(\{y\})}{\nu_G(\{y\})} = 1 \geq \frac{P_{m^G}(\Omega)}{\nu_G(\Omega)}.$$
Therefore, $\Omega$ is $m^G$-calibrable.

\end{remark}

   In \cite{MRT1} it is proved that, for  the metric random walk space $[\R^N, d, m^J]$, each ball is a $J$-calibrable set. In the next example we will see that this result is not true in general.

\begin{example}
Let $V(G)=\{ x_1,x_2,\ldots , x_7 \} $ be a finite weighted discrete graph with the following weights:
$ w_{x_1,x_2}=2 , \ w_{x_2,x_3}=1 , \ w_{x_3,x_4}=2 , \ w_{x_4,x_5}=2  , \ w_{x_5,x_6}=1 , \ w_{x_6,x_7}=2 $ and $w_{x_i,x_j}=0$ otherwise.
Then, if  $E_1=B(x_4,\frac{5}{2})=\{ x_2, x_3, \ldots , x_6 \}$, by \eqref{lambdagraph} we  have
$$\frac{P_{m^G}(E_1)}{\nu_G(E_1)} =\frac{w_{x_1x_2}+w_{x_6x_7}}{d_{x_2}+d_{x_3}+d_{x_4}+d_{x_5} +d_{x_6}} =\frac{1}{4}.$$
But, taking $E_2=B(x_4,\frac{3}{2})=\{ x_3, x_4, x_5 \}\subset E_1$, we have
$$\frac{P_{m^G}(E_2)}{\nu_G(E_2)} =  \frac{w_{x_2x_3}+w_{x_5x_6}}{d_{x_3}+d_{x_4}+d_{x_5}} = \frac{1}{5}.$$
Consequently, the ball $B(x_4,\frac{5}{2})$ is not $m$-calibrable.

\end{example}

In the next Example we will see that there exist metric random walk spaces with sets that do not contain  $m$-Cheeger sets.
  \begin{example}\label{exx1} Consider the same graph of Example \ref{Nexx1}, that is,
 $V(G)=\{x_0,x_1,\ldots,x_n\ldots \}$ with the following weights:
$$w_{x_{2n}x_{2n+1}}=\frac{1}{2^n} , \quad w_{x_{2n+1}x_{2n+2}}=\frac{1}{3^n} \quad \hbox{for} \ n=0, 1, 2, \ldots , $$
and $w_{x_i,x_j}=0$ otherwise. If  $\Omega:=\{ x_1, x_2, x_3\ldots \}$,
 then
$\frac{P_{m^G}(D)}{\nu_G(D)} > 0$
 for every $D\subset \Omega$  with $\nu_G(D)>0$
 but, working as in Example \ref{Nexx1}, we get
 $h_1^m(\Omega)=0$.
     Therefore, $\Omega$ has no $m$-cheeger set.
   \end{example}

It is well known (see \cite{FK}) that the classical Cheeger constant
\begin{equation}\label{ClassicCheeger}
h_1(\Omega):= \inf \left\{ \frac{Per(E)}{\vert E \vert} \, : \, E\subset \Omega, \  \vert E \vert >0 \right\},
\end{equation}
for a bounded smooth domain $\Omega$, is an optimal Poincar\'{e} constant, namely, it coincides with the first eigenvalue of the $1$-Laplacian:
$$h_1(\Omega)=\Lambda_1(\Omega):= \inf \left\{ \frac{\displaystyle\int_\Omega \vert Du \vert +\displaystyle\int_{\partial \Omega} \vert u \vert d \mathcal{H}^{N-1}}{ \displaystyle\Vert u \Vert_{L^1(\Omega)}} \, : \, u \in BV(\Omega), \ \Vert u \Vert_{L^\infty(\Omega)} = 1 \right\}.$$
In order to get a nonlocal version of this result, we introduce the following constant.
For $\Omega \subset X$ with $0<\nu(\Omega)< \nu(X)$, we define
$$\begin{array}{l}\displaystyle
\Lambda_1^m(\Omega)= \inf \left\{ TV_m(u) \ : \ u \in  L^1(X,\nu), \ u= 0 \ \hbox{in} \ X \setminus \Omega, \ u \geq 0, \  \int_X u(x)  d\nu(x) = 1 \right\}
\\ \\
\displaystyle \phantom{\Lambda_1^m(\Omega)}
= \inf \left\{\frac{ TV_m (u)}{\displaystyle \int_X  u(x)  d\nu(x)} \ : \ u \in L^1(X,\nu), \ u= 0 \ \hbox{in} \ X \setminus \Omega,\ u \geq 0, \ u\not\equiv 0  \right\}.
\end{array}
$$

\begin{theorem}  Let $\Omega \subset X$ with $0 < \nu(\Omega) < \nu(X)$. Then,
\begin{equation}\label{first} h_1^m(\Omega) = \Lambda_1^m(\Omega).\end{equation}
\end{theorem}
\begin{proof}
Given a $\nu$-measurable subset $E \subset \Omega$  with $\nu(E )> 0$, we have
$$\frac{ TV_m(\1_E)}{\Vert \1_E \Vert_{L^1(X, \nu)}} = \frac{P_m(E)}{\nu (E)}.$$
Therefore, $\Lambda_1^m(\Omega) \leq h_1^m(\Omega)$.  For the opposite inequality we will follow an idea used in~\cite{FK}.
Given $u \in  L^1(X,\nu)$, with $u= 0$ in $X \setminus \Omega$, $u \geq 0$ and $u\not\equiv 0$, we have
$$TV_m(u) = \displaystyle\int_{0}^{+\infty} P_m(E_t(u))\, dt = \displaystyle\int_{0}^{  \Vert u\Vert_{L^\infty(X,\nu)}} \frac{ P_m(E_t(u))}{ \nu (E_t(u))} \nu (E_t(u))\, dt $$ $$\geq h_1^m(\Omega) \displaystyle\int_{0}^{+\infty} \nu (E_t(u))\, dt = h_1^m(\Omega) \int_X  u(x)  d\nu(x)$$
  where the first equality follows by the coarea formula \eqref{coaerea} and the last one by Cavalieri's  Principle. Taking the infimum over $u$ in the above expression we get $\Lambda_1^m(\Omega) \geq h_1^m(\Omega)$.
\end{proof}

Let us recall that, in the local case, a set $\Omega \subset \R^N$ is called {\it calibrable} if
$$\frac{\mbox{Per}(\Omega )}{\vert \Omega  \vert}  = \inf \left\{ \frac{\mbox{Per}(E)}{\vert E\vert} \ : \ E \subset \Omega , \ E \ \hbox{ with finite perimeter,} \ \vert E \vert > 0 \right\}.$$ The following characterization  of convex calibrable sets is proved in \cite{ACCh}.
\begin{theorem}\label{BCNth}(\cite{ACCh})
Given a bounded convex set $\Omega \subset \R^N$ of class $C^{1,1}$, the following assertions are equivalent:
\item(a) $\Omega $ is calibrable.
\item(b) $\1_\Omega $ satisfies $- \Delta_1 \1_\Omega   = \frac{\mbox{Per}(\Omega )}{\vert \Omega  \vert} \1_\Omega $, where $\Delta_1 u:= {\rm div} \left( \frac{Du}{\vert Du \vert}\right)$.
     \item(c) $\displaystyle (N-1)  \underset{x \in \partial \Omega }{\rm ess\, sup} H_{\partial\Omega} (x) \leq  \frac{\mbox{Per}(\Omega )}{\vert \Omega  \vert}.$
\end{theorem}

\begin{remark}\label{otrorem01}

(1) Let $\Omega \subset X$ be a $\nu$-measurable set with $0<\nu(\Omega)<\nu(X)$ and assume that there exists a constant $\lambda >0$ and a measurable function $\tau$ such that $\tau(x)=1$ for $x\in\Omega$ and
\begin{equation}\label{alares} - \lambda \tau \in \Delta_1^m \1_\Omega  \ \  \hbox{on} \ X.
\end{equation}
Then, by Theorem \ref{chsubd}, there exists $\hbox{\bf g}\in L^\infty(X\times X, \nu \otimes m_x)$ antisymmetric with $\Vert \hbox{\bf g} \Vert_{L^\infty(X \times X,\nu\otimes m_x)} \leq 1$ satisfying
$$
-\int_{X}\g(x,y)\,dm_x(y)= \lambda \tau(x) \quad \hbox{for }\nu-\mbox{a.e }x\in X
$$
and
  $$-\int_{X} \int_{X}\hbox{\bf g}(x,y)dm_x(y)\,\1_\Omega(x)d\nu(x)=\mathcal{F}_m(\1_\Omega) = P_m(\Omega).$$
 Then,
$$\begin{array}{l}\displaystyle\lambda \nu (\Omega) = \int_{X} \lambda \tau(x)\1_\Omega(x) d\nu(x)
\\[16pt]
\displaystyle
\phantom{\lambda \nu (\Omega)}
= -\int_{X} \left( \int_{X}\g(x,y)\,dm_x(y) \right)\1_\Omega(x) d\nu(x)
\\[16pt]
\displaystyle
\phantom{\lambda \nu (\Omega)}
= P_m(\Omega)
\end{array}$$
and, consequently,
$$\lambda =    \frac{P_m(\Omega)}{\nu(\Omega )}=:\lambda_\Omega^m.$$

\noindent (2) Let $\Omega \subset X$ be a $\nu$-measurable set with $0<\nu(\Omega)<\nu(X)$, and $\tau$ a $\nu$-measurable function with $\tau(x)=1$ for $x\in\Omega$. Then
\begin{equation}\label{tue1519} - \lambda^m_\Omega \tau \in \Delta_1^m \1_\Omega \, \quad \hbox{in} \ X \ \Longleftrightarrow \ - \lambda^m_\Omega \tau \in \Delta_1^m 0 \, \quad \hbox{in} \ X.
\end{equation}
Indeed, the left to right implication  follows from the fact that
$$\partial \mathcal{F}_m(u)\subset\partial \mathcal{F}_m(0),$$ and for the converse implication, we have that
there exists $\g \in L^\infty(X \times X,\nu\otimes m_x)$, $\g(x,y) = -\g(y,x)$ for almost all $(x,y) \in X \times X$,  $\Vert \g \Vert_{L^\infty(X \times X,\nu\otimes m_x)} \leq 1$, satisfying
$$
- \lambda^m_\Omega \tau(x)=\int_{X} \g(x,y)\,dm_x(y)   \quad \hbox{for }\nu-\mbox{a.e. }x\in X.
$$
 Now, multiplying by $\1_\Omega$, integrating over $X$ and applying integrating by parts we get
$$\begin{array}{c}\displaystyle  \lambda^m_\Omega \nu( \Omega)  = \lambda^m_\Omega \int_{X} \tau(x)\1_\Omega (x) d\nu(x) = - \int_{X} \int_{X}  \g(x,y) \1_\Omega (x)  dm_x(y)d\nu(x) \\[10pt]
\displaystyle \qquad = \frac{1}{2} \int_{X} \int_{X} \g(x,y) (\1_\Omega (y) - \1_\Omega (x) ) dm_x(y)d\nu(x)
\\[10pt]
\displaystyle \qquad \le \frac{1}{2} \int_{X} \int_{X} \left|\1_\Omega (y) - \1_\Omega (x)\right| dm_x(y)d\nu(x) = P_m(\Omega) .
\end{array}$$
 Then, since $P_m(\Omega)=\lambda^m_\Omega \nu( \Omega)$, the previous inequality is, in fact, an equality and, therefore, we get
$$\g(x,y) \in \hbox{sign}(\1_\Omega (y) - \1_\Omega (x) ) \quad \hbox{for }(\nu \otimes m_x)-\mbox{a.e. }(x,y) \in X \times X,$$
and, consequently,
$$-  \lambda^m_\Omega \tau \in \Delta^m_1 \1_\Omega  \quad \hbox{in} \ X.$$

\end{remark}

The next result is the nonlocal version of the fact that (a) is equivalent to (b) in Theorem~\ref{BCNth}.

 \begin{theorem}\label{falsa}    Let $\Omega \subset X$ be a $\nu$-measurable  set with $0<\nu(\Omega)<\nu(X)$.  Then, the following assertions are equivalent:
 \item{ (i)}       $\Omega$ is $m$-calibrable,
  \item{ (ii)} there exists  a $\nu$-measurable function $\tau$ equal to $1$ in $\Omega$    such that
\begin{equation}\label{22alares} - \lambda^m_\Omega \tau \in \Delta_1^m \1_\Omega \, \quad \hbox{in} \ X,
\end{equation}
 \item{ (iii)}
 \begin{equation}\label{22alaresvi2213} - \lambda^m_\Omega \tau^* \in \Delta_1^m \1_\Omega \, \quad \hbox{in} \ X,
\end{equation}
 for  $$\tau^*(x)=\left\{
\begin{array}{ll}
1   &\quad\hbox{if } x\in  \Omega ,\\[5pt]
\displaystyle - \frac{1}{\lambda_\Omega^m} m_x(\Omega)&\quad\hbox{if } x\in X\setminus\Omega.
\end{array}
\right.$$
\end{theorem}

\begin{proof}
Observe that, since we are assuming that the metric random walk space is $m$-connected, we have $P_m(\Omega)>0$ and, therefore, $\lambda_\Omega^m>0$.

 $(iii)\Rightarrow (ii)$ is trivial.

$(ii)\Rightarrow (i)$: Suppose   that
 there exists  a $\nu$-measurable function $\tau$ equal to $1$ in $\Omega$  satisfying~\eqref{22alares}. Hence,  there exists $\hbox{\bf g}\in L^\infty(X\times X, \nu \otimes m_x)$ antisymmetric with $\Vert \hbox{\bf g} \Vert_{L^\infty(X \times X,\nu\otimes m_x)} \leq 1$ satisfying
\begin{equation}\label{esla1}
-\int_{X}\g(x,y)\,dm_x(y)= \lambda_\Omega^m \tau(x) \quad \nu-\mbox{a.e. } x\in X
\end{equation}
and
  \begin{equation}\label{esla2}-\int_{X} \int_{X}\hbox{\bf g}(x,y)dm_x(y)\,\1_\Omega(x)d\nu(x)=P_m(\Omega).\end{equation}
Then, for $F\subset \Omega$ with $\nu(F) >0$, since   $\g$  antisymmetric, by using the reversibility of $\nu$ with respect to $m$,  we have
$$
\begin{array}{l}
\displaystyle
 \lambda^m_\Omega \nu( F)= \lambda^m_\Omega \int_{X} \tau(x) \1_F(x) d\nu(x) = - \int_{X} \int_{X}\g(x,y) \1_F(x) \,dm_x(y) d\nu(x) \\[10pt]
 \qquad \displaystyle = \frac{1}{2}\int_{X} \int_{X} \g(x,y) (\1_F(y) - \1_F(x))  \,dm_x(y) d\nu(x) \leq P_m(F).
 \end{array}
 $$
Therefore, $h_1^m(\Omega) = \lambda^m_\Omega$ and, consequently, $\Omega$ is $m$-calibrable.

 $(i)\Rightarrow (iii)$ Suppose   that $\Omega$ is $m$-calibrable.
Let
$$\tau^*(x)=\left\{
\begin{array}{ll}
1   &\quad\hbox{if } x\in  \Omega ,\\[5pt]
\displaystyle - \frac{1}{\lambda_\Omega^m} m_x(\Omega)&\quad\hbox{if } x\in X\setminus\Omega.
\end{array}
\right.$$
We claim that $-\lambda^m_\Omega \tau^* \in \Delta_1^m0$, that is,
\begin{equation}\label{fr2144}\lambda^m_\Omega \tau^* \in \partial  \mathcal{F}_m (0).
\end{equation}
Take   $ w \in  L^2(X,\nu)$ with $\mathcal{F}_m(w)<+\infty$.
Since
$$w(x) = \int_0^{+\infty} \1_{E_t(w)}(x) dt - \int_{-\infty}^0(1-\1_{E_t(w)})(x) dt,$$
and $$\displaystyle\int_X\tau^*(x)d\nu(x)=
 \int_\Omega1d\nu(x)-\frac{1}{\lambda_\Omega^m}\int_{X\setminus\Omega} m_x(\Omega)d\nu(x)=\nu(\Omega)-\frac{1}{\lambda_\Omega^m}P_m(\Omega)=
 0,$$
 we have
$$
  \int_{X} \lambda^m_\Omega \tau^*(x) w(x) d\nu(x)   =  \lambda^m_\Omega  \int_{-\infty}^{+\infty} \int_{X} \tau^*(x) \1_{E_t(w)}(x) d\nu(x) dt.
$$
Now, using that $\tau^*=1$ in $\Omega$ and $\Omega$ is $m$-calibrable  we have that
 $$\begin{array}{l}\displaystyle \lambda_\Omega^m\int_{-\infty}^{+\infty} \int_{X} \tau^*(x) \1_{E_t(w)}(x) d\nu(x)dt =  \lambda_\Omega^m \int_{-\infty}^{+\infty} \nu( E_t(w)\cap \Omega ) dt +\lambda_\Omega^m \int_{-\infty}^{+\infty} \int_{E_t(w)\setminus \Omega} \tau^*(x)  d\nu(x) dt \\ \\ \displaystyle \qquad \le  \int_{-\infty}^{+\infty}  P_m( E_t(w)\cap \Omega ) dt
 +\lambda_\Omega^m\int_{-\infty}^{+\infty} \int_{E_t(w)\setminus \Omega} \tau^*(x)  d\nu(x) dt.
  \end{array}$$
By Proposition~\ref{launion01} and the coarea formula given in Theorem \ref{coarea1}  we get
$$\begin{array}{l}\displaystyle  \int_{-\infty}^{+\infty}  P_m( E_t(w)\cap \Omega ) dt
  \\ [10pt] \qquad \displaystyle =  \int_{-\infty}^{+\infty}  P_m( E_t(w)\cap \Omega ) dt +\int_{-\infty}^{+\infty}  P_m(  E_t(w)\setminus \Omega ) dt -\int_{-\infty}^{+\infty} 2L_m(E_t(w)\setminus \Omega,E_t(w)\cap \Omega)dt\\[10pt] \qquad \qquad\qquad\qquad \displaystyle - \int_{-\infty}^{+\infty}  P_m(  E_t(w)\setminus \Omega ) dt+\int_{-\infty}^{+\infty}2L_m(E_t(w)\setminus \Omega,E_t(w)\cap \Omega)dt
 \\[10pt] \qquad \displaystyle =\int_{ -\infty}^{+\infty}  P_m(  E_t(w) ) dt - \int_{-\infty}^{+\infty}  P_m(  E_t(w)\setminus \Omega ) dt+\int_{-\infty}^{+\infty}2L_m(E_t(w)\setminus \Omega,E_t(w)\cap \Omega)dt
 \\[10pt] \qquad \displaystyle
 = \mathcal{F}_m(w) - \int_{-\infty}^{+\infty}  P_m(  E_t(w)\setminus \Omega ) dt+\int_{-\infty}^{+\infty}2L_m(E_t(w)\setminus \Omega,E_t(w)\cap \Omega)dt. \end{array}$$
 Hence, if we prove that $$I=- \int_{-\infty}^{+\infty}  P_m(  E_t(w)\setminus \Omega ) dt+\int_{-\infty}^{+\infty}2L_m(E_t(w)\setminus \Omega,E_t(w)\cap \Omega)dt +\lambda_\Omega^m\int_{-\infty}^{+\infty} \int_{E_t(w)\setminus \Omega} \tau^*(x)  d\nu(x) dt \leq 0,$$ we get
 \begin{equation}\label{casgoog}
\int_{X} \lambda^m_\Omega \tau^*(x) w(x) d\nu(x) \leq \mathcal{F}_m(w),
\end{equation}
which proves~\eqref{fr2144}.
 Now, since $$P_m(  E_t(w)\setminus \Omega )= L_m(E_t(w)\setminus \Omega, X \setminus (E_t(w)\setminus \Omega) ) = L_m(E_t(w)\setminus \Omega, (E_t(w)\cap \Omega)\overset{.}{\cup} (X \setminus E_t(w)) ),$$
 and $\tau^*(x)=- \frac{1}{\lambda_\Omega^m} m_x(\Omega)$ for $x\in X\setminus\Omega$,  we have
$$
\begin{array}{l}
\displaystyle
I   =- \int_{-\infty}^{+\infty}  L_m(E_t(w)\setminus \Omega,  X\setminus E_t(w)  ) dt + \int_{-\infty}^{+\infty}L_m(E_t(w)\setminus \Omega,E_t(w)\cap \Omega)dt \\ \\ \displaystyle
\qquad\qquad\qquad\qquad- \int_{-\infty}^{+\infty} \int_{E_t(w)\setminus \Omega} \int_{ \Omega}dm_x(y)  d\nu(x) dt
\\ \\ \displaystyle
\qquad \le
 \int_{-\infty}^{+\infty}L_m(E_t(w)\setminus \Omega,E_t(w)\cap \Omega)dt -\int_{-\infty}^{+\infty} L_m(E_t(w)\setminus \Omega,\Omega)dt
\le 0.
\end{array}
$$
Then, by \eqref{tue1519}, we have that  $$- \lambda^m_\Omega \tau^* \in \Delta_1^m \1_\Omega \, \quad \hbox{in} \ X,$$  and this concludes the proof.
 \end{proof}

Even though, in principle,  the $m$-calibrability of a set is a nonlocal concept, in the next result we will see that the $m$-calibrability of a set depends   only  on the set itself.

\begin{theorem}\label{trasen002}
Let $\Omega  \subset X$ be a $\nu$-measurable set   with $0<\nu(\Omega)<\nu(X)$.  Then, $\Omega$ is  $m$-calibrable if, and only if, there exists an antisymmetric  function $\g$ in $\Omega\times\Omega$ such that
    \begin{equation}\label{trasen001}
  -1\le \g(x,y)\le 1 \qquad \hbox{for $(\nu \otimes m_x)$-\mbox{a.e. }$(x,y) \in \Omega \times \Omega$},\end{equation}    and
\begin{equation}\label{trasen001post}\lambda_\Omega^m = -\int_{\Omega }\g(x,y)\,dm_x(y) + 1 - m_x(\Omega), \quad x \in \Omega.
  \end{equation}
\end{theorem}
Observe that, on account of~\eqref{secondf021},  \eqref{trasen001post} is equivalent to
\begin{equation}\label{ju1522tena01}
    m_x(\Omega)=\frac{1}{\nu(\Omega)}\int_\Omega m_z(\Omega)d\nu(z)-\int_{\Omega}  \g(x,y) \,dm_x(y)   \qquad\hbox{for $\nu$-a.e. }x\in\Omega  . \end{equation}

\begin{proof} By Theorem \ref{falsa}, we have that  $\Omega$ is $m$-calibrable if, and only if, there exists $\hbox{\bf g}\in L^\infty(X\times X, \nu \otimes m_x)$ antisymmetric, $\Vert \hbox{\bf g} \Vert_{L^\infty(X \times X,\nu\otimes m_x)} \leq 1$ with  $g(x,y) \in {\rm sign}(\1_\Omega (y) - \1_\Omega (x))$ for $\nu \otimes m_x$-a.e. $(x,y) \in X\times X$, satisfying
\begin{equation}\label{NNesla1}
\lambda_\Omega^m = -\int_{X}\g(x,y)\,dm_x(y)\quad \hbox{for } \nu-\mbox{a.e. } x\in \Omega
\end{equation}
and
\begin{equation}\label{NNesla1NN}
m_x(\Omega) = \int_{X}\g(x,y)\,dm_x(y) \quad \hbox{for }\nu-\mbox{a.e. } x\in X \setminus \Omega.
\end{equation}
Now,   having in mind that $g(x,y) = -1$ if $x \in \Omega$ and $y \in X \setminus \Omega$, we have that, for $x\in\Omega$,
$$\lambda_\Omega^m = 1 - \frac{1}{\nu(\Omega)}\int_\Omega m_x(\Omega) d \nu(x) = -\int_{X}\g(x,y)\,dm_x(y) = -\int_{\Omega }\g(x,y)\,dm_x(y) -\int_{X \setminus \Omega}\g(x,y)\,dm_x(y) $$ $$ = -\int_{\Omega }\g(x,y)\,dm_x(y) + m_x(X \setminus \Omega) = -\int_{\Omega }\g(x,y)\,dm_x(y) + 1 - m_x(\Omega).$$
Bringing together~\eqref{NNesla1} and these equalities we get~\eqref{trasen001} and~\eqref{trasen001post}.

  Let us now suppose that we have an antisymmetric  function $\g$ in $\Omega\times\Omega$ satisfying~\eqref{trasen001} and~\eqref{trasen001post}.
  To check that  $\Omega$ is $m$-calibrable we need to find $\tilde \g(x,y)\in\hbox{sign}\left(\1_\Omega(y)-\1_\Omega(x)\right)$
   antisymmetric such that
$$\left\{\begin{array}{l}
\displaystyle -\lambda_\Omega^m=\int_X\tilde\g(x,y)dm_x(y),\quad x\in\Omega,\\[14pt]
\displaystyle m_x(\Omega)=\int_X\tilde\g(x,y)dm_x(y),\quad x\in X\setminus \Omega,
\end{array}\right.
$$
which is equivalent to
$$\left\{\begin{array}{l}
\displaystyle -\lambda_\Omega^m=\int_\Omega\tilde\g(x,y)dm_x(y)-m_x(X\setminus\Omega), \quad x\in\Omega,\\[14pt]
\displaystyle m_x(\Omega)=\int_{X\setminus\Omega}\tilde\g(x,y)dm_x(y)+m_x(\Omega),\quad x\in X\setminus \Omega,
\end{array}\right.
$$
since, necessarily, $\tilde \g(x,y)=-1$ for $x\in \Omega$ and $y\in X\setminus\Omega$, and $\tilde \g(x,y)=1$ for $x\in X\setminus\Omega$  and $y\in \Omega$. Now, the second equality in this system is satisfied if we take $\tilde\g(x,y)=0$ for $x,y\in X\setminus\Omega$, and the first one is equivalent to~\eqref{ju1522tena01} if we take $\tilde\g(x,y)=\g(x,y)$ for $x,y\in\Omega$.
 \end{proof}

 Set
\begin{equation}\label{omegam}
\Omega_m:=\Omega\cup\partial_m\Omega
\end{equation}
 where $$\partial_m\Omega=\{ x\in X\setminus \Omega \ : \ m_x(\Omega)>0\}.$$

\begin{corollary}\label{corolarioomegamcal}
A $\nu$-measurable set $\Omega\subset X$ is $m$-calibrable if, and only if, it is $m^{\Omega_m}$-calibrable as a subset of $[\Omega_m,d,m^{\Omega_m}]$ with reversible measure~$\nu\res\Omega_m$  (see Example~\ref{JJ}~(5)).
\end{corollary}

\begin{remark}\label{mon1751}
 (1) Let $\Omega \subset X$ be a $\nu$-measurable set   with $0<\nu(\Omega)<\nu(X)$.
Observe that, as we have proved,
 \begin{equation}\label{e1Cali}
 \Omega \ \hbox{is   $m$-calibrable} \ \Longleftrightarrow  \ -\lambda_\Omega^m\,\1_\Omega +m_{_{(.)}}(\Omega)\,\1_{X\setminus\Omega} \in \Delta_1^m\1_\Omega \, .
 \end{equation}

\item{  (2)} Let $\Omega \subset X$ be a $\nu$-measurable set. If
\begin{equation}\label{tu13002}-\lambda_\Omega^m\,\1_\Omega +h\,\1_{X\setminus\Omega} \in \Delta_1^m\1_\Omega
\end{equation}
for some $\nu$-measurable function $h$, then
there exists $\hbox{\bf g}\in L^\infty(X\times X, \nu \otimes m_x)$ antisymmetric with \\ $\Vert \hbox{\bf g} \Vert_{L^\infty(X \times X,\nu\otimes m_x)} \leq 1$ satisfying
$$\g(x,y)\in {\rm sign}(\1_{\Omega}(y) -  \1_{\Omega}(x)) \quad (\nu \otimes m_x)-a.e. \ (x,y) \in X \times X $$
and
$$
-\lambda_\Omega^m \,\1_\Omega(x) +h(x)\,\1_{X\setminus\Omega}(x)=\int_{X}\g(x,y)\,dm_x(y)  \quad \nu-\mbox{a.e }x\in X.
$$
 Hence, if $$
\hbox{$\g$ is $\nu\otimes m_x$-integrable}$$
we have that
$$\int_{X\setminus\Omega}h(x)d\nu(x)=P_m(\Omega).$$
Indeed, from~\eqref{tu13002}, for $x\in X\setminus\Omega$,
 $$
 \begin{array}{l}
 \displaystyle
 h(x)=\int_{X}\g(x,y)\,dm_x(y)=\int_{\Omega}\g(x,y)\,dm_x(y)+\int_{X\setminus\Omega}\g(x,y)\,dm_x(y)
 \\ \\
 \displaystyle
 \qquad= \int_{\Omega}\,dm_x(y)+\int_{X\setminus\Omega}\g(x,y)\,dm_x(y)
 \\ \\
 \displaystyle
 \qquad=  m_x(\Omega)+\int_{X\setminus\Omega}\g(x,y)\,dm_x(y).
 \end{array}
 $$
 Hence, integrating over $X\setminus\Omega$ with respect to $\nu$, we get
 $$\int_{X\setminus\Omega}h(x)d\nu(x)=P_m(\Omega)+\int_{X\setminus\Omega}\int_{X\setminus\Omega}\g(x,y)\,dm_x(y)d\nu(x).$$
Moreover,  since $\g$ is antisymmetric and $\nu\otimes m_x$-integrable, we have
$$\int_{X\setminus\Omega}\int_{X\setminus\Omega}\g(x,y)\,dm_x(y)d\nu(x)
=\int_{(X\setminus\Omega)\times (X\setminus\Omega)}\g(x,y)\,d(\nu\otimes m_x)(x,y)=0,
 $$
 and, consequently, we get
\begin{equation}\label{hnocero}
  \int_{X\setminus\Omega}h(x)d\nu(x)=P_m(\Omega).
\end{equation}

 As a consequence of \eqref{hnocero},  if $\nu(X)<\infty$,   since the metric random walk space is $m$-connected, the relation \begin{equation}\label{ju1522}-\lambda_\Omega^m\,\1_\Omega \in \Delta_1^m\1_\Omega \quad\hbox{in }X  \end{equation}
does not hold true for any $\nu$-measurable set $\Omega$ with $0<\nu(\Omega)<\nu(X)$ (recall that, for these $\Omega$, $P_m(\Omega)>0$ by \cite[Theorem 2.21\&2.24]{MST0} thus $h$ is non--null by \eqref{hnocero}).
  Now, if $\nu(X)= +\infty$, then \eqref{ju1522} may be satisfied, as shown in the next example.
\end{remark}

\begin{example} Consider the metric random walk space
$[\R, d, m^J]$ with $\nu=\mathcal{L}^1$ and $J=\frac12\1_{[-1,1]}$. Let us see that
$$-\lambda_{[-1,1]}^{m^J}\1_{[-1,1]}\in \Delta_1^{m^J}\1_{[-1,1]},$$
where $\lambda_{[-1,1]}^{m^J}=\frac14$.
Indeed, take $\g(x,y)$ to be antisymmetric and defined as follows for $y<x$:
$$\g(x,y)=-\1_{\{y<x<y+1<0\}}(x,y)-\frac12\1_{\{-1<y < x< 0\}}(x,y)+\frac12\1_{\{0<y< x< 1\}}(x,y)+\1_{\{0<x-1<y<x\}}(x,y).$$
Then,
$\g\in L^\infty(\R\times \R, \nu \otimes m^J_x)$, $\Vert \g \Vert_{L^\infty(\R \times \R,\nu\otimes m^J_x)} \leq 1$,
$$\g(x,y)\in {\rm sign}(\1_{[-1,1]}(y) - \1_{[-1,1]}(x)) \quad \hbox{for }(\nu \otimes m^J_x)-a.e. \ (x,y) \in \R \times \R ,$$
and
$$
-\frac14 \1_{[-1,1]}(x) =\int_{\R}\g(x,y)\,dm^J_x(y)  \quad \hbox{for }\nu-\mbox{a.e }x\in \R.
$$
Note that $\g$ is not $\nu \otimes m^J_x$ integrable.

\end{example}

\begin{remark} As a consequence of Theorem \ref{BCNth}, it holds  that (see \cite[Introduction]{ACCh} or \cite[Section 4.4]{ACMBook}) a bounded convex set $\Omega \subset \R^N$ is calibrable if, and only if, $u(t,x) = \left( 1 - \frac{\mbox{Per}(\Omega )}{\vert \Omega  \vert} t\right)^+ \1_\Omega(x)$ is a solution of the Cauchy problem
\begin{equation}\label{cauchyCali}
\left\{ \begin{array}{ll} u_t -\Delta_1 u\ni 0 \quad \hbox{in} \ (0, \infty) \times \R^N, \\[8pt] u(0) = \1_\Omega.\end{array}\right.
\end{equation}
That is, a calibrable set $\Omega$ is that for which the gradient descent flow associated to the total variation tends to decrease linearly the height of $\1_\Omega$ without distortion of its boundary.

Now, as a consequence of \eqref{e1Cali}, we   can obtain a similar result in our context if we introduce an abortion term in  the corresponding Cauchy problem. The appearance of this term is due to the nonlocality of the diffusion considered.  Let $\Omega \subset X$ be a $\nu$-measurable set   with $0<\nu(\Omega)<\nu(X)$, then $\Omega$ is $m$-calibrable if, and only if, $u(t,x) = \left( 1 - \lambda^m_\Omega t\right)^+ \1_\Omega(x)$ is a solution of
\begin{equation}\label{CaychyPCali}
\left\{ \begin{array}{ll} u_t(t,x) -\Delta^m_1 u(t,x)\ni - m_{x}(\Omega)\,\1_{X\setminus\Omega}(x) \1_{[0,1/\lambda_\Omega^m)}(t)      &\hbox{in} \ (0,\infty) \times X,
\\[12pt] u(0,x) = \1_\Omega (x),  &x \in X. \end{array}\right.
\end{equation}
  Note that the only if direction follows by the uniqueness of the solution.
\end{remark}

The following result relates the $m$-calibrability with the $m$-mean curvature, this is the nonlocal version of one of the implications in the equivalence between
(a) and (c) in Theorem~\ref{BCNth}.

 \begin{proposition}\label{falsa.22} Let $\Omega \subset X$ be a $\nu$-measurable set with $0<\nu(\Omega)<\nu(X) $.
Then,
\begin{equation}\label{Nlacondcoj02}  \displaystyle \Omega \ \hbox{  $m$-calibrable} \ \Rightarrow \ \frac{1}{\nu(\Omega)}\int_\Omega m_x(\Omega)d\nu(x)
\le 2\,\hbox{$\nu$-$\underset{x\in\Omega}{\rm ess\ inf}\ m_x(\Omega)$}.
\end{equation}
Equivalently,
\begin{equation}\label{Nlacondcoj021802}\Omega \ \hbox{   $m$-calibrable} \ \Rightarrow \  \hbox{ $\nu$-$\underset{x\in\Omega}{\rm ess\,sup}$} \  H^{m}_{\partial \Omega}(x) \leq  \lambda^m_\Omega.
\end{equation}
\end{proposition}

\begin{proof}
By Theorem~\ref{trasen002}, there exists an antisymmetric  function $\g$ in $\Omega\times\Omega$ such that
  \begin{equation}\label{trasen001prueba}
  -1\le \g(x,y)\le 1 \qquad \hbox{for $(\nu \otimes m_x)$-\mbox{a.e. }$(x,y) \in \Omega \times \Omega$},\end{equation}    and
  \begin{equation}\label{ju1522tena01prueba}
   \frac{1}{\nu(\Omega)}\int_\Omega m_z(\Omega)d\nu(z)=m_x(\Omega)+\int_{\Omega}  \g(x,y) \,dm_x(y)   \qquad\hbox{for $\nu$-a.e. }x\in\Omega  . \end{equation}
  Hence, $$\frac{1}{\nu(\Omega)}\int_\Omega m_z(\Omega)d\nu(z)\le 2 m_x(\Omega)\quad \hbox{for $\nu$-a.e. $x\in\Omega$,}$$
  from where~\eqref{Nlacondcoj02} follows.

  The equivalent thesis~\eqref{Nlacondcoj021802}  follows from~\eqref{Nlacondcoj02} and
   the fact that
   \begin{center}
   $\hbox{$\nu$-$\underset{x\in\Omega}{\rm ess\,sup}\ H^{m}_{\partial \Omega}(x)$} \leq  \lambda^m_\Omega$ \ $\Longleftrightarrow$ \
$ \displaystyle \frac{1}{\nu(\Omega)}\int_\Omega m_x(\Omega)d\nu(x)
\le 2 \,\hbox{$\nu$-$\underset{x\in\Omega}{\rm ess\ inf}\ m_x(\Omega)$}.
$\end{center}
\noindent  For this last equivalence recall from \eqref{defcur} that
$$H^m_{\partial \Omega}(x) =  1 - 2 m_x(\Omega) $$
and that
$$\lambda_\Omega^m=\frac{P_m(\Omega)}{\nu(\Omega)}=1-\frac{1}{\nu(\Omega)}\int_\Omega m_x(\Omega) d\nu(x).$$
\end{proof}

The converse of Proposition~\ref{falsa.22} is not true in general, an example is given in \cite{MRT1} (see also~\cite{MRTLibro}) for $[\R^3, d, m^J]$, with $d$ the Euclidean distance and   $J= \frac{1}{|B_1(0)|} \1_{B_1(0)}$.
 Let us see an example, in the case of graphs, where the converse of Proposition~\ref{falsa.22} is not true
 \begin{example} Let $V(G)=\{ x_1,x_2,\ldots , x_8 \} $ be a finite weighted discrete graph with the following weights:
$ w_{x_1,x_2}=  w_{x_2,x_3} =  w_{x_6,x_7} =  w_{x_7,x_8}=2 , \ w_{x_3,x_4}= w_{x_4,x_5} = 1 , \ w_{x_4,x_5}=10$ and $w_{x_i,x_j}=0$ otherwise. If $\Omega:= \{ x_2, x_3, x_4, x_5, x_6, x_7 \}$, we have
$$\lambda_\Omega^{m^G} = \frac{1}{9} \quad \hbox{and} \quad H^{m^G}_{\partial \Omega}(x) \leq 0 \ \quad \forall \, x \in \Omega.$$
Therefore, \eqref{Nlacondcoj021802} holds. However, $\Omega$ is not $m^G$-calibrable since, if $A:= \{ x_4, x_5 \}$, we have
$$\frac{P_{m^G}(A)}{\nu_G(A)} = \frac{1}{11}.$$
 \end{example}

 \begin{proposition}\label{calvscon01} Let $\Omega \subset X$ be a $\nu$-measurable set with $0<\nu(\Omega)<\nu(X) $.
 \item(1)
If $\Omega=\Omega_1\cup\Omega_2$ with $\nu(\Omega_1\cap\Omega_2)=0$, $\nu(\Omega_1)>0$, $\nu(\Omega_2)>0$, and $L_m(\Omega_1,\Omega_2)=0$ (whenever this non-trivial decomposition is satisfied we will write $\Omega=\Omega_1\cup_m\Omega_2$), then
$$\min\{\lambda_{\Omega_1}^m,\lambda_{\Omega_2}^m\}\le \lambda_\Omega^m.$$
\item(2) If $\Omega=\Omega_1\cup_m\Omega_2$ is $m$-calibrable, then each $\Omega_i$ is $m$-calibrable, $i=1,2$, and
$$\lambda_\Omega^m=\lambda_{\Omega_1}^m=\lambda_{\Omega_2}^m.$$

\end{proposition}

\begin{proof}
(1)  is a direct consequence of Proposition~\ref{launion01} and the fact that, for $a,b,c,d$ positive real numbers, $\min\left\{\frac{a}{b},\frac{c}{d}\right\}\le\frac{a+c}{b+d}.$
(2) is a direct consequence of (1) together with the definition of $m$-calibrability.
\end{proof}

\section{The Eigenvalue Problem for the $1$-Laplacian in Metric Random Walk Spaces}\label{Eigenpair}
Let $[X,d,m]$ be a metric random walk  space with invariant and reversible measure~$\nu$ and   assume that $[X,d,m]$ is $m$-connected.

In this section we introduce the eigenvalue problem associated with the $1$-Laplacian $\Delta^m_1$ and its relation with the Cheeger minimization problem.  For the particular case of finite weighted discrete graphs where the weights are either $0$ or $1$, this problem was first studied by  Hein and B\"uhler (\cite{HB}) and a more complete study was subsequently performed by Chang in \cite{Chang1}.

\begin{definition}\label{defeigenpair}{
A pair $(\lambda, u) \in \R \times L^2(X, \nu)$ is called an {\it $m$-eigenpair} of the $1$-Laplacian $\Delta^m_1$ on $X$ if $\Vert u \Vert_{L^1(X,\nu)} = 1$ and    there exists $\xi \in {\rm sign}(u)$  (i.e., $\xi(x) \in {\rm sign}(u(x))$ for every $x\in X$)  such that
\begin{equation}\label{deff1}
\lambda \, \xi \in \partial \mathcal{F}_m(u) = - \Delta^m_1 u.
\end{equation}
The function $u$ is called an {\it $m$-eigenfunction} and $\lambda$ an {\it $m$-eigenvalue} associated to $u$.
}
\end{definition}

   Observe that, if $(\lambda, u)$ is an $m$-eigenpair of $\Delta^m_1$, then $(\lambda, - u)$ is also an $m$-eigenpair of $\Delta^m_1$.

\begin{remark}\label{1257m} By Theorem \ref{chsubd}, the following statements  are equivalent:

\noindent (1) $(\lambda, u)$ is an  $m$-eigenpair of the $1$-Laplacian $\Delta^m_1$ .

 \noindent (2)  There exists
$\hbox{\bf g}\in L^\infty(X\times X, \nu \otimes m_x)$ antisymmetric with $\Vert \hbox{\bf g} \Vert_{L^\infty(X \times X,\nu\otimes m_x)} \leq 1$,
        such that
   \begin{equation}\label{1-lapla.var-ver2} \left\{ \begin{array}{ll}
\displaystyle-\int_{X}\g(x,y)\,dm_x(y)= \lambda \xi(x) \quad \hbox{for }\nu-\mbox{a.e. }x\in X,
    \\ \\ \displaystyle -\int_{X} \int_{X}\hbox{\bf g}(x,y)dm_x(y)\,u(x)d\nu(x)= TV_m(u). \end{array}\right.
    \end{equation}

\noindent (3) There exists   $\hbox{\bf g}\in L^\infty(X\times X, \nu \otimes m_x)$ antisymmetric with $\Vert \hbox{\bf g} \Vert_{L^\infty(X \times X,\nu\otimes m_x)} \leq 1$,
        such that
   \begin{equation}\label{1-lapla.var-ver202} \left\{ \begin{array}{ll}
\displaystyle-\int_{X}{\bf g}(x,y)\,dm_x(y)= \lambda \xi(x) \quad \hbox{for }\nu-\mbox{a.e. }x\in X,
    \\ \\ \displaystyle {\bf g}(x,y)(u(y)-u(x))=|u(y)-u(x)|\quad \hbox{for }\nu \otimes m_x-\hbox{a.e. } (x,y)\in X\times X; \end{array}\right.
    \end{equation}

\noindent (4) There exists
$\hbox{\bf g}\in L^\infty(X\times X, \nu \otimes m_x)$ antisymmetric with $\Vert \hbox{\bf g} \Vert_{L^\infty(X \times X,\nu\otimes m_x)}\leq 1$,
        such that
   \begin{equation}\label{forpadib} \left\{ \begin{array}{ll}
\displaystyle-\int_{X}\g(x,y)\,dm_x(y)= \lambda \xi(x) \quad \hbox{for }\nu-\mbox{a.e. }x\in X,
    \\ \\ \displaystyle \lambda= TV_m(u); \end{array}\right.
    \end{equation}

\end{remark}

\begin{remark}
 Note   that, since $TV_m(u)=\lambda$ for any $m$-eigenpair $(\lambda, u)$ of  $\Delta^m_1$, then
$$\lambda=TV_m(u)= \frac12 \int_X\int_X \vert u(y)-u(x)\vert dm_x(y)d\nu(x)\leq \frac12 \int_X\int_X (\vert u(y)\vert + \vert u(x)\vert)  dm_x(y)d\nu(x)=\Vert u \Vert_1=1,$$
thus $$0\leq \lambda \leq 1.$$
\end{remark}

\begin{example}\label{1Laplagrapheigh}  Let $[V(G), d_G, m^G]$  be the metric random walk space given in Example~\ref{JJ}~(3) with invariant and reversible measure $\nu_G$. Then, a pair $(\lambda, u) \in \R \times L^2(V(G), \nu_G)$ is an $m^G$-eigenpair of  $\Delta^{m^G}_1$ if $\Vert u \Vert_{L^1(V(G),\nu_G)} = 1$ and
 there exists $\xi \in {\rm sign}(u)$ and $\hbox{\bf g}\in L^\infty(V(G)\times V(G), \nu_G \otimes m^{G}_x)$ antisymmetric with $\Vert \hbox{\bf g} \Vert_{L^\infty(V(G)\times V(G),\nu_G\otimes m^G_x)} \leq 1$
        such that
   \begin{equation}\label{1-lapla.var-verexam} \left\{ \begin{array}{ll}
\displaystyle- \sum_{y \in V(G)}\g(x,y)\frac{w_{xy}}{dx} = \lambda \xi(x) \quad \hbox{for }\nu_G-\mbox{a.e. }x\in V(G),
    \\ \\ \displaystyle \hbox{\bf g}(x,y) \in {\rm sign}(u(y) - u(x))\quad \hbox{for }\nu_G \otimes m^{G}_x-\hbox{a.e. } (x,y)\in V(G)\times V(G). \end{array}\right.
    \end{equation}

 In  \cite{Chang1}, Chang  gives the $1$-Laplacian spectrum for some special graphs like the Petersen graph, the complete graph $K_n$, the circle graph with $n$ vertices $C_n$, etc.
We will now provide an example in which the vertices have loops.
     Let $V = V(G)=\{ a, b \}$ and $w_{aa} = w_{bb} = p$, $w_{ab} = w_{ba} = 1- p$, with $0 < p <1$. Then, $(\lambda, u) \in \R \times L^2(V, \nu_G)$ is an $m^G$-eigenpair of  $\Delta^{m^G}_1$ if $\vert u(a) \vert + \vert u(b) \vert = 1$ and
 there exists $\xi \in {\rm sign}(u)$ and $\hbox{\bf g}\in L^\infty(V\times V, \nu_G \otimes m^{G}_x)$ antisymmetric with $\Vert \hbox{\bf g} \Vert_{L^\infty(V\times V,\nu_G\otimes m^G_x)} \leq 1$
        such that
          \begin{equation}\label{1-lapla.var-verexamI} \left\{ \begin{array}{ll}
          \displaystyle \g(a,a)=\g(b,b)=0, \ \g(a,b)=-\g(b,a),
\\ \\ \displaystyle-\g(a,b) (1-p)= \lambda \xi(a), \\ \\ \displaystyle \displaystyle\g(a,b)(1-p)= \lambda \xi(b),
    \\ \\ \displaystyle {\bf g}(a,b)(u(b)-u(a))=|u(b)-u(a)|.
    \end{array}\right.
    \end{equation}
  Now, it is easy to see from system \eqref{1-lapla.var-verexamI},  using a case-by-case argument, that the
  $m$-eigenvalues of $\Delta^{m^G}_1$ are $$\lambda=0\ \hbox{ and } \ \lambda=1-p,$$
and the    following pairs are $m$-eigenpairs of $\Delta^{m^G}_1$ (observe that the measure $\nu_G$ is not normalized):
     $$
     \begin{array}{l}\displaystyle
     \lambda = 0, \quad \hbox{and} \quad (u(a),u(b))= (1/2, 1/2),
     \\[10pt] \displaystyle
       \lambda = 1 - p,  \quad \hbox{and} \quad (u(a),u(b)) =(0,-1)+\mu(1,1)
       \quad\forall 0\le\mu\le 1.
     \end{array}
     $$
  For example, suppose that $(\lambda, u)$ is an $m$-eigenpair with $u(a)=u(b)$. Then, $u(a)=u(b)=\frac12$ ($u(a)=u(b)=-\frac12$ yields the same eigenvalue) and, therefore, $\xi=1$ thus $\lambda=0$. Alternatively, we could have $u(a)>u(b)$ thus $g(a,b)=-1$ and we continue by using \eqref{1-lapla.var-verexamI}.

     Observe that, if a locally finite weighted discrete graph contains a vertex $x$ with no loop, i.e. $w_{x,x}=0$, then $\left(1,\frac{1}{d_x}\delta_x\right)$ is an $m$-eigenpair of the $1$-Laplacian.
      Conversely, if $1$ is an $m$-eigenvalue of $\Delta^{m^G}_1$, then there exists at least one vertex in the graph with no loop (this follows easily from Proposition~\ref{chhar}).
\end{example}

    We have the following relation between $m$-calibrable sets and $m$-eigenpairs of $\Delta^{m}_1$.

  \begin{theorem}\label{eigencalib} Let $\Omega\subset X$ be a $\nu$-measurable set with $0<\nu(\Omega)<\nu(X)$. We have:
  \item(i) If $(\lambda^m_\Omega, \frac{1}{\nu(\Omega)} \1_\Omega)$ is an $m$-eigenpair of $\Delta_1^m$, then $\Omega$ is $m$-calibrable.\\\
  \item(ii) If $\Omega$ is $m$-calibrable and
  \begin{equation}\label{foormul}
  m_x(\Omega) \leq \lambda_\Omega^m \quad   \hbox{for $\nu$-almost every } \ x \in X \setminus \Omega,
  \end{equation}
then $(\lambda^m_\Omega, \frac{1}{\nu(\Omega)} \1_\Omega)$ is an $m$-eigenpair of $\Delta_1^m$.

  \end{theorem}
  \begin{proof} (i): Since  $(\lambda^m_\Omega, \frac{1}{\nu(\Omega)} \1_\Omega)$ is an $m$-eigenpair of $\Delta_1^m$, there exists $\xi \in {\rm sign}(\1_\Omega)$ such that $- \lambda_\Omega^m  \xi \in \Delta_1^m(\1_\Omega)$. Then, by Theorem \ref{falsa}, we have that $\Omega$ is $m$-calibrable.

\noindent (ii): If  $\Omega$ is $m$-calibrable, by Theorem \ref{falsa}, we have
 \begin{equation}\label{N22alaresvi2213} - \lambda^m_\Omega \tau^* \in \Delta_1^m \1_\Omega \, \quad \hbox{in} \ X
\end{equation}
 for  $$\tau^*(x)=\left\{
\begin{array}{ll}
1   &\quad\hbox{if } x\in  \Omega ,\\[5pt]
\displaystyle - \frac{1}{\lambda_\Omega^m} m_x(\Omega)&\quad\hbox{if } x\in X\setminus\Omega.
\end{array}
\right.$$
Now, by \eqref{foormul}, we have that $\tau^* \in {\rm sign} (\1_\Omega)$ and, consequently, $\left(\lambda^m_\Omega, \frac{1}{\nu(\Omega)} \1_\Omega\right)$ is an $m$-eigenpair of $\Delta_1^m$.
  \end{proof}

 In the next example we see that, in Theorem~\ref{eigencalib}, the reverse implications of~(i) and~(ii) are false in general.

  \begin{example}\label{ejeigenpair}(1)
 Let $G=(V,E)$ be the weighted discrete graph where $V= \{ a,b,c \}$ is the vertex set and the weights are given by $w_{ab} = w_{ac} = w_{bc} = \frac12$ and $w_{aa}=w_{bb}=w_{cc}=0$. Then, $m_a = \frac12 \delta_b +\frac12 \delta_c$, $m_b = \frac12 \delta_a +\frac12 \delta_c$, $m_c = \frac12 \delta_a +\frac12 \delta_b$ and $\nu_G = \delta_a + \delta_b + \delta_c$. By Remark \ref{observ1}(2), we have  that $\Omega:= \{a,b \}$ is  $m^G$-calibrable.  However,
 $ \lambda_\Omega^{m^G} = \frac12$  and $(\frac12, \1_\Omega)$ is not an $m$-eigenpair of $\Delta_1^m$ since $0 \not\in{\rm med}_\nu (\1_\Omega)$ (see Corollary \ref{meddiaa} and the definition of ${\rm med}_\nu$ above that Corollary). Therefore, \eqref{foormul} does not hold (it follows by a simple calculation that $m^G_c (\Omega) = 1>\frac12=\lambda_\Omega^{m^G}$).

\item(2)  Consider the locally finite weighted discrete graph $[\mathbb{Z}^2, d_{\mathbb{Z}^2}, m^{\mathbb{Z}^2}]$, where $d_{\mathbb{Z}^2}$ is the Hamming distance and the weights are defined as usual: $w_{xy}=1$ if $d_{\mathbb{Z}^2}(x,y)=1$ and $w_{xy}=0$ otherwise  (see Example \ref{JJ} (3)).   For ease of notation we denote $m:=m^{\Z^2}$. Let
$$\Omega_k:=\{(i,j)\in \mathbb{Z}^2:0\le i,j\le k-1\} \ \hbox{for} \ k\ge 1 .$$
It is easy to see that
$$\lambda_{\Omega_k}^m=\frac{1}{k}.$$
For $1\le k \le 4$ these sets are $m$-calibrable and satisfy \eqref{foormul}. Therefore, for $1 \le k \le 4$, $\left(\frac{1}{k},\frac{1}{\nu(\Omega_k)}\1_{\Omega_k}\right)$ is an $m$-eigenpair of the $1$-Laplacian in $\mathbb{Z}^2$ and with the same reasoning they are still $m$-eigenpairs of the $1$-Laplacian in the metric random walk space $\left[(\Omega_k)_{m},d_{\mathbb{Z}^2}, m^{(\Omega_k)_{m}}\right]$ (recall Corollary \ref{corolarioomegamcal}, for ease of notation let $m_k:=m^{(\Omega_k)_{m}}$). For this last space, recall the definition of $(\Omega_k)_{m}$ from \eqref{omegam} and that of $m_k=m^{(\Omega_k)_{m}}$ from Example \ref{JJ}~(5). Note further that, in the case of graphs, $\partial_{m^{G}}\Omega$ is the set of vertices outside of $\Omega$ which are related to vertices in $\Omega$, i.e., the vertices outside of $\Omega$ which are at a graph distance of $1$ from $\Omega$. For example, $\Omega_2=\{(0,0),(1,0),(1,1),(0,1)\}$ and
$(\Omega_2)_{m}=\Omega_2\cup\partial_{m}\Omega_2$,
where
$$\partial_{m}\Omega_2=\{(2,0),(2,1),(1,2),(0,2),(-1,1),
(-1,0),(0,-1),(1,-1)\}.$$
Moreover, recalling again Example \ref{JJ}~(5), we have that $(m_2)_x(\{y\})=m_x(\{y\})$ for every $x$, $y\in(\Omega_2)_m$, i.e., the probabilities associated to the jumps between different vertices in $(\Omega_2)_m$ do not vary. On the other hand,
$$(m_2)_x(\{x\})=m_x(\partial_{m}\Omega_2)=\frac14+\frac14=\frac12,$$
for every $x\in \partial_m\Omega_2$ (note that, in this case, each vertex in $\partial_m\Omega_2$ is related to $2$ vertices outside of $(\Omega_2)_m$).  Consequently, informally speaking, a loop ``appears'' at each vertex of $\partial_m\Omega_2$ since there is now the possibility of staying at the same vertex after a jump. However, this new metric random walk space $\left[(\Omega_2)_{m},d_{\mathbb{Z}^2}, m_2\right]$ can be reframed so as to regard it as associated to a weighted discrete graph, thus making the previous formal comment rigorous. In other words, we may define a weighted discrete graph which gives the same associated metric random walk space. This is easily done by taking the vertex set $V:=(\Omega_2)_m$ and the following weights: $w_{x,y}=1$ for $x$, $y\in (\Omega_2)_m$ with $d_{\Z^2}(x,y)=1$, $w_{x,x}=2$ for $x\in \partial_m(\Omega_2)_m$ and $w_{x,y}=0$ otherwise (see Figure \ref{fig:partiguales}).

Let us see what happens for $$\Omega_5:=\{(i,j)\in \mathbb{Z}^2:0\le i,j\le 4\} .$$
In this case,
$$\lambda_{\Omega_5}^{m}=\frac15,$$ and an algebraic calculation gives that $\left(\frac15,\frac{1}{\nu({\Omega_5})}\1_{\Omega_5}\right)$ is an $m$-eigenpair in $\mathbb{Z}^2$ (see Figure~\ref{fig001}). Moreover, $\left(\frac15,\frac{1}{\nu({\Omega_5})}\1_{\Omega_5}\right)$ is also an $m^A$-eigenpair of the $1$-Laplacian in the metric random walk space
$$\left[A:=\{(i,j)\in \mathbb{Z}^2:-2\le i,j\le 6\}, d_{\mathbb{Z}^2}, m^A\right]$$ or even in the metric random walk space obtained, in the same way, with the smaller set shown in Figure~\ref{fig001}.
\begin{figure}[h]
\includegraphics[scale=0.55]{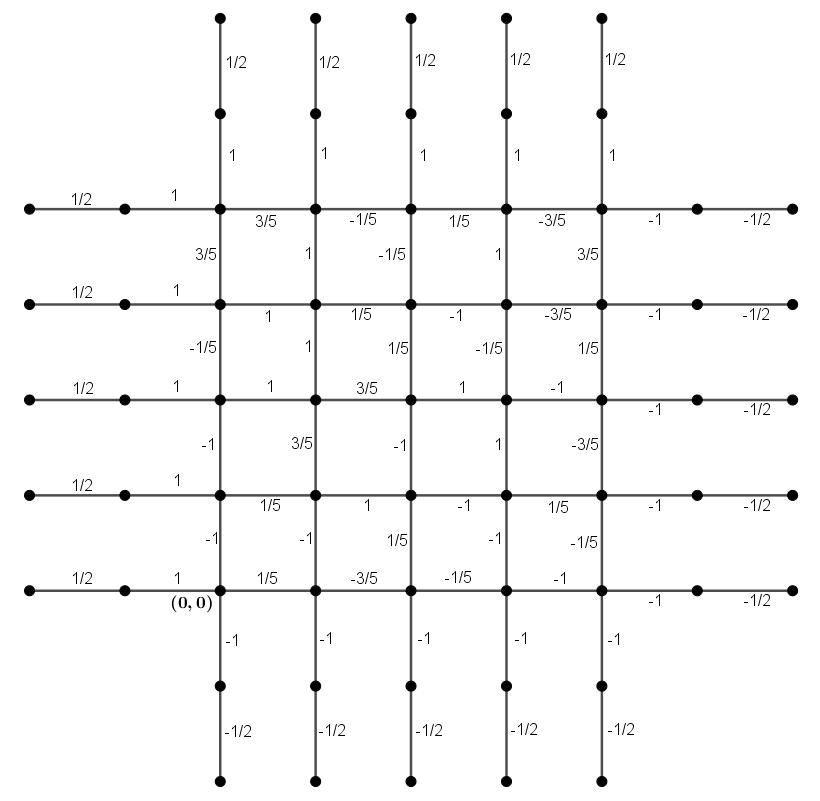}
\caption{The  numbers in the graph are the values of a function $\g(x,y)$ satisfying~(\ref{forpadib}), where $x$ is the vertex to the left of the number represented in the graph and $y$ the one to the right, or, alternatively, $x$ is the one above and $y$ the one below. Elsewhere, $\g(x,y)$ is taken as $0$. The vertex $(0,0)$ is labelled in the graph. As an example, $\g((0,0),(1,0))=1/5$ and $\g((0,1),(0,0))=-1$.}
  \label{fig001}
\end{figure}
 However,
$$(m_5)_{(i,j)}({\Omega_5})= \frac14\quad\forall (i,j)\in(\Omega_5)_{m}\setminus{\Omega_5}$$
so \eqref{foormul} is not satisfied. Furthermore, $\left(\frac15,\frac{1}{\nu({\Omega_5})}\1_{\Omega_5}\right)$ fails to be an $m_5$-eigenpair of the $1$-Laplacian in the metric random walk space $\left[(\Omega_5)_{m},d_{\mathbb{Z}^2}, m_5 \right]$ since the condition on the median  given in Corollary \ref{meddiaa} is not satisfied;
nevertheless, $\Omega_5$ is still  $m_5$-calibrable in this setting.
  \end{example}

  \begin{remark}\label{13021901} Let us give some characterizations of~\eqref{foormul}.
  \item(1) In terms of the  $m$-mean curvature we have that,
  $$
 \eqref{foormul}  \Longleftrightarrow \hbox{$\nu$-}\underset{x\in \Omega^c}{\hbox{esssup}}\,H_{\partial \Omega^c}^m(x)\le\frac{1}{\nu(\Omega)}\int_\Omega H_{\partial \Omega}^m(x)d\nu(x),$$
where $\Omega^c=X\setminus \Omega$. Indeed, \eqref{foormul} is equivalent to $$1-2m_x(\Omega)\ge 1-2\frac{P_m(\Omega)}{\nu(\Omega)}=\frac{\nu(\Omega)-2P_m(\Omega)}{\nu(\Omega)} \quad \hbox{for $\nu$-almost every }  x \in \Omega^c,$$
and this inequality can be rewritten as
 $$-H_{\partial \Omega}^m(x)\le  \frac{1}{\nu(\Omega)}\int_\Omega H_{\partial \Omega}^m(y)d\nu(y) \quad \hbox{for $\nu$-almost every }  x \in \Omega^c$$
  thanks to~\eqref{defcur} and~\eqref{1secondf021}. Hence, since $ H_{\partial \Omega}^m(x)= -H_{\partial \Omega^c}^m(x)$, we are done.

\item(2)  Furthermore, we have that
$$
 \eqref{foormul}  \Longleftrightarrow  \frac{1}{\nu(\Omega)}\int_\Omega m_x(\Omega)d\nu(x)\,\le \hbox{$\nu$-}\underset{x\in \Omega^c}{\hbox{essinf}}\, m_x(\Omega^c).$$
 Indeed, in this case, on account of~\eqref{secondf021}, we rewrite~\eqref{foormul} as
 $$1-m_x(\Omega^c)\le 1-\frac{1}{\nu(\Omega)}\int_\Omega m_y(\Omega)d\nu(y) \quad \hbox{for $\nu$-almost every }  x \in \Omega^c,$$ or, equivalently,
 $$ \frac{1}{\nu(\Omega)}\int_\Omega m_y(\Omega)d\nu(y)\le m_x(\Omega^c) \quad \hbox{for $\nu$-almost every }  x \in \Omega^c,$$
 which gives us the characterization.
  \end{remark}

In the next example we give $m$-eigenpairs of the $1$-Laplacian for the metric random walk spaces given in Example~\ref{JJ}~(1).

    \begin{example} Let $\Omega \subset \R^N$ with $\mathcal{L}^N(\Omega) < \infty$ and consider the metric random walk space $[\Omega, d, m^{J,\Omega}]$ given in  Example~\ref{JJ}~(1) with $J:= \frac{1}{\mathcal{L}^N(B_r(0))} \1_{B_r(0)}$. Moreover, assume that  there exists a ball $B_\rho(x_0) \subset \Omega$ such that ${\rm dist}(B_\rho(x_0), \R^N \setminus \Omega) > r$. Then, by \eqref{perimomega}, we have
  $$P_{ m^{J,\Omega}} (B_\rho(x_0)) = P_{ m^{J}} (B_\rho(x_0)),$$
  and, since $B_\rho(x_0)$ is $m^J$-calibrable, we have that $B_\rho(x_0)$ is $m^{J,\Omega}$-calibrable. Assume also that  $\mathcal{L}^N(B_\rho(x_0)) < \frac12 \mathcal{L}^N(B_r(0))$. Let us see that
   \begin{equation}\label{Nfoormul}
  m^{J,\Omega}_x(B_\rho(x_0)) \leq \lambda_{B_\rho(x_0)}^{m^{J, \Omega}} \quad   \hbox{for $\mathcal{L}^N$-almost every }  x \in \Omega  \setminus B_\rho(x_0).
  \end{equation}
By Remark \ref{13021901}, \eqref{Nfoormul} is equivalent to
\begin{equation}\label{oollee}
\frac{1}{\mathcal{L}^N(B_\rho(x_0))}\int_{B_\rho(x_0)} m^{J,\Omega}_x(B_\rho(x_0))dx \,\le \hbox{$\mathcal{L}^N$-}\underset{x\in \Omega \setminus B_\rho(x_0)}{\hbox{essinf}}\,  m^{J,\Omega}_x(\Omega \setminus B_\rho(x_0)) ).
\end{equation}
Now, for $x\in \Omega$, we have
$$m^{J,\Omega}_x(B_\rho(x_0)) = m^{J}_x(B_\rho(x_0)) = \frac{1}{\mathcal{L}^N(B_r(0))} \int_{B_\rho(x_0)} \1_{B_r(0)}(x-y) dy \leq \frac12.$$
Then, for $x\in \Omega \setminus B_\rho(x_0)$, we have
$$m^{J,\Omega}_x(\Omega \setminus B_\rho(x_0)) ) = 1 - m^{J,\Omega}_x( B_\rho(x_0)) \geq  \frac12 \geq \frac{1}{\mathcal{L}^N(B_\rho(x_0))}\int_{B_\rho(x_0)} m^{J,\Omega}_x(B_\rho(x_0))dx.$$
Hence, \eqref{Nfoormul} holds. Therefore, by Theorem \ref{eigencalib}, we have that $$\left(\lambda^{m^{J,\Omega}}_{B_\rho(x_0)}, \frac{1}{\mathcal{L}^N(B_\rho(x_0))} \1_{B_{\rho(x_0)}}\right)$$ is an $m^{J,\Omega}$-eigenpair of $\Delta^{m^{J,\Omega}}_1$.

Similarly, for the metric random walk space $[\R^n, d, m^{J}]$ with $J= \frac{1}{\mathcal{L}^N(B_r(0))} \1_{B_r(0)}$, and for $\mathcal{L}^N(B_\rho(x_0)) < \frac12 \mathcal{L}^N(B_r(0))$, we have that $$\left(\lambda^{m^{J}}_{B_\rho(x_0)},  \frac{1}{\mathcal{L}^N(B_\rho(x_0))} \1_{B_{\rho(x_0)}}\right)$$ is an $m^{J}$-eigenpair of $\Delta^{m^{J}}_1$.

  \end{example}

\subsection{The $m$-Cheeger Constant of a  Metric Random Walk Space with Finite Measure}\

    In this subsection we give a relation between the non-null $m$-eigenvalues of the $1$-Laplacian and the $m$-Cheeger constant of $X$ when $\nu(X)<+\infty$.

 From now on in this section we assume that $[X, d, m]$  is a metric random walk space with invariant and reversible  probability measure $\nu$. Assuming that $\nu(X)=1$ is not a loss of generality since, for $\nu(X)<+\infty$, we may work with $\frac{1}{\nu(X)}\nu$. Observe that  $\lambda_D^m=\frac{P_m(D)}{\nu(D)}$ remains unchanged if we consider the normalized measure,  and the same is true for the $m$-eigenvalues of the $1$-Laplacian.

 In \cite{MST0} we have defined the {\it $m$-Cheeger constant} of $X$ as
$$
 h_m(X):= \inf \left\{ \frac{P_m (D)}{\min\{ \nu(D), \nu(X \setminus D)\}} \ : \  D \subset X, \ 0 < \nu(D) < 1 \right\} $$
 or, equivalently,
 \begin{equation}\label{1523m}
 h_m(X)= \inf \left\{ \frac{P_m (D)}{\nu(D)} \ : \  D \subset X, \ 0 < \nu(D) \leq \frac{1}{2}\right\}.
\end{equation}
Note that, as a consequence of \eqref{secondf021}, we get $$h_m(X)\le 1.$$
Furthermore, observe that this definition is consistent with the definition on graphs (see \cite{Ch}, also \cite{BJ}):

\begin{example}
Let $[V(G), d_G, m^G]$  be the metric random walk space given in Example~\ref{JJ}~(3) with invariant and reversible measure $\nu_G$. Then, for $E \subset V(G)$, since
    $$P_{m^G}(E) =  \sum_{x \in E} \sum_{y \not\in  E}  w_{x,y} \quad \hbox{and} \quad \nu_G(E):= \sum_{x \in E} d_x,$$
    we have
\begin{equation}\label{lambdagraph}
\frac{P_{m^G}(E)}{\nu_G(E)} = \frac{1}{\sum_{x \in E} d_x}\sum_{x \in E} \sum_{y \not\in  E}  w_{x,y}.
\end{equation}
    Therefore,
    \begin{equation}\label{cheeger1}
    h_{m^G}(V(G)) = \inf \left\{ \frac{1}{\sum_{x \in E} d_x}\sum_{x \in E} \sum_{y \not\in  E}  w_{x,y} \ : \ E \subset V(G), \ 0 < \nu_G(E) \leq \frac12 \nu_G(V)) \right\}.
    \end{equation}
    This minimization problem is closely related with the {\it balance graph cut problem}  that appears in Machine Learning Theory (see \cite{G-TS,G-TSBLBr}).
\end{example}

  Recall that in Section~\ref{aire001} we defined a different $m$-Cheeger constant (see~\eqref{cheeg1}), however, the $m$-Cheeger constant $h_m(X)$ is a global constant of the metric random walk space while the $m$-Cheeger constant $h_1^m(\Omega)$ is defined for non-trivial $\nu$-measurable subsets of the space. Note that, if $\nu(X)=1$, then
$$h_m(X)\le h_1^m(\Omega)$$
for any $\nu$-measurable set $\Omega\subset X$ such that $0<\nu(\Omega)\le 1/2$; and, if $h_m(X)=\frac{P_m (\Omega)}{\nu(\Omega)}$ for a $\nu$-measurable set $\Omega\subset X$ such that $0<\nu(\Omega)\le 1/2$, then $h_m(X)=h_1^m(\Omega)$ and, moreover, $\Omega$ is $m$-calibrable.

 \begin{proposition}\label{eigen1}
  Assume that $\nu$ is a probability measure (and, therefore, ergodic). Let $(\lambda, u)$ be an $m$-eigenpair of  $\Delta^m_1$. Then,
    \item(i) $\lambda = 0 \ \iff \ u \ \hbox{is constant $\nu$-a.e., that is, $u= 1$, or $u=-1$}.$

  \item(ii)   $\lambda \not= 0 \ \iff $  there exists $\xi \in {\rm sign}(u)$ such that
 $\displaystyle \int_X \xi(x) d \nu(x) = 0.$
 \end{proposition}

 Observe that $(0,1)$ and $(0,-1)$ are $m$-eigenpairs of the $1$-Laplacian in metric random walk spaces with an invariant and reversible probability measure.

  \begin{proof} (i) By \eqref{1-lapla.var-ver202}, if $\lambda = 0$, we have that $TV_m(u) =0$ and then, by Lemma \ref{lemita1}, we get that  $u$ is constant $\nu$-a.e.  thus, since $\Vert u\Vert_{L^1(X,\nu)}=1$ (and we are assuming $\nu(X)=1$), either $u=1$, or $u=-1$. Similarly, if $u$ is constant $\nu$-a.e. then $TV_m(u) =0$ and, by  \eqref{1-lapla.var-ver202}, $\lambda =0$.

\noindent (ii) ($\Longleftarrow$) If $\lambda=0$, by (i), we have that $u=1$, or $u=-1$, and this is a contradiction with the existence of $\xi \in {\rm sign}(u)$ such that
 $  \int_X \xi(x) d \nu(x) = 0$.
  ($\Longrightarrow$) There exists $\xi \in {\rm sign}(u)$ and
$\hbox{\bf g}\in L^\infty(X\times X, \nu \otimes m_x)$ antisymmetric with $\Vert \hbox{\bf g} \Vert_{L^\infty(X \times X,\nu\otimes m_x)} \leq 1$ satisfying \eqref{1-lapla.var-ver2}. Hence, since ${\bf g}$ is antisymmetric, by the reversibility of $\nu$, we have
$$\lambda \int_X \xi(x) d\nu(x) = -\int_X \int_{X}\g(x,y)\,dm_x(y) d\nu(x) = 0.$$
Therefore, since $\lambda \not= 0$,
$$\int_X \xi(x) d\nu(x) =0.$$
  \end{proof}

Recall now that, given a function $u : X \rightarrow \R$,   $\mu \in \R$ is a {\it median} of $u$ with respect to the measure $\nu$ if
$$\nu(\{ x \in X \ : \ u(x) < \mu \}) \leq \frac{1}{2} \nu(X) \quad \hbox{and} \quad \nu(\{ x \in X \ : \ u(x) > \mu \}) \leq \frac{1}{2} \nu(X).$$
We denote by ${\rm med}_\nu (u)$ the set of all medians of $u$. It is easy to see that
$$\mu \in {\rm med}_\nu (u) \iff - \nu(\{ u = \mu \}) \leq \nu(\{ x \in X \ : \ u(x) > \mu \}) - \nu(\{ x \in X \ : \ u(x) < \mu \}) \leq \nu(\{ u = \mu \}),$$
from where it follows that
\begin{equation}\label{sigg1}
0 \in {\rm med}_\nu (u) \iff \exists \xi \in {\rm sign}(u) \ \hbox{such that} \ \int_X \xi(x) d \nu(x) = 0.
\end{equation}

By Proposition \ref{eigen1} and relation~\eqref{sigg1}, we have the following result that was   obtained for finite graphs by Hein and B\"uhler in \cite{HB}.

 \begin{corollary}\label{meddiaa}
 If  $(\lambda, u)$ is   an $m$-eigenpair of  $\Delta^m_1$ then
    $$ \lambda \not= 0\ \Longleftrightarrow \ 0 \in {\rm med}_\nu (u).$$
\end{corollary}

 Observe that,  by this corollary, if   $\lambda\neq 0$  is an $m$-eigenvalue of  $\Delta^m_1$, then there exists an $m$-eigenvector $u$ associated to $\lambda$ such that its   $0$-superlevel  set $E_0(u)$ has  positive $\nu$-measure. In fact, for any $m$-eigenvector $u$, either $u$ or $-u$ will satisfy this condition.

\begin{proposition}\label{chhar} If $(\lambda, u)$ is an $m$-eigenpair with $\lambda >0$ and $\nu(E_0(u)) > 0$, then $\left(\lambda, \frac{1}{\nu(E_0(u))} \1_{E_0(u)}\right)$ is an $m$-eigenpair,  $\lambda=\lambda_{E_0(u)}^m$  and $E_0(u)$ is $m$-calibrable. Moreover $\nu(E_0(u))\le \frac12$.
\end{proposition}

\begin{proof} First observe that, by Corollary~\ref{meddiaa},  we have  that $\nu(E_0(u))\le \frac12$.
 Since $(\lambda,u)$ is an $m$-eigenpair, there exists $\xi\in\hbox{sign}(u)$ such that $$-\lambda\xi\in \Delta_1^mu;$$
 hence, there exists $\g(x,y)\in\hbox{sign}(u(y)-u(x))$ antisymmetric with $\Vert \hbox{\bf g} \Vert_{L^\infty(X \times X,\nu\otimes m_x)} \leq 1$ , such that
 $$-\int_{X}{\bf g}(x,y)\,dm_x(y)=  \lambda \xi(x) \quad \hbox{for} \ \nu-\mbox{a.e. }x\in X.$$
 Now,
 $$\xi(x) = \left\{\begin{array}{ll}
1&\hbox{if } x\in{E_0(u)} \hbox{ (since $u(x)>0$)},\\ \\
 \in[-1,1] &\hbox{if } x\in X\setminus{E_0(u)},
\end{array}
\right.$$
 and, therefore, $\xi\in\hbox{sign}(\1_{E_0(u)})$.  On the other hand,
$$\g(x,y) = \left\{\begin{array}{ll}
\in[-1,1]&\hbox{if } x,y\in{E_0(u)},\\ \\
-1&\hbox{if } x\in{E_0},\ y\in X\setminus {E_0(u)} \hbox{ (since $u(x)>0,\ u(y)\le 0$)},
\\ \\
 1&\hbox{if } x\in X\setminus{E_0(u)},\ y\in   {E_0(u)} \hbox{ (since $u(x)\le 0,\ u(y)> 0$)},\\ \\
 \in[-1,1] &\hbox{if } x,y\in X\setminus{E_0(u)},
\end{array}
\right.$$
and, consequently,  $\g(x,y)\in\hbox{sign}(\1_{E_0(u)}(y)-\1_{E_0(u)}(x))$. Therefore, we have that
 $\left(\lambda,\frac{1}{\nu({E_0(u)})}\1_{E_0(u)}\right)$ is an $m$-eigenpair of~$\Delta^m_1$. Moreover, by Theorem \ref{eigencalib}, we have that ${E_0(u)}$  is  $m$-calibrable.
\end{proof}

\begin{remark}\label{calvscon02}
As a consequence of Proposition~\ref{calvscon01}, when we search for $m$-eigenpairs of the $1$-Laplacian we can restrict ourselves to $m$-eigenpairs of the form $\left(\lambda, \frac{1}{\nu(E)} \1_{E}\right)$ where $E$ is m-calibrable and not decomposable as $E=E_1\cup_m E_2$. Indeed, suppose that $\left(\lambda, \frac{1}{\nu(E)} \1_{E}\right)$ is an $m$-eigenpair and $E=E_1\cup_m E_2$ for some $E_1$, $E_2\subset E$. Then, by \eqref{1-lapla.var-ver202}, there exist $\xi \in {\rm sign}(\1_E)$ and $\hbox{\bf g}\in L^\infty(X\times X, \nu \otimes m_x)$ antisymmetric with $\Vert \hbox{\bf g} \Vert_{L^\infty(X \times X,\nu\otimes m_x)} \leq 1$,
        such that
   \begin{equation} \left\{ \begin{array}{ll}
\displaystyle-\int_{X}{\bf g}(x,y)\,dm_x(y)= \lambda \xi(x) \quad \nu-\mbox{a.e. }x\in X,
    \\ \\ \displaystyle {\bf g}(x,y)\in\hbox{sign}(\1_{E}(y)-\1_{E}(x))\quad \nu \otimes m_x-\hbox{a.e. } (x,y)\in X\times X. \end{array}\right.
    \end{equation}
Then, we may take the same $\xi$ and $g(x,y)$ to see that $\left(\lambda, \frac{1}{\nu(E_1)} \1_{E_1}\right)$ is also  an $m$-eigenpair. Indeed, since $\lambda_E^m=\lambda_{E_1}^m$,  we only need to verify that ${\bf g}(x,y)\in\hbox{sign}(\1_{E_1}(y)-\1_{E_1}(x))$ $\nu \otimes m_x$-a.e.. For $x\in E_1$ we have:
\begin{itemize}
  \item if $y\in E_1$, then $\1_E(y)-\1_E(x)=0=\1_{E_1}(y)-\1_{E_1}(x)$,
  \item if $y\in X\setminus E$, then $\1_E(y)-\1_E(x)=-1=\1_{E_1}(y)-\1_{E_1}(x)$,
\end{itemize}
and, since $L_m(E_1,E_2)=0$, we have that $\nu \otimes m_x(E_1\times E_2)=0$ so the condition is satisfied. Similarly for $x\in E_2$ (again $\nu \otimes m_x(E_2\times E_1)=0$). If $x\in X\setminus E$ then,
\begin{itemize}
  \item if $y\in E_1$, $\1_E(y)-\1_E(x)=1=\1_{E_1}(y)-\1_{E_1}(x)$,
  \item if $y\in E_2$, $\1_E(y)-\1_E(x)=1\in\hbox{sign}(0)=\hbox{sign}(\1_{E_1}(y)-\1_{E_1}(x))$
  \item if $y\in X\setminus E$, $\1_E(y)-\1_E(x)=0=\1_{E_1}(y)-\1_{E_1}(x)$.
\end{itemize}
\end{remark}

Let
 $$\Pi(X):= \left\{ u \in L^1(X, \nu) \ : \ \Vert u \Vert_{L^1(X,\nu)} = 1  \ \hbox{and} \ 0 \in {\rm med}_\nu (u) \right\}$$
and
\begin{equation}\label{minnb}\lambda_1^m(X) := \inf \left\{ TV_m(u) \ : \  u \in \Pi(X) \right\}.\end{equation}
In \cite{MST0} we proved the following result.

 \begin{theorem}[\cite{MST0}]\label{charact} Let $[X, d, m]$ be a metric random walk space with invariant and reversible  probability measure~$\nu$.  Then,

 \item(i)
 $
 h_m(X) = \lambda_1^m(X).
$
 \item(ii) For $\Omega \subset X$ $\nu$-measurable with $\nu(\Omega) = \frac12$,
 $h_m(X) = \lambda_\Omega^m \iff   \1_\Omega - \1_{X \setminus \Omega} \ \hbox{ is a minimizer of } \ \eqref{minnb}.$
\end{theorem}

By  Corollary~\ref{meddiaa}, if $(\lambda, u)$ is  an $m$-eigenpair of $\Delta^m_1$ and $\lambda\not= 0$  then
$u \in \Pi(X)$. Now, $TV_m(u)=\lambda$,
thus,   as a corollary of  Theorem~\ref{charact}~(i), we have the following result.
  Recall that, for finite graphs, it is well known that the first non--zero eigenvalue coincides with the Cheeger constant (see \cite{Chang1}).
\begin{theorem}\label{charactCorol}
If $\lambda\not= 0$ is an $m$-eigenvalue of  $\Delta^m_1$ then
$$ h_m(X) \leq \lambda.$$
\end{theorem}

  This result also follows by Proposition~\ref{chhar} since $\nu(E_0(u))\le \frac12$.

 In the next result we will see  that if the infimum in~\eqref{1523m}  is attained then $h_m(X)$ is an $m$-eigenvalue of~$\Delta^m_1$.

 \begin{theorem}\label{d10s001}  Let $\Omega$ be a $\nu$-measurable subset of $X$ such that $0<\nu(\Omega)\le\frac 12$.

\item(i)   If $\Omega$ and $X\setminus\Omega$
are $m$-calibrable then $\left(\lambda_\Omega^m,\frac{1}{\nu(\Omega)}\1_\Omega\right)$ is an $m$-eigenpair of  $\Delta^m_1$.

\item(ii)
If $h_m(X)=\lambda^m_\Omega$ then $\Omega$ and $X\setminus\Omega$
are $m$-calibrable

\item(iii)
If $h_m(X)=\lambda^m_\Omega$     then $\left(\lambda_\Omega^m,\frac{1}{\nu(\Omega)}\1_\Omega\right)$ is an $m$-eigenpair of  $\Delta^m_1$.
 \end{theorem}

\begin{proof}
First of all, observe that, since $\nu(\Omega)\le\frac 12$,
$$\lambda_{X\setminus\Omega}^m \le \lambda_\Omega^m. $$

 \item{\it (i)}: By Theorem~\ref{trasen002}, since $\Omega$  is  $m$-calibrable, there exists an antisymmetric  function $\g_1$ in $\Omega\times\Omega$ such that
    \begin{equation}\label{trasen001proof001}
  -1\le \g_1(x,y)\le 1 \qquad \hbox{for $(\nu \otimes m_x)$-\mbox{a.e. }$(x,y) \in \Omega \times \Omega$},\end{equation}
    and
\begin{equation}\label{trasen001postpro0f001}\lambda_\Omega^m = -\int_{\Omega }\g_1(x,y)\,dm_x(y) + 1 - m_x(\Omega)\quad\hbox{$\nu$-a.e. $x\in\Omega$};
  \end{equation}
and, since $X\setminus\Omega $  is  $m$-calibrable, there exists an antisymmetric  function $\g_2$ in $(X\setminus\Omega)\times(X\setminus\Omega)$ such that
    \begin{equation}\label{trasen001proof001com}
  -1\le \g_2(x,y)\le 1 \qquad \hbox{for $(\nu \otimes m_x)$-\mbox{a.e. }$(x,y) \in (X\setminus\Omega)\times(X\setminus\Omega)$},\end{equation}    and
\begin{equation}\label{trasen001postpro0f001com}\lambda_{X\setminus\Omega}^m = -\int_{X\setminus\Omega }\g_2(x,y)\,dm_x(y) + 1 - m_x(X\setminus \Omega)\quad\hbox{$\nu$-a.e. $x\in X\setminus\Omega$}.
  \end{equation}
  Consequently, by taking
  $$\g(x,y)=\left\{
  \begin{array}{ll}
  \g_1(x,y)&\hbox{ if }x,y\in \Omega,\\[6pt]
  -1&\hbox{ if }x\in \Omega,y\in X\setminus\Omega,\\[6pt]
   1&\hbox{ if }x\in X\setminus\Omega,y\in\Omega,\\[6pt]
    -\g_2(x,y)&\hbox{ if }x,y\in X\setminus\Omega,
  \end{array}\right.
    $$
    we have that $\g(x,y)\in\hbox{sign}\left(\1_\Omega(y)-\1_\Omega(x)\right)$. Moreover,  from~\eqref{trasen001postpro0f001},
 $$\lambda_{\Omega}^m =  -\int_{X}\g(x,y)\,dm_x(y)\quad\hbox{for $\nu$-a.e. $x\in \Omega$},$$
 and,  since $\lambda_{X\setminus\Omega}^m \le \lambda_\Omega^m $, from~\eqref{trasen001postpro0f001com},
 $$-\lambda_\Omega^m\le  -\lambda_{X\setminus\Omega}^m= -\int_{X}\g(x,y)\,dm_x(y) \le \lambda_\Omega^m  \quad\hbox{for $\nu$-a.e. $x\in X\setminus\Omega$}.$$
 Hence, by Remark~\ref{1257m}~(2), we conclude that $\left(\lambda_\Omega^m,\frac{1}{\nu(\Omega)}\1_\Omega\right)$ is an $m$-eigenpair of  $\Delta^m_1$.

\item{\it (ii)}:  Since $h_m(X)=\frac{P_m(\Omega)}{\nu(\Omega)}$ and $0<\nu(\Omega)\le\frac12$, we have $h_m(X)=h_1^m(\Omega)=\frac{P_m(\Omega)}{\nu(\Omega)}$ and, consequently, $\Omega$ is $m$-calibrable. Let us suppose that $X\setminus\Omega$ is not $m$-calibrable. Then, there exists $E\subset X\setminus\Omega$ such that $\nu(E)<\nu(X\setminus\Omega)$ and $$\lambda_E^m<\lambda_{X\setminus\Omega}^m \ .$$
Now, this implies that  $\nu(E)>\frac12$ since, otherwise, we get
 $$\lambda_E^m<\lambda_{X\setminus\Omega}^m\le\lambda_\Omega^m=h_m(X)$$
which is a contradiction. Moreover, since $\nu(E)<\nu(X\setminus\Omega)$, $\lambda_E^m<\lambda_{X\setminus\Omega}^m$ also implies that
   $$P_m(E)<P_m(X\setminus\Omega)=P_m(\Omega).$$
   However, since $\nu(E)>\frac12$, we have that $\nu(X\setminus E)<\frac12$ and, consequently, taking into account that $\nu(\Omega)\le\nu(X\setminus E)$, we get
 $$\lambda_{X\setminus E}^m=\frac{P_m(E)}{\nu(X\setminus E)}<\frac{P_m(\Omega)}{\nu(\Omega)}=h_m(X),$$
which is also a contradiction.

  Finally, {\it (iii)}  is a direct consequence of {\it (i)} and {\it (ii)}.
\end{proof}

As a consequence of Proposition \ref {chhar} and Theorem \ref{d10s001}, we have the following result.

\begin{corollary}\label{important1} If $h_m(X)$ is a positive $m$-eigenvalue of~$\Delta^m_1$, then, for  any eigenvector $u$ associated to $h_m(X)$ with $\nu(E_0(u))>0$,
 \begin{center}\label{ok1}
 $\left(h_m(X),\frac{1}{\nu(E_0(u))}\1_{E_0(u)}\right)$ is an $m$-eigenpair of~$\Delta^m_1$,
 \end{center}
 $\nu(E_0(u))\le \frac12$, and
 \begin{equation}\label{biieen}h_m(X)=\lambda_{E_0(u)}^m.\end{equation}
   Moreover, both ${E_0(u)}$ and $X\setminus{E_0(u)}$ are $m$-calibrable.
\end{corollary}

 \begin{remark}\label{1330m}
  For $\Omega \subset X$ with $\nu(\Omega) = \frac12$ (thus $\lambda_\Omega^m=2P_m(\Omega)$) we have   that:
  \item(1)
   $\Omega$ and $X\setminus\Omega$
are $m$-calibrable  if, and only if, $\left(2P_m(\Omega),t\1_\Omega-(2-t)\1_{X\setminus\Omega}\right)$ is an $m$-eigenpair of  $\Delta^m_1$  for any $t\in[0,2]$.
\item(2)
If $h_m(X)=   2P_m(\Omega)$ then $\left(2P_m(\Omega),t\1_\Omega-(2-t)\1_{X\setminus\Omega}\right)$ is an $m$-eigenpair of  $\Delta^m_1$  for all $t\in[0,2]$.

\end{remark}

\begin{example} In  Figure~\ref{fig:partiguales}, following the notation in Example~\ref{ejeigenpair}(2), we consider the metric random walk space $\left[X:=(\Omega_2)_{m}, d_{\mathbb{Z}^2}, m_2:=m^{(\Omega_2)_{m}}\right]$. In Figure~\ref{fig:partiguales}(A), we show this space partitioned into two $m_2$-calibrable sets,  $E =\{ (-1,0), (0,0), (1,0), (-1,1), (0,1), (1,1)\}$  and $X\setminus E$, of equal measure, hence, by the previous remark, both $(\lambda_E^{m_2},\frac{1}{\nu(E)}\1_E)$ and $(\lambda_{ E}^{m_2},\frac{1}{\nu(E)}\1_{X\setminus E})$ are $m_2$-eigenpairs. However, the Cheeger constant $h_{m_2}(X)$ is smaller than the eigenvalue $\lambda_{E}^{m_2}$    since, for $D = \{ (1,-1), (1,0), (2,0), (2,1),(1,1),  (1,2)\}$, we have $\lambda_D^{m_2} = \frac16$   (see Figure~\ref{fig:partiguales}(B)).

\begin{figure}[h]\centering
  \begin{subfigure}[t]{0.4\textwidth}
  \includegraphics[width=\textwidth]{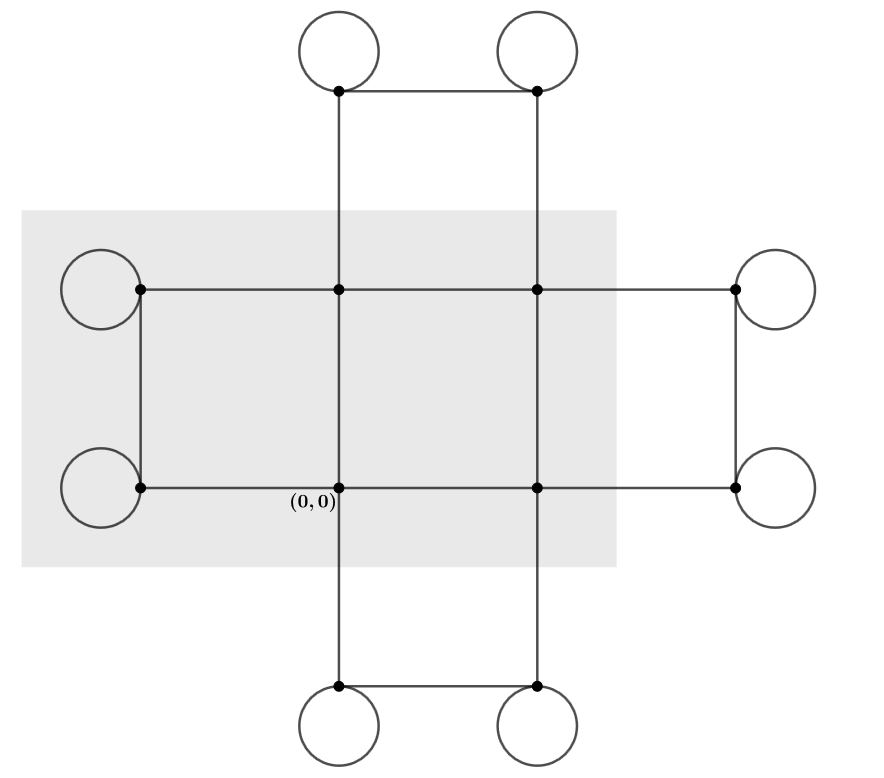}
  \caption{Let $E$ be the set formed by the vertices in the shaded region. Then $\lambda_E^{m_2}=\frac{1}{4}$. }
  \label{fig:test1}
  \end{subfigure}\hspace{1cm}
 \begin{subfigure}[t]{0.36\textwidth}
  \includegraphics[width=\textwidth]{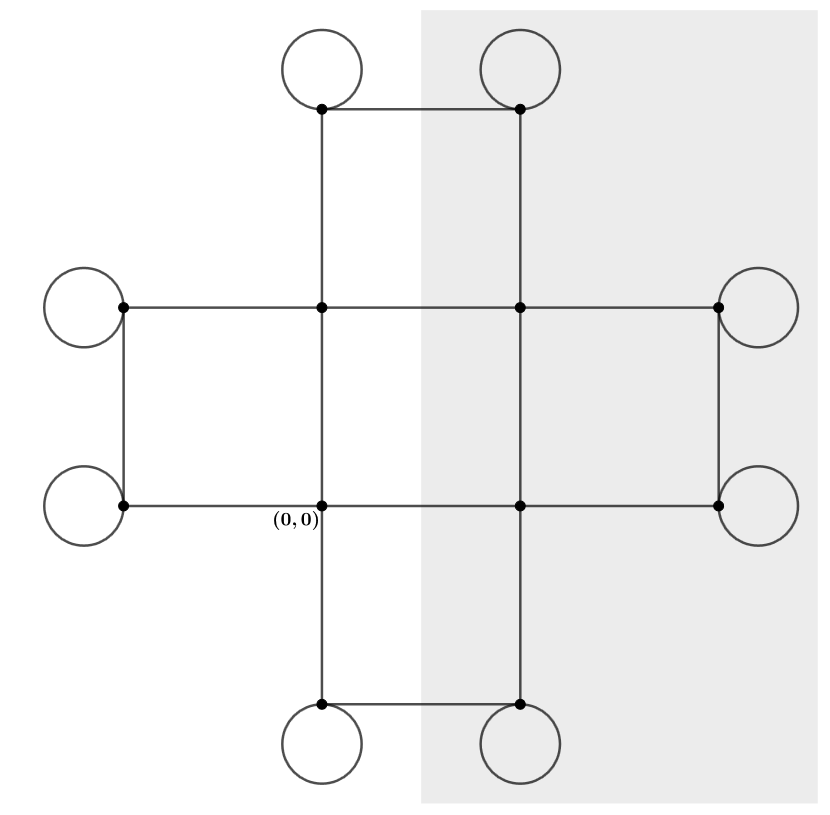}
  \caption{Let $D$ be the set formed by the vertices in the shaded region. Then $\lambda_D^{m_2}=\frac{1}{6}$.}
  \label{fig:test2}
  \end{subfigure}
  \caption{The line segments represented in the figures correspond to the edges between adjacent vertices, with $w_{xy}=1$ for any pair of these neighbouring vertices. The loops that ``appear'' when considering $m_2$ (see Example~\ref{ejeigenpair}(2)) are represented by circles.
  }\label{fig:partiguales}
\end{figure}

\end{example}

\begin{remark}\label{corol001} By Theorems~\ref{charactCorol} and~\ref{d10s001}, and Corollary \ref{important1}, for  finite weighted connected discrete graphs, we have that
 \begin{equation}\label{changeig}h_m(X) \ \hbox{ is  the first non-zero eigenvalue of} \   \Delta^{m^G}_1 \end{equation}
(as already proved in~\cite{Chang1}, \cite{Changetal01} and~\cite{HB}) and, to solve the optimal Cheeger cut problem, it is enough to find an  eigenvector associated to $h_m(X)$ since  then $\{E_0(u),X\setminus E_0(u)\}$ or $\{E_0(-u),X\setminus E_0(-u)\}$ is a Cheeger cut.

  \end{remark}

  In the next examples we will see that \eqref{changeig} is not true in general. We obtain infinite weighted connected discrete graphs (with finite invariant and reversible measure) for which there is no first positive $m$-eingenvalue.

\begin{example}\label{Nexx1}(1)
 Let $[V(G),d_G,m^G]$ be the metric random walk space defined in Example \ref{JJ}~(3) with vertex set  $V(G)=\{x_0,x_1,\ldots,x_n,\ldots \}$ and weights defined as follows:
$$w_{x_{2n}x_{2n+1}}=\frac{1}{2^n} , \quad w_{x_{2n+1}x_{2n+2}}=\frac{1}{3^n} \quad \hbox{for} \ n=0, 1, 2, \dots \hbox{ and } w_{x,y}=0 \hbox{ otherwise.}$$
 We have $d_{x_0}=1,\ d_{x_1}=2$ and, for $n\geq 1$,
$$
\begin{array}{c}\displaystyle d_{x_{2n}}=w_{x_{2n-1}x_{2n}}+w_{x_{2n}x_{2n+1}}=
\frac{1}{3^{n-1}}+\frac{1}{2^n},\\ \\ \displaystyle d_{x_{2n+1}}=w_{x_{2n}x_{2n+1}}+w_{x_{2n+1}x_{2n+2}}=\frac{1}{2^n}+\frac{1}{3^n}.
\end{array}
$$
Furthermore,
$$\nu_G(V) = \sum_{i=0}^\infty  d_{x_i} = 3+\sum_{n=1}^\infty \frac{1}{3^{n-1}}+\frac{1}{2^n} + \frac{1}{2^n}+\frac{1}{3^n} = 7. $$
Observe that the measure $\nu_G$ is not normalized, but this does not affect the result  because the constants $\lambda^m_\Omega$ and the $m$-eigenvalues of the $1$-Laplacian are independent of this normalization.

Consider $E_n:=\{ x_{2n}, x_{2n+1}\}$ for $n\ge 1$.  By (2) in Remark \ref{observ1},
we have that $E_n$ is $m^G$-calibrable. On the other hand,
$$m_{x_{2n-1}}(E_n)=\frac{1}{1+(\frac32)^{n-1}} \ , \ \ m_{x_{2n+2}}(E_n)=\frac{1}{1+\frac34(\frac32)^{n-1}}=\lambda_{E_n}^{m^G} \ , \ \hbox{and} \  m_x(E_n)=0 \ \hbox{else in  $V\setminus E_n$} . $$
Hence,
$$ m_x(E_n) \leq \lambda_{E_n}^{m^G} \quad \hbox{for all} \ x \in V \setminus E_n.$$
Then, by Theorem \ref{eigencalib}, we have that
 $(\lambda^{m^G}_{E_n}, \frac{1}{\nu(E_n)} \1_{E_n})$ is a $m^G$-eigenpair of $\Delta_1^{m^G}$.
 Now,
 $$\lim_{n \to \infty} \lambda_{E_n}^{m^G} = \lim_{n \to \infty}\frac{2^{n+1}}{2^{n+1}+3^{n}} =0.$$
Consequently, both by Theorem \ref{charactCorol} and by definition of $h_{m^G}(V(G))$, we get
 $$ h_{m^G}(V(G)) = 0.$$

 \item(2)   Let $0<s<r<\frac12$. Let $[V(G),d_G,m^G]$ be the metric random walk space defined in Example \ref{JJ}~(3) with vertex set $V(G)=\{x_0,x_1,\ldots,x_n,\ldots \}$ and weights defined as follows:
$$w_{x_0,x_1}= \frac{r}{1-r}+\frac{s}{1-s},\quad w_{x_{n}x_{n+1}}=r^n+s^n \quad \hbox{for} \ n= 1, 2, 3, \dots \hbox{ and } w_{x,y}=0 \hbox{ otherwise}.$$
Then,
\begin{center}$ h_{m^G}(V(G)) = \displaystyle \frac{1-r}{1+r}$ \
is not an $m^G$-eigenvalue of $\Delta_1^{m^G}$.\end{center}
Indeed, to start with, observe that
$\nu_G(V(G))=\frac{4r}{1-r}+\frac{4s}{1-s}$,
$$\nu_G(\{x_0\})\le\frac{\nu_G(V(G))}{2},\  \nu_G(\{x_0,x_1\})>\frac{\nu_G(V(G))}{2},$$
$$\nu_G(\{x_1\})\le\frac{\nu_G(V(G))}{2},\  \nu_G(\{x_1,x_2\})>\frac{\nu_G(V(G))}{2},$$
and, for
$E_n:=\{x_n,x_{n+1},x_{n+2},\dots\},$ $n\ge 2$,
$$\nu_G(E_n)\le \frac{\nu_G(V(G))}{2}.$$ Now, for $n\ge 2$,
$$\lambda_{E_n}^m=\frac{r^{n-1}+s^{n-1}}{r^{n-1}+s^{n-1}+2\left(\frac{r^n}{1-r}+\frac{s^n}{1-s}\right)}
=\frac{r^{n-1}+s^{n-1}}{\frac{1+r}{1-r}r^{n-1}+\frac{1+s}{1-s}s^{n-1}}
$$
 decreases as $n$ increases (therefore, the sets $E_n$ are not $m$-calibrable), and
 $$\lim_n\lambda_{E_n}^m=\frac{1-r}{1+r}.$$
  Let us see that, for any
  $E\subset V(G)$ with   $0<\nu_G(E)\le  \frac{\nu(V(G))}{2}$, we have $\lambda_E^m>\frac{1-r}{1+r}$. Indeed, to start with, observe that if $E=\{x_0\}$ or $E=\{x_1\}$ then $\lambda_{\{x_0\}}^m= \lambda_{\{x_1\}}^m=1> \frac{1-r}{1+r}$. Moreover, we have that $\{x_0,x_1\}   \not\subset E$ and $\{x_1,x_2\}   \not\subset E$ since  $\nu_G(\{x_0,x_1\})\not\le\frac{\nu_G(V(G))}{2}$ and  $\nu_G(\{x_1,x_2\})\not\le\frac{\nu_G(V(G))}{2}$. Therefore, it remains to see what happens for sets $E$ satisfying
  \item{ (i)}
  $x_0\in E$, $x_1\notin E$ and $x_n\in E$ for some $n\ge 2$,
  \item{ (ii)} $x_1\in E$, $x_0\notin E$ and $x_n\in E$ for some $n\ge 3$,
  \item{ (iii)} $x_0\notin  E$, $x_1\notin E$ and $x_n\in E$ for some $ n\ge 2$.

\noindent For the case (i), let $n_1\in \mathbb{N}$ be the first index $n\ge 2$ such that $x_n\in E$; for the case (ii), let $n_2\in \mathbb{N}$ be the first index $n\ge 3$ such that $x_n\in E$; and for the case (iii), let $n_3\in \mathbb{N}$ be the first index $n\ge 2$ such that $x_n\in E$. Now, for the case (i) we have that
$$\lambda_E^m\ge\lambda_{\{x_0\}\cup E_{n_1}}\ge\lambda_{E_{n_1}}.$$
Indeed, the first equality follows from the fact that $P_m(E)\ge P_m(\{x_0\}\cup E_{n_1})$ and $\nu(E)\le \nu(\{x_0\}\cup E_{n_1})$ and the second one follows since
$$\lambda_{\{x_0\}\cup E_{n_1}}=\frac{\frac{r}{1-r}+\frac{s}{1-s} +P_m(E_{n_1})}{\frac{r}{1-r}+\frac{s}{1-s}
+\nu(E_{n_1})}>\frac{ P_m(E_{n_1})}{ \nu(E_{n_1})}=\lambda_{E_{n_1}}.$$
Hence, $\lambda_E^m>\frac{1-r}{1+r}$.
With a  similar argument we get, in the case (ii),
$$\lambda_E^m\ge\lambda_{\{x_1\}\cup E_{n_2}}\ge\lambda_{E_{n_2}}>\frac{1-r}{1+r};$$
and, in the case (iii),
$$\lambda_E^m\ge \lambda_{E_{n_3}}>\frac{1-r}{1+r}.$$

    Consequently, $h_{m^G}(V(G)) = \frac{1-r}{1+r}$ and, by Corollary \ref{important1}, it is not an $m$-eigenvalue of $\Delta_1^{m^G}$.
  \end{example}

\bigskip
\noindent {\bf Acknowledgment.} The authors wish to thank the anonymous referee whose comments after a detailed reading of the paper allowed them to improve its presentation. The authors have been partially supported  by the Spanish MCIU and FEDER, project PGC2018-094775-B-100. The second author was also supported by the Spanish Ministerio de Ciencia, Innovaci\'on y Universidades under Grant BES-2016-079019.

\end{document}